\documentclass[reqno]{amsart}
\usepackage{amsfonts}
\usepackage{amsthm}
\usepackage{amstext}
\usepackage{amsmath}
\usepackage{amscd}
\usepackage{tikz-cd}
\usepackage{amssymb}
\usepackage[mathscr]{eucal}
\usepackage{url}
\usepackage{graphicx}
\usepackage{verbatim}
\usepackage{enumitem}
\usepackage{hyperref}

\newtheorem{theorem}{Theorem}
\newtheorem{corollary}[theorem]{Corollary}
\newtheorem{proposition}[theorem]{Proposition}
\newtheorem{lemma}[theorem]{Lemma}

\newtheorem*{theoremA}{Theorem~A}
\newtheorem*{theoremB}{Theorem~B}

\theoremstyle{definition}
\newtheorem{definition}[theorem]{Definition}
\newtheorem{example}[theorem]{Example}
\newtheorem{remark}[theorem]{Remark}
\newtheorem{problem}[theorem]{Problem}

\newtheorem*{remark*}{Remark}

\numberwithin{equation}{section}

\newcommand\mL{L\kern-0.08cm\char39}
\newcommand{\md}{d\kern-0.035cm\char39\kern-0.03cm}
\newcommand{\mt}{t\kern-0.035cm\char39\kern-0.03cm}
\newcommand{\ml}{l\kern-0.035cm\char39\kern-0.03cm}

\renewcommand{\epsilon}{\varepsilon}

\newcommand{\Homeo}{{\mathcal H}}
\newcommand{\RRR}{{\mathbb R}}
\newcommand{\ZZZ}{{\mathbb Z}}
\newcommand{\NNN}{{\mathbb N}}
\newcommand{\AAa}{{\mathcal A}}

\newcommand{\im}{\operatorname{\rm Im}}
\newcommand{\Int}{\mathop{\rm Int}}
\newcommand{\Cob}{\operatorname{\rm Cob}}
\newcommand{\cob}{\mathop{\rm cob}}
\newcommand{\dist}{\mathop{\rm dist}}
\newcommand{\diam}{\mathop{\rm diam}}
\newcommand{\Bd}{\mathop{\rm Bd}}
\newcommand{\Id}{\operatorname{Id}}
\newcommand{\Hit}{\operatorname{Hit}}
\newcommand{\pr}{\mathop{\rm Pr}}

\newcommand{\Coin}{\mathop{\rm Coin}}
\newcommand{\card}{\mathop{\rm card}}
\newcommand{\Ydeg}{Y^{\rm deg}}
\newcommand{\sgn}{\operatorname{sgn}}

\newcommand{\Orb}{\mathop{\rm Orb}}

\newcommand{\mymod}{\operatorname{mod}}

\newcommand{\tor}{\mathop{\rm tor}}

\newcommand{\abs}[1]{\lvert#1\rvert}

\newcommand{\caseI}{(11-22)}
\newcommand{\caseII}{(12-21)}
\newcommand{\caseIy}{(11-22*)}
\newcommand{\caseIIy}{(12-21*)}

\makeatletter
\@namedef{subjclassname@2020}{%
  \textup{2020} Mathematics Subject Classification}
\makeatother

\begin{document}
\begin{large}

\title[Minimal direct products]{Minimal direct products}

\author{Mat\'u\v s Dirb\'ak}
\author{\mL ubom\'\i r Snoha}
\author{Vladim\'\i r \v Spitalsk\'y}
\address{Department of Mathematics, Faculty of Natural Sciences, Matej Bel University,
Tajovsk\'{e}ho 40, 974 01 Bansk\'{a} Bystrica, Slovakia}

\email{[Matus.Dirbak, Lubomir.Snoha, Vladimir.Spitalsky]@umb.sk}

\thanks{This work was supported by the Slovak Research and Development Agency under contract No.~APVV-15-0439 and partially by VEGA grant 1/0158/20.}

\subjclass[2020]{Primary 37B05; Secondary 37B45}

\keywords{Minimal map, minimal space, product-minimal space, homeo-product-minimal space, direct product.}

\begin{abstract}
A space is called minimal if it admits a minimal continuous selfmap. 
We give examples of metrizable continua $X$ admitting both minimal homeomorphisms and minimal noninvertible maps, whose squares $X\times X$ are not minimal, i.e., they admit neither minimal homeomorphisms nor minimal noninvertible maps, thus providing a definitive answer to a question posed by Bruin, Kolyada and the second author in 2003. (In 2018, Boro\'nski, Clark and Oprocha provided an answer in the case when only homeomorphisms were considered.)

Then we introduce and study the notion of product-minimality. We call a compact metrizable space~$Y$ product-minimal if, for every minimal system $(X,T)$ given by a metrizable space $X$ and a continuous selfmap~$T$, there is a continuous map $S\colon Y\to Y$ such that the product $(X\times Y,T\times S)$ is minimal. 
If such a map $S$ always exists in the class of homeomorphisms, we say that $Y$ is a 
homeo-product-minimal space. 
We show that many classical examples of minimal spaces, including compact connected metrizable abelian groups, compact connected manifolds without boundary admitting a free action of a nontrivial compact connected Lie group, and many others, are in fact homeo-product-minimal.
\end{abstract}

\maketitle

\setcounter{tocdepth}{1}
\tableofcontents


\section{Introduction and main results}\label{S:intro}

Throughout the paper we consider only metrizable spaces; however, sometimes we emphasize metrizability,
mainly in definitions and statements of results.
By a dynamical system we mean a pair $(X,T)$, where $X$ is a (metrizable) space 
and $T\colon X\to X$ is a continuous map. Note that if not stated otherwise,
$X$ is not necessarily compact and $T$ is not necessarily a homeomorphism. 
The system is called \emph{minimal} if there is no 
proper subset $M \subseteq X$ that is nonempty, closed and $T$-invariant 
(i.e., $T(M)\subseteq M$). In such a case, we also say that the map $T$ itself 
is minimal. Clearly, a system $(X, T)$ is minimal if and only if the 
(forward) orbit $\Orb_T(x) = \{x, T(x), T^2(x), \dots \}$ of every point 
$x\in X$ is dense in $X$. Recall that for a homeomorphism $T$ on a compact space $X$, minimality is equivalent to the density of all full orbits.

Throughout the present paper, a (metrizable) space admitting a minimal map 
is said to be a \emph{minimal space}. The classification of 
(compact) minimal spaces is a well-known open problem in topological 
dynamics that is solved only in some particular cases; for some references see, 
e.g.,~\cite{KS}.

\subsection{Minimal spaces with nonminimal squares}\label{SS:minXnonminX2}

Even such a basic and natural question, which is explicitly posed 
in~\cite[p.~126]{BKS}, as to whether the product of two compact minimal spaces 
is a minimal space has not been answered to date in its full generality, 
though recently a negative answer has been provided in the special case 
when homeomorphisms rather than continuous maps are considered. In fact, 
Boro\'nski, Clark and Oprocha~\cite{BCO} found a continuum $Y$
 admitting a minimal \emph{homeomorphism} such that $Y\times Y$ does not admit any 
minimal \emph{homeomorphisms} (recall that a \emph{continuum} is a compact connected metrizable space). One of the aims of the present paper is to 
solve the abovementioned problem completely, 
i.e.,~to find a space admitting both a minimal homeomorphism and a minimal noninvertible map,
whose square is not minimal, i.e., it admits neither a minimal homeomorphism nor a minimal noninvertible map.

To present out results, we need to introduce some terminology. Recall that, by \cite{DST}, a compact metrizable 
space $X$ is called a \emph{Slovak space} if it has at least three elements, 
admits a minimal homeomorphism $T$ and the group $\Homeo(X)$ of all homeomorphisms on $X$ 
is $\Homeo(X) = \{T^n \colon \, n\in \mathbb Z\}$. In~\cite{DST} it has been
proved that if $X$ is a Slovak space then it is a continuum, the cyclic 
group $\Homeo(X)$ is infinite (i.e., isomorphic to $\mathbb Z$) and all its 
elements, except the identity, are minimal homeomorphisms. 

In \cite[Section 4]{DST} a class of Slovak spaces has been constructed; these particular Slovak spaces are said to be \emph{DST spaces} throughout the present paper. These spaces are described in detail in Section~\ref{Sec:the.Slovak}. 
For now, let us only mention that no DST space admits a minimal noninvertible map.

The third part of the following theorem gives a definitive negative answer to the mentioned question from \cite{BKS};
there, $\mathbb T^n$ denotes the $n$-torus.
The first two parts of the theorem are of independent interest. (For the definition of a product-minimal space used in the second part, see Definition~\ref{D:pum space}.)

\begin{theoremA}\label{T:Slovak} Let $X$ be an arbitrary DST space.
\begin{enumerate}
\item 
The space $X$ admits a minimal homeomorphism, but $X\times X$ does not admit any minimal continuous map (see Theorem~\ref{T:thmA1}).

\item 
Let $Y$ be a product-minimal path-connected continuum.
Then $X\times Y$ is a minimal space\footnote{The space $X\times Y$ admits a minimal homeomorphism or a minimal noninvertible map or both, depending on $Y$.} with nonminimal square	
(see Theorem~\ref{T:XxY-nonminimal-square}).

\item 
Let $n\ge 2$ be an integer.
Then the space $X\times\mathbb T^n$ admits a minimal homeomorphism as well as a minimal noninvertible map,
but its square is not minimal
(see Theorem~\ref{T:XxT-nonminimal-square}).
\end{enumerate}

\end{theoremA}

We do not know whether Theorem~A is true for \emph{every} Slovak space $X$, but it is true for DST spaces. Let us now compare such 
a DST space $X$ with the space $Y$, which is constructed in \cite{BCO} to prove a weaker version of (1). 
Since the construction of $Y$ in \cite{BCO} is inspired by that of $X$ in \cite{DST}, the spaces $X$ and $Y$ share some common features. However, there are also significant differences between them.

\begin{itemize}
\item Both $X$ and $Y$ admit minimal homeomorphisms. While $Y^2$ does not admit minimal
homeomorphisms, leaving open the possibility that it admits minimal continuous maps, 
we prove that $X^2$ does not admit minimal continuous maps at all (to prove that 
$X^2$ does not admit minimal homeomorphisms is much easier, see 
Theorem~\ref{T:TA.homeo}).

\item Both $X$ and $Y$ are nondegenerate indecomposable continua, so they have 
uncountably many composants, see~\cite[Theorems 11.13 and 11.15]{Nad}. 

\item For both $X$ and $Y$, it is true that all but one of their composants are 
continuous injective images of the real line. 

\item The exceptional composant of $Y$ consists of countably many pseudo-arcs connected by
arcs, while the exceptional composant of $X$ is the union of countably 
many topologist's sine curves (see Figure~\ref{Fig:gamma} in Section~\ref{Sec:the.Slovak}). 
Therefore, each nondegenerate 
proper subcontinuum of $X$ is decomposable, 
while $Y$ does not have this property.

\item Not only for $N=2$ but even for every $N\ge 2$, the space $Y^N$ does not admit a minimal homeomorphism.
The same is true for $X$ (see Remark~\ref{R:homeo-case-XN}). However, it is not known whether $Y^N$ ($N\ge 2$) and $X^N$ ($N\ge 3$) admit
minimal noninvertible maps.
\end{itemize}

Every DST space is a one-dimensional continuum. Theorem~A(2) shows that there are other possibilities for the topological dimension of a minimal space that has nonminimal square. In fact, 
such spaces can have arbitrary positive topological dimensions, including infinity.

Concerning the minimality of spaces and their powers, Theorem~A shows that there is a space $X$ such that the set $S_X=\{n\in\mathbb N\colon X^n \text{ is minimal}\}$
contains $1$ and does not contain $2$. This situation suggests the following (very general but) interesting problem.
\begin{problem}
	Characterize all sets $S\subseteq \mathbb N$ for which there exists a compact metrizable 
	space $X$ such that $X^n$ is a minimal space if and only if $n\in S$ (that is, such that $S_X=S$).
\end{problem}

Our interest in this paper is in discrete dynamical systems. Let us mention, however, that for continuous flows 
a counterexample as in Theorem~A is not possible. In fact, the class of compact metrizable spaces admitting 
minimal continuous flows is closed with respect to at most countable products; 
see~\cite[Theorem 25]{Di} for an even stronger result.

In connection with Theorem~A, let us also mention some other interesting facts.
\begin{itemize}
	\item 
	There exist nonminimal continua $X,Z$ such that
	the product $X\times Z$ is minimal; see Example~\ref{e:1}.
	
	\item 
	There exist minimal continua $X,Z$ such that $X\times Z$
	admits a minimal direct product and every 
	skew-product on $X\times Z$ is a direct product;
	see Example~\ref{e:3}.\footnote{As usual, by a skew product on $X\times Z$ we mean a continuous map of the form $F(x,z)=(f(x), g(x,z))$. We often write $g_x$ instead of $g(x,\cdot)$ and $(f,g_x)$ instead of $F$.}
	
	\item 
	There exist continua $X,Y$ such that all the spaces $X,Y,X\times Y$ 
	admit minimal homeomorphisms and every homeomorphism on $X\times Y$ takes 
	the form of a direct product; see Example~\ref{e:factorwise-rigid}. 
	Therefore, $X\times Y$ is a factorwise rigid\footnote{The Cartesian product $X\times Y$
		is called \emph{factorwise rigid} provided that if $h\colon X \times Y \to X\times Y$
		is a homeomorphism, then $h$ is either of the form
		$h(x, y) = (f(x), g(y))$, where $f\colon X\to X$ and $g\colon Y\to Y$ are homeomorphisms, or
		$h(x, y) = (f(y), g(x))$, where $f\colon Y\to X$ and $g\colon X\to Y$ are
		homeomorphisms. Of course, if $X$ and $Y$ are not homeomorphic, then the latter case cannot
		occur. The notion of factorwise rigidity has been used since~\cite{BL}, but already in~\cite{KKT} 
		it was proven that the product of two Menger universal curves has the described property.} continuum admitting 
	a minimal homeomorphism. 
\end{itemize}

\subsection{Product-minimal spaces}\label{SS:PMspaces}

Minimal direct product systems may seem too special, for instance, compared with
minimal skew products, but sometimes they naturally appear and prove to be useful. Let us illustrate this situation by an example. First, recall that by a theorem due to H. Weyl (see, 
e.g.,~\cite[Chapter I, Theorem~4.1]{KuiNie}), if $(a_n)_{n=1}^\infty$ is 
a sequence of distinct integers, then the sequence $(a_n\theta)_{n=1}^\infty$ 
is uniformly distributed $\mymod 1$ for almost all real numbers $\theta$. 
Using this fact, one can show (see, e.g.,~\cite[Proposition 1]{KST}) that 
if $(X,T)$ is a (not necessarily compact) minimal dynamical system, then there exists an (irrational) rotation $R$ of the circle 
$\mathbb S^1$ such that the direct product $(X\times \mathbb S^1, T\times R)$ 
is minimal. This result has proven to be useful in the description of minimal 
sets of fiber-preserving maps in graph bundles (see~\cite{KST}). Note that this 
shows that circle $\mathbb S^1$ is a (homeo-)product-minimal space 
according to Definition~\ref{D:pum space} below.

The spaces discovered in \cite{DST} arose as a result of a (more or less) explicit construction, and they can certainly be said to admit ``few'' minimal maps. In many other situations in the literature, the minimality of a space $Y$ is verified by the Baire category method; in this way, it is shown that minimality is a typical property among the continuous maps on $Y$ from a particular class. One is then naturally led to a (vague) claim that such a space $Y$ admits ``many'' minimal maps.

If we want to make the expression ``many'' be precise, there are several possible notions to pursue, including cardinality, topological size, algebraic size and others. Each of these approaches has advantages in one situation or another. 
What we expect from the notion of ``many'' is that the family of all minimal maps on $Y$ is large enough compared with all minimal dynamical systems. For instance, on the circle there is a large enough family of minimal maps in the sense described above: for an arbitrary prescribed minimal dynamical system, the circle supports a minimal rotation disjoint from it.

These ideas lead us to the following definition.

\begin{definition}\label{D:pum space}
	A compact metrizable space $Y$ is said to be \emph{product-minimal}, 
	or a \emph{PM space}, if for every minimal dynamical system $(X,T)$ with $X$ metrizable, not necessarily compact, there exists 
	a continuous map $S\colon Y\to Y$ such that the direct product 
	$(X\times Y, T \times S)$ is minimal.  If
	such a map $S$ always exists in the class of homeomorphisms, we say that $Y$ is a 
	\emph{homeo-product-minimal space} or an \emph{HPM space}. 
\end{definition}

A long (incomplete) list of (homeo-)product-minimal spaces can be found in Theorem~B below. By definition, every HPM space is a PM space and every PM space is minimal.  There are PM spaces that are not HPM; see Theorem~\ref{T:cantoroid} and Remark~\ref{R:PMnotHPM}. Furthermore, every DST space $X$ is a minimal continuum that is not product-minimal (see Proposition~\ref{P:Slovak.not.HPM}(b) and Example~\ref{e:M-not-PM}). Notice also that no Slovak space is homeo-product-minimal (see Proposition~\ref{P:Slovak.not.HPM}(a)).

\medskip

Note that in Definition~\ref{D:pum space}, the space $X$ in the base is not assumed to be compact because, quite surprisingly, our methods used to prove Theorem~B do not require the compactness of $X$. 
We emphasize, however, that the metrizability of $X$ is essential. Because if we allowed, for instance, the universal compact minimal dynamical system in the base, then (homeo-)product-minimality would reduce to the class of degenerate spaces alone.

On the other hand, a product-minimal space $Y$ is compact by definition. 
In fact, the compactness of $Y$ is used heavily in our proof of Theorem~B (Lemma~\ref{L:fiber} indicates a reason for this assumption). Moreover, the compactness assumption is less restrictive 
than it might seem. In fact, an obvious necessary condition for the product-minimality
of a space $Y$ is its minimality. However, minimal metrizable spaces are 
either compact or not locally compact (by Gottschalk's theorem~\cite{Gott},
if $X$ is a noncompact Hausdorff space with a compact subset having nonempty 
interior, then $X$ does not admit any minimal map). Therefore, if we do not 
want to complicate things by considering spaces $Y$ that are not locally 
compact, we may and must assume the compactness of $Y$.\footnote{An example of a space $Y$ 
that satisfies the definition of (homeo-)product-minimality except for the 
compactness requirement is the space $\mathbb Q$ of  rational numbers. In fact, given a minimal system $(X,T)$, we can find a minimal rotation $R$ of $\mathbb S^1$ such that the product $T\times R$ is minimal. By restricting $R$ on a single full orbit, we obtain a homeomorphism $S$ on $\mathbb Q$ with $T\times S$ minimal.}

\medskip

If, in Definition~\ref{D:pum space}, we were interested in \emph{skew product} systems on $X\times Y$ rather than in direct product systems, then many more spaces $Y$ would be product-minimal, at least if $X$ were assumed to be compact. For instance, by~\cite{GW,DM}, every compact connected manifold $Y$ without boundary has the property that for every (not necessarily invertible) minimal dynamical system $(X,T)$ on a compact metrizable space $X$ there exists a minimal skew product system on $X\times Y$ with the base map $T$. Among such manifolds $Y$ there are also those that are not minimal, say $\mathbb{S}^2$. Therefore, such minimal skew product systems cannot generally have the form of a direct product. However, as we have seen, in the case of $Y=\mathbb S^1$ such a minimal skew product system exists even in the class of direct products. Moreover, Theorem~B(8) shows that this is also true for many other compact connected manifolds $Y$ without boundary.
	
Let us mention that if $X$ and $Y$ are disjoint unions of two circles (or both $X$ and $Y$ consist of two points only), then $X\times Y$ admits a minimal skew product but does not admit any minimal direct products; see Example~\ref{e:2}. (We believe that there also exist minimal continua $X,Y$ with this property. Obviously, none of them can be a PM space though.)

In connection with this example, the following problem seems interesting.

\begin{problem}\label{Prob:skew}
	Let $Y$ be a \emph{minimal} continuum  such that for every (say, compact) minimal system $(X,T)$, the product $X\times Y$ admits a minimal skew product $(T,g_x)$ with the base $T$. Does it follow that $Y$ is a PM space? If, in addition, such a skew product can always be found with invertible fiber maps $g_x$, then does it follow that $Y$ is an HPM space?
\end{problem}

Notice that if $(T,g_x)$ is a minimal skew product on $X\times Y$ and
a minimal direct product of the form $T\times S$ exists on $X\times Y$, then it is still possible that $T\times g_{x}$ (in fact, $g_x$) is nonminimal for every $x\in X$, as Example~\ref{E:skew.direct.fiber} shows.
Further, without the assumption of the minimality of $Y$, the answer to Problem~\ref{Prob:skew} would be negative. Indeed, for instance, the sphere $Y=\mathbb S^2$ would be a counterexample by its nonminimality and \cite[Theorem~1]{GW}. 

\medskip

Another main aim of the present paper, in addition to Theorem~A, is to study which of the familiar minimal spaces are (homeo-)product-minimal. This is an interesting
problem on its own, but one can also hope that results in this direction can be a useful tool (recall that the product-minimality of
the circle has been used in~\cite{KST} in the study of the topology of minimal sets).
As a trivial observation, notice that the product of a minimal space with a product-minimal space is minimal.
Therefore, the (incomplete) list of product-minimal spaces in Theorem~B given below
enables the production of new minimal spaces from the old spaces (of course, the minimality of some of them also follows from earlier results, such as \cite{GW,DM}).

In the next theorem, we summarize our main results on 
product-minimal spaces. Notice that part (1) generalizes the 
fact that the circle is an HPM space. 
For the definitions of cantoroids and Sierpi\'nski curves, we refer the reader to 
Sections~\ref{S:CtrCtrd} and~\ref{S:Sierp},
respectively.

\begin{theoremB}\label{T:PM spaces}
	Each of the following spaces is an HPM space, hence also a PM space.
	\begin{enumerate}
		\item Every compact connected metrizable abelian group 
		(for a stronger result, see Theorem~\ref{T:groups}, cf.~Theorem~\ref{skew.dir}).
		
		\item Every space of the form $Y\times Z$, where $Y$ is a 
		nondegenerate HPM space and $Z$ is a compact metrizable space 
		admitting a minimal action of an arc-wise connected topological group
		(see Theorem~\ref{T:prod.PM.arc}; cf. Corollary~\ref{C:prod.PM.arc}).
		
		\item Every quotient space $(\Gamma\times Z)/\Lambda$, obtained from 
		$\Gamma\times Z$ by applying the diagonal action of $\Lambda$, where
		$\Gamma$ is an infinite compact connected metrizable abelian group, 
		$\Lambda$ is a finite subgroup of $\Gamma$ and
		$Z$ is a compact connected (not necessarily abelian) metrizable group, 
		on which the group $\Lambda$ acts by automorphisms 
		(see Theorem~\ref{T:fact.groups} and, for a more general result,
		Theorem~\ref{T:fact.space}).
		
		\item The Klein bottle (see Theorem~\ref{T:Klein}).
		
		\item The Cantor space (see Theorem~\ref{T:Cantor}).
		
		\item The Sierpi\'nski curve on the 2-torus and the Sierpi\'nski curve 
		on the Klein bottle (see Theorem~\ref{T:Sierp}).
		
		\item Every compact metrizable space $Y$ admitting a minimal continuous flow 
		whose centralizer in the group of homeomorphisms of $Y$ acts transitively 
		on $Y$ in the algebraic sense (see Theorem~\ref{T:flow.centr}, cf.~also 
		Theorem~\ref{T:flow.centr.gp}).
		
		\item Every compact connected manifold $Y$ without boundary admitting
        a free action of a nontrivial compact connected Lie group
		(see Theorem~\ref{T:manif.homeo-PM.Lie}; for an analogous result in the smooth category, see Theorem~\ref{T:manif.diff-PM.Lie}).
		
		\item All odd-dimensional spheres and compact connected Lie groups (see Remarks~\ref{R:manif.homeo-PM} and~\ref{R:manif.homeo-PM.Lie}).
				
	\end{enumerate}
Moreover,
\begin{enumerate}
\item [(10)] every cantoroid is a PM space (see Theorem~\ref{T:cantoroid}).
\end{enumerate}	
\end{theoremB}  

Since the circle is a homeo-product-minimal space, the following question is natural (and, apparently, quite challenging).

\begin{problem}
Is the pseudo-circle a (homeo-)product-minimal space?
\end{problem}

All solenoids are homeo-product-minimal by virtue of Theorem~B(1) or~(7). We do not know whether this is also true for generalized solenoids (see Section~\ref{Sec:the.Slovak} for the definition). Generalized solenoids that are not solenoids still admit minimal continuous flows
but, being nonhomogeneous (see footnote~\ref{ftn:generalized-solenoid} in Section~\ref{Sec:the.Slovak}), 
they do not satisfy the assumption from Theorem~B(7), as follows from Remark~\ref{R:homogeneous}.

\begin{problem}
Are all generalized solenoids (homeo-)product-minimal spaces?
\end{problem}

The notion of product-minimality is related to that of the disjointness of dynamical systems. Indeed, recall that the product of two compact minimal systems is minimal if and only if the two systems are disjoint. The theory of disjointness, initiated in \cite{F}, is now a deeply developed part of the theory of dynamical systems (see \cite{Gla} and the references therein). The main problem in the theory consists of determining the class of all systems that are disjoint from every system in a given class. However, when studying product-minimal systems, these results cannot be directly applied, since in the definition of a product-minimal space no additional information on the system $(X,T)$ is available, except its minimality.

\medskip

The paper is organized as follows. In Section~\ref{S:pm-observations}, we summarize some useful facts on minimal products and product-minimal spaces; in particular, we describe operations under which the class of (homeo-)product-minimal spaces is closed. In Sections~\ref{S:thmB-groups}--\ref{S:thmB-smooth-manifolds},
we prove all the theorems that constitute the ten parts of Theorem~B. 
Then, in Section~\ref{S:examples}, we present the examples mentioned in the introduction. Finally, in Sections~\ref{Sec:the.Slovak}--\ref{Sec:other.min}, we prove all three parts of Theorem~A.

Sections~\ref{S:pm-observations}--\ref{S:examples}, which center around Theorem~B, are mostly
attributable to the first author, and Sections~\ref{Sec:the.Slovak}--\ref{Sec:other.min},
which center around Theorem~A, are mostly attributable to the second and third authors.
However, both parts are closely related, and we believe that they deserve
to be published together as a whole.
We started this research more than 10 years ago, being motivated by the question of
whether the product of minimal spaces has to be minimal and by the fact that the circle
is a homeo-product-minimal space (in the terminology introduced above).


\section{First observations on minimal products and product-minimal spaces}\label{S:pm-observations}

Let us begin by noticing that some new (homeo-)product-minimal spaces can be produced from the old ones.

\begin{proposition}
	The class of all (homeo-)product-minimal spaces is closed with respect to at most countable products.
\end{proposition}
\begin{proof}
	We prove the proposition only for HPM spaces; the proof for PM spaces is analogous. So let $Y_i$, $i=1,2,\dots$, be 
	HPM spaces. Then $Y_i$ are compact metrizable by definition and hence so is $Y=\prod_{i=1}^\infty Y_i$. Let $(X,T)$ be a minimal system. Since $Y_1$ is HPM, there is a homeomorphism $S_1$ on $Y_1$ such that $(X\times Y_1, T\times S_1)$ is minimal. Since $Y_2$ is HPM, there is a homeomorphism $S_2$ on $Y_2$ such that $(X\times Y_1 \times Y_2, T\times S_1 \times S_2)$ is minimal. Continuing in this way, we find a sequence of homeomorphisms $S_i\colon Y_i\to Y_i$ ($i\in\mathbb N$) such that the product $\big(X\times \prod_{i=1}^N Y_i, T\times \prod_{i=1}^N S_i\big)$ is minimal for every $N\in\mathbb N$. 
	Then $S = \prod_{i=1}^{\infty} S_i$ is a homeomorphism on $Y$. Moreover, the system $(X\times Y, T\times S)$ is minimal, being an inverse limit of the minimal systems $\big(X\times \prod_{i=1}^N Y_i, T\times \prod_{i=1}^N S_i\big)$ ($N=1,2,\dots$).
\end{proof}

\begin{remark}
	There are also other, much less obvious ways to produce new (homeo-)product-minimal spaces from the old ones:
	\begin{itemize}
		\item by multiplying with an appropriate (not necessarily minimal) space; see Theorem~\ref{T:prod.PM.arc} and Corollary~\ref{C:prod.PM.arc},
		\item by passing to special quotient spaces of products; see Theorem~\ref{T:fact.space} and Theorem~\ref{T:fact.groups}, and
		\item by passing to special almost 1-1 extensions; see Proposition~\ref{P:general}.
	\end{itemize}
\end{remark}

In the following proposition, $\mathfrak{c}$ stands for the cardinality of the continuum.

\begin{proposition}\label{P:PM.con.comp}
	A PM space is either connected or has $\mathfrak{c}$ components. Consequently, a nondegenerate PM space has no isolated points.
\end{proposition}
\begin{proof}
	Let $Y$ be a PM space. By minimality, $Y$ has either finitely many or $\mathfrak{c}$ components. So assume that the number $k$ of components of $Y$ is finite, $k\geq2$. Let $(X,T)$ be the cyclic permutation of a space $X$ with $k$ elements. If $S\colon Y\to Y$ is minimal then $S$ cyclically permutes the $k$ components of $Y$. Consequently, the nonminimal system $(X\times X,T\times T)$ is a factor of $(X\times Y,T\times S)$ and so the latter system is not minimal. This proves the first statement of the proposition. To prove the second statement, notice that if a minimal system has an isolated point then it is a (finite) periodic orbit.
\end{proof}

As explained in Subsection~\ref{SS:PMspaces}, intuitively, a necessary condition for a space $Y$ to be HPM or PM is that it admits sufficiently many minimal homeomorphisms or minimal continuous maps. Therefore, the following proposition is not surprising.

\begin{proposition}\label{P:Slovak.not.HPM}
	Let $Y$ be a nondegenerate compact metrizable space.
	\begin{enumerate}
		\item[(a)] If the set of all minimal homeomorphisms on $Y$ is (at most) countable then $Y$ is not a homeo-product-minimal space. In particular, no Slovak space is homeo-product-minimal.
		\item[(b)] If the set of all minimal continuous maps on $Y$ is (at most) countable then $Y$ is not a product-minimal space. In particular, no DST space is product-minimal.
	\end{enumerate} 
\end{proposition}
\begin{proof}
	We prove part (a); the proof of part (b) is similar. Let $(X,T)$ be an arbitrary minimal subsystem of the direct product of all the systems $(Y,S)$ with $S$ being a minimal homeomorphism on $Y$. By compactness of $Y$ and the countability assumption, $X$ is a (nonempty) compact metrizable space. 
	 We claim that there is no homeomorphism $S\in\Homeo(Y)$ with $(X\times Y,T\times S)$ minimal. 
	 So let $S\in\Homeo(Y)$ be minimal. By minimality of $T$, $X$ has full projections onto all coordinates, so	 
	 $(Y,S)$ is a factor of $(X,T)$.
	 Since the product $(Y\times Y,S\times S)$ is not minimal because $Y$ is not a singleton, it follows that the product $(X\times Y,T\times S)$ is not minimal. Thus, $Y$ is not a homeo-product-minimal space.
\end{proof}

In a dynamical system $(X,T)$, a point $x$ whose orbit $\{x, T(x), T^2(x), \dots\}$
is dense in $X$, is called a \emph{transitive point}. A system is called \emph{point-transitive}
if it has a transitive point. 
Clearly, if $(X,T)$ is a point-transitive system then the space $X$, being separable and metrizable, has a countable basis.
Notice that a system is minimal if and only if 
all points are transitive or, equivalently, if and only if the \emph{$\omega$-limit set} $\omega_T(x)$ of
every point $x\in X$ is the whole space $X$ (recall that $\omega_T(x)$ is the set of all limit points of the \emph{trajectory} $(T^n(x))_{n=0}^\infty$ of $x$). A system $(X,T)$ is called \emph{(topologically) transitive}
if for each pair of nonempty open sets $U$ and $V$ in $X$ there is $n\in \mathbb N$ such that
$T^n(U) \cap V \neq \emptyset$.
If $(X,T)$ is a dynamical system and a subset $A$ of $X$ is $T$-invariant, 
then we also use a less precise notation $(A,T)$ for the system $(A, T|_A)$.

The following lemma is certainly well known, but we record it for a future reference nonetheless.

\begin{lemma}\label{L:fiber}
	Let $(X,T)$ be a minimal system and $Y$ be a compact metrizable space.
	Then a skew product map $Q$ on the space $X\times Y$ of the form $Q(x,y)=(T(x), S_x(y))$
	is minimal if and only if there is $x_0\in X$ such that the whole fiber $\{x_0\}\times Y$ consists of transitive points of $Q$. 
\end{lemma}

\begin{proof}
	One implication is trivial. To prove the converse implication, assume that $\{x_0\}\times Y$ 
	consists of transitive points of $Q$. Fix $(x,y)\in X\times Y$. Using minimality of $T$ and 
	compactness of $Y$, we find $y_0\in Y$ such that 
	$\omega_{Q}(x,y) \ni (x_0,y_0)$. Then, since $(x_0,y_0)$ is transitive, it follows that $\omega_Q(x,y)\supseteq\omega_Q(x_0,y_0)=X\times Y$, hence $(x,y)$ is also transitive.
\end{proof}

Let $(X,T)$ be a dynamical system.
It is called \emph{totally minimal}
if all the iterates $T^n$, $n\in\mathbb N$, are minimal.
A point $x\in X$ is said to be \emph{periodically recurrent} 
if for every neighborhood $U$ of $x$ there exists a positive integer $p$ such that $x$ 
returns back into $U$ with period~$p$, i.e., $T^{kp}(x) \in U$ for $k=0,1,2,\ldots$.

The following result is presumably well known, but we are unable to give a reference, so 
a proof is included.

\begin{proposition}\label{P:period.rec}
		Let $(X,T)$ and $(Y,S)$ be dynamical systems given by continuous
		maps on metrizable spaces, $Y$ being a compact space. Assume that $(X,T)$ is minimal and
		$(Y,S)$ is totally minimal. If $(X,T)$ contains a
		periodically recurrent point $x_0$ then the product $(X\times Y,T
		\times S)$
		is minimal.
\end{proposition}

\begin{proof}
		Denote $F=T\times S$. By virtue of Lemma~\ref{L:fiber}, it suffices to show that $\omega_F(x_0,y)=X\times Y$ for every $y\in Y$. First we prove that for every
		$y\in Y$, $\omega_F(x_0,y) \supseteq \{x_0\}\times Y$.
		To this end fix $y,y' \in Y$ and a basic neighborhood $U\times V$
		of the point $(x_0,y')$. There exists $p\in \mathbb N$ such that
		$T^{kp}(x_0) \in U$ for all $k=0,1,2,\dots$. Since $S^p$ is
		a minimal map, there exist infinitely many times
		$k\in \mathbb N$ with $F^{kp}(x_0,y)=(T^{kp}(x_0),(S^p)^k(y)) \in U\times V$. Since
		the argument works for each basic neighborhood $U\times V$ of $(x_0,y')$, 
		we have $\omega_F(x_0,y)\ni(x_0,y')$. Consequently, $\omega_F(x_0,y) \supseteq \{x_0\}\times Y$.
		
		Now we show that $\omega_F(x_0,y)=X\times Y$.
		Clearly, for every $n\in\mathbb N$, $\omega_F(x_0,y) \supseteq F^n(\omega_F(x_0,y))\supseteq F^n(\{x_0\}\times Y)=\{T^n(x_0)\}\times S^n(Y)$. 
		Since $S$ is surjective by its minimality and compactness of $Y$, and the $T$-orbit of $x_0$ is dense in $X$, it follows that the 
		closed set $\omega_F(x_0,y)$ contains the union of a dense set of fibers. Hence $\omega_F(x_0,y)
		=X\times Y$.
\end{proof}

We conclude this section with the following property of (homeo-)product-minimal spaces.

\begin{proposition}\label{P:trans.PM}
	Let $Y$ be a (homeo-)product-minimal space. Let $(X,T)$ be a dynamical system given by a continuous map $T$ on a metrizable space $X$. Suppose that $X$ has no isolated points and $T$ is point-transitive. Then there is a minimal continuous map (homeomorphism) $S\colon Y\to Y$ such that $(X\times Y,T\times S)$ is point-transitive.
\end{proposition}
\begin{proof}
	Fix a transitive point $x_0$ for $T$ and denote its orbit by $X'$. Since $x_0$ is not isolated, it is recurrent and so the restriction of $T$ onto $X'$ is minimal. Since the space $Y$ is (homeo-)product-minimal, there is a minimal continuous map (homeomorphism) $S\colon Y\to Y$ such that the product $(X'\times Y,T|_{X'}\times S)$ is minimal. Finally, by density of $X'\times Y$ in $X\times Y$ it follows that the product $(X\times Y,T\times S)$ is point-transitive.
\end{proof}

\section{Compact connected abelian groups}\label{S:thmB-groups}

\subsection{Auxiliary results}

Let $G$ be a compact abelian group. The multiplicative semigroup of positive integers $\mathbb N^*$ acts on $G$ by endomorphisms $A_n$ defined by $A_n(g)=g^n$ ($g\in G$, $n\in\mathbb N^*$); the corresponding action will be denoted by $\pi$. If $G$ is connected then all $A_n$ are epimorphisms \cite[Theorem~8.4]{HofMor} and hence they preserve the (normalized) Haar measure $\mu$ of $G$ \cite[p.~20]{Wal}. We shall denote by $\widehat{G}$ the dual group of characters of $G$. Recall that  $\widehat{G}$ is an orthonormal basis of the (complex) Hilbert space $L^2(\mu)$ \cite[p.~13]{Wal}. Moreover, $\widehat{G}$ is torsion-free if and only if $G$ is connected \cite[Corollary~8.5]{HofMor}.

\begin{lemma}\label{s.mix}
	Let $G$ be a compact connected abelian group and let $\mu$ be the Haar
	measure on $G$. Then the action $\pi$ of $\mathbb N^*$ on $G$ is
	mixing with respect to $\mu$, that is, for all Borel sets $B,C\subseteq G$,
	\begin{equation*}
	\lim_{n\to\infty}\mu(A_n^{-1}(B)\cap C)=\mu(B)\mu(C).
	\end{equation*}
\end{lemma}
\begin{proof}
	\emph{1st step.} We show that for all $\gamma,\delta\in\widehat{G}$,
	\begin{equation}\label{Eq:s.mix.1}
	\lim_{n\to\infty}\int(\gamma\circ A_n)\delta \,d\mu=\int\gamma \,d\mu\int\delta\, d\mu.
	\end{equation}
	
	Fix $\gamma,\delta\in\widehat{G}$. If $\gamma=1$ then \eqref{Eq:s.mix.1} holds obviously. So assume that $\gamma\neq1$. Then $\int\gamma\,d\mu=0$. Since $\widehat{G}$
	is torsion-free, we have $(\gamma\circ A_n)\delta=\gamma^n\delta=1$
	for at most one $n\in\mathbb N^*$; hence $\int(\gamma\circ A_n)\delta\,d\mu=0$ for all sufficiently large $n$. Thus,
	\begin{equation*}
	\lim_{n\to\infty}\int(\gamma\circ A_n)\delta\,d\mu=0=\int\gamma\,d\mu\int\delta\,d\mu.
	\end{equation*}
	
	\emph{2nd step.} We show that for all $f,g\in L^2(\mu)$, 
	\begin{equation}\label{Eq:s.mix.2}
	\lim_{n\to\infty}\int (f\circ A_n)g\,d\mu=\int f\,d\mu\int g\,d\mu.
	\end{equation}
	
	Since $\widehat{G}$ is an orthonormal basis of $L^2(\mu)$ and we have \eqref{Eq:s.mix.1}, the proof of \eqref{Eq:s.mix.2} is routine and goes as follows. Given $f\in L^2(\mu)$, denote by $R_f$ the set of all $g\in L^2(\mu)$ such that \eqref{Eq:s.mix.2} holds. Similarly, given $g\in L^2(\mu)$, denote by $L_g$ the set of all $f\in L^2(\mu)$ such that \eqref{Eq:s.mix.2} holds. A standard argument shows that all $R_f$ and $L_g$ are closed linear subspaces of $L^2(\mu)$. Now, by the 1st step of the proof, for every $\delta\in\widehat{G}$ we have $L_{\delta}\supseteq\widehat{G}$ and so $L_{\delta}=L^2(\mu)$. This means that, for every $f\in L^2(\mu)$, $R_f\supseteq\widehat{G}$ and so $R_f=L^2(\mu)$, which finishes the 2nd step. 
	
	\emph{3rd step.} We prove the lemma.
	
	So fix Borel sets $B,C\subseteq G$. By applying \eqref{Eq:s.mix.2}
	to the characteristic functions $f=\chi_B$, $g=\chi_C$ of $B,C$, we obtain
	\begin{eqnarray*}
		\lim_{n\to\infty}\mu(A_n^{-1}(B)\cap C)&=&\lim_{n\to\infty}\int\chi_{A_n^{-1}(B)\cap C}\,d\mu=\lim_{n\to\infty}		\int(\chi_B\circ A_n)\chi_C\,d\mu\\
		&=&\int\chi_B\,d\mu\int\chi_C\,d\mu=
		\mu(B)\mu(C).
	\end{eqnarray*}
\end{proof}

Let $G$ be a compact abelian group. Recall that the action $\pi$ of $\mathbb N^*$ on $G$ is
\emph{topologically transitive} or \emph{topologically mixing} if for all nonempty open sets $U,V\subseteq G$, $A_n^{-1}(U)\cap V\neq\emptyset$ holds for some $n$ or for every sufficiently large $n$, respectively.

\begin{lemma}\label{L:mix.group}
	Let $G$ be a compact abelian group and let $\pi$ be the action of
	$\mathbb N^*$ on $G$. Then the
	following conditions are equivalent:
	\begin{enumerate}
		\item[(1)] $\pi$ is mixing with respect to the Haar measure $\mu$ on $G$,
		\item[(2)] $\pi$ is topologically mixing,
		\item[(3)] $\pi$ is topologically transitive,
		\item[(4)] $G$ is connected.
	\end{enumerate}
\end{lemma}
\begin{proof}
	Implication $(1)\Rightarrow(2)$ follows from the fact that $\mu$ has full
	support and $(4)\Rightarrow(1)$ follows from Lemma~\ref{s.mix}.
	Implication $(2)\Rightarrow(3)$ is obvious. We verify $(3)\Rightarrow(4)$
	by contradiction.
	
	So assume that $\pi$ is topologically transitive and $G$ is not
	connected. Denote by $G_0$ the identity component of $G$.
	Then $G_0$ is a closed subgroup of $G$ and the quotient group $G/G_0$
	is totally disconnected. Denote by $\pi'$ the natural action of
	$\mathbb N^*$ on $G/G_0$ and by $A_n'$ ($n\in\mathbb N^*$) the
	corresponding acting endomorphisms. Since $\pi$ is topologically transitive and factors onto $\pi'$ (via the quotient morphism $G\to G/G_0$),
	it follows that $\pi'$ is also topologically transitive. Now $G/G_0$, being nontrivial, compact and totally disconnected, has a proper
	open (hence also closed) subgroup $H\subsetneq G/G_0$ \cite[Theorem~1.34]{HofMor}. Obviously, for any $n\in\mathbb N^*$,
	we have $A'_n(H)\subseteq H$ and so $A_n'(H)\cap ((G/G_0)\setminus H)=
	\emptyset$. Since both $H$ and $(G/G_0)\setminus H$ are nonempty and open,
	this contradicts the transitivity of $\pi'$.
\end{proof}

Now we recall some useful facts from the theory of uniformly distributed
sequences in compact groups. We start with the definition of a uniformly
distributed sequence, which is due to Eckmann (see \cite{Eck} or \cite[p.~221]{KuiNie}).

\begin{definition}\label{D:equidistr}
	A sequence $(g_n)_{n=1}^{\infty}$ in a compact group $G$ is
	\emph{uniformly distributed} in $G$ if for every continuous complex
	valued function $\varphi$ on $G$ we have
	$$
	\lim_{n\to\infty}\frac{1}{n}\sum_{i=1}^n\varphi(g_i)=\int\varphi\,d\mu\,,
	$$
	where $\mu$ denotes the Haar measure on $G$.
\end{definition}

Since $\mu$ has full support, it follows that for every uniformly distributed
sequence $(g_n)_{n=1}^{\infty}$ in $G$, the set $\{g_n\,:\,n\in\mathbb N\}$ is
dense in $G$.

We shall make use of the following result due to Hartman and Ryll-Nardzewski \cite[Satz~7]{HarRyl},
cf.~\cite[p.~279]{KuiNie}.

\begin{lemma}\label{unif.distr}
	Let $G$ be a compact, connected, metrizable, abelian
	group and let $(k_n)_{n=1}^{\infty}$ be an increasing sequence
	of positive integers. Then the sequence $(g^{k_n})_{n=1}^{\infty}$
	is uniformly distributed in $G$ for $\mu-$almost every
	$g\in G$, where $\mu$ denotes the Haar measure on $G$.
\end{lemma}


\subsection{Main result}

Let $(X,T)$ be a dynamical system. Recall that the \emph{omega-limit} of a set
$A\subseteq X$ is defined by
$$
\omega_T(A)=\bigcap_{n=0}^{\infty}\overline{\bigcup_
	{m=n}^{\infty}T^m(A)}\,.
$$
If $A\subseteq B\subseteq X$ then $\omega_T(A)\subseteq\omega_T(B)$. If $X$ is compact and $A\subseteq X$ then $T(\omega_T(A))=\omega_T(A)$ and so $\omega_T\left(\omega_T(A)\right)=
\omega_T(A)$. Given $x\in X$, we write $\omega_T(x)$ instead of $\omega_T(\{x\})$.
Recall that a system $(X,T)$ is minimal if and only if $\omega_T(x)=X$ for every $x\in X$.

Given a compact abelian group $G$ and a point $g\in G$, we shall denote by $R_g$ the rotation of $G$ by $g$.

As we already observed in the introduction, circle $\mathbb S^1$ is a homeo-product-minimal space. 
We show that much more is true.

\begin{theorem}\label{T:groups}
	Let $(X,T)$ be a minimal dynamical system given by a continuous
	map $T$ on a metrizable space $X$ and let $G$ be a compact
	connected metrizable abelian group. Then the product $(X\times G,
	T\times R_g)$ is minimal for every $g$ from a residual subset
	of full Haar measure. Consequently, every compact connected metrizable abelian group is homeo-product-minimal.
\end{theorem}
\begin{proof}
	Set
	\begin{equation}\label{Eq:Q.expr.1}
	Q=\{g\in G\colon T\times R_g \text{ is minimal}\}.
	\end{equation}
	
	\emph{1st step.} We show that 
	\begin{equation}\label{Eq:Q.expr.2}
	Q=\{g\in G\colon \text{there is } (x_0,h_0)\in X\times G \text{ with } \omega_{T\times R_g}(x_0,h_0)\supseteq \{x_0\}\times G\}.
	\end{equation}
	
	The inclusion ``$\subseteq$'' is clear. To prove the converse one, fix $(x_0,h_0)\in X\times G$ with $\omega_{T\times R_g}(x_0,h_0)
	\supseteq \{x_0\}\times G$. We claim that, in fact, $\omega_{T\times R_g}
	(x_0,h_1)\supseteq\{x_0\}\times G$ for every $h_1\in G$. To see this, fix $h_1\in G$. Since $\Id_X\times R_{h_1h_0^{-1}}$ is a topological self-conjugacy of $(X\times G,T\times R_g)$, it commutes with the operator $\omega_{T\times R_g}$ and so we have
	\begin{eqnarray*}
	\omega_{T\times R_g}(x_0,h_1)&=&\omega_{T\times R_g}((\Id_X\times R_{h_1h_0^{-1}})(x_0,h_0))=(\Id_X\times R_{h_1h_0^{-1}})(\omega_{T\times R_g}(x_0,h_0))\\
	&\supseteq&(\Id_X\times R_{h_1h_0^{-1}})(\{x_0\}\times G)=\{x_0\}\times G.
	\end{eqnarray*}

	To show that $T\times R_g$ is minimal, fix $x\in X$ and $h\in G$. Since $x_0\in\omega_T(x)$ and
	$G$ is compact, there is
	$h_1\in G$ such that $(x_0,h_1)\in\omega_{T\times R_g}(x,h)$.
	Therefore
	\begin{equation*}
	\omega_{T\times R_g}(x,h)\supseteq\omega_{T\times R_g}(x_0,h_1)
	\supseteq \{x_0\}\times G,
	\end{equation*}
	whence it follows that
	\begin{equation*}
	\omega_{T\times R_g}(x,h)\supseteq\omega_{T\times R_g}(\{x_0\}\times G)=X\times G,
	\end{equation*}
	the last equality being secured by minimality of $T$ and surjectivity of $R_g$. This shows that $T\times R_g$ is minimal.
	
	\emph{2nd step.} We show that $Q$ is of type $G_{\delta}$.
	
	To this end, fix any $(x_0,h_0)\in X\times G$. It follows from \eqref{Eq:Q.expr.1} and \eqref{Eq:Q.expr.2} that
	\begin{equation}\label{Eq:Q.expr.3}
	Q=\{g\in G\colon \omega_{T\times R_g}(x_0,h_0)=X\times G\}.
	\end{equation}
	Since $X$ is metrizable and admits
	a minimal map, it has a countable basis $\{U_n\,:\,n\in\mathbb N\}$.
	Fix also a countable basis $\{V_m\,:\,m\in\mathbb N\}$ of $G$. For $n,m,l\in\mathbb N$, set
	\begin{equation*}
	Q_{n,m}^l(x_0,h_0)=\{g\in G\,:\,(T\times R_g)^k(x_0,h_0)\in U_n\times V_m
	\text{ for some }k\geq l\}.
	\end{equation*}
	Then $Q=\bigcap_{n,m,l=1}^{\infty}Q_{n,m}^l(x_0,h_0)$ by virtue of \eqref{Eq:Q.expr.3}. Since all $Q_{n,m}^l(x_0,h_0)$ are obviously open, $Q$ is indeed of type $G_{\delta}$.
	
	\emph{3rd step.} We show that $Q$ has full Haar measure.
	
	Fix $(x_0,h_0)\in X\times G$. By minimality of $T$, there is an increasing sequence $(k_n)_{n=1}^{\infty}$
	of positive integers with $T^{k_n}(x_0)\to x_0$.
	By Lemma~\ref{unif.distr}, the sequence $(g^{k_n})_{n=1}
	^{\infty}$ is uniformly distributed in $G$ for $\mu-$almost every $g\in G$.
    We show that every such element $g$ belongs to $Q$.
	Indeed, for such $g$, the set
	$$
	\{ R_g^{k_n}(h_0)\,:\,n\in\mathbb N\}=\{g^{k_n}h_0\,:\,
	n\in\mathbb N\}
	$$
	is dense in $G$ and hence, using that $T^{k_n}(x_0)\to x_0$, we obtain
	$$
	\omega_{T\times R_g}(x_0,h_0)\supseteq\{x_0\}\times G\,.
	$$
	By virtue of \eqref{Eq:Q.expr.2}, this means that $g\in Q$.
	
	\emph{4th step.} We show that $Q$ is residual.
	
	Since we already know that $Q$ is of type $G_{\delta}$, it remains to show that it is dense. This follows from the fact that $Q$ has full Haar measure, which is well known to have full support. Nevertheless, we wish to present also another proof based on Lemma~\ref{L:mix.group}.
	
	Fix $x_0\in X$ and $h_0\in G$. To get residuality of $Q$, we show that the open sets $Q_{n,m}^l(x_0,h_0)$ are dense. So fix $n,m,l\in\mathbb N$ and a nonempty open set $V\subseteq G$.
	Since the natural action of $\mathbb N^*$ on $G$ is topologically mixing by Lemma~\ref{L:mix.group}, there is $k_0\geq l$ such that $V\cap A_k^{-1}(h_0^{-1}V_m)\neq\emptyset$ for every $k\geq k_0$. By minimality of $T$, we may fix $k\geq k_0$ with $T^k(x_0)\in
	U_n$. Also, choose $g\in V\cap A_k^{-1}(h_0^{-1}V_m)$. Then $(T\times R_g)^k
	(x_0,h_0)=(T^k(x_0),g^kh_0)\in U_n\times h_0^{-1}V_mh_0=U_n\times V_m$. Thus, $g\in V\cap Q_{n,m}^l(x_0,h_0)$, which proves the density of $Q_{n,m}^l(x_0,h_0)$.
\end{proof}

\begin{remark}\label{R:groups.rot.base}
Consider now a special case of Theorem~\ref{T:groups}, when the minimal map $T$ in the base is also a rotation of a
compact abelian group $X$ by the element $x$. Let $G$ be as in Theorem~\ref{T:groups} and fix $g\in G$.
Then we have a (well known) equivalence of the following conditions:\footnote{Conditions (i) and (ii) are equivalent obviously, the equivalence of (ii) and (iii) is well known, and conditions (iii) and (iv) are equivalent by the standard isomorphism between $\widehat{X\times G}$ and $\widehat{X}
\times\widehat{G}$.}
\begin{enumerate}
	\item[(i)] the product $T\times R_g$ is minimal,
	\item[(ii)] $(x,g)$ is a topological generator of the group $X\times G$,
	\item[(iii)] the annihilator
	$(x,g)^{\perp}=\{\chi\in\widehat{X\times G}\colon \chi(x,g)=1\}$ of $(x,g)$ in $\widehat{X\times G}$ vanishes,
	\item[(iv)] if $\delta\in\widehat{X}$ and $\gamma\in\widehat{G}$
	satisfy $\delta(x)=\gamma(g)$, then $\delta,\gamma$ are both trivial
	characters.
\end{enumerate}
\end{remark}

\begin{theorem}\label{skew.dir}
	Let $(X,T)$ be a minimal dynamical system given by a
	continuous map on a metrizable space. Assume that $Z$
	is a compact metrizable space admitting a compact connected
	abelian group $G$ of homeomorphisms whose natural action
	on $Z$ is minimal. Then the product $(X\times Z, T\times g)$
	is minimal for every $g$ from a residual subset $G'\subseteq G$ of full Haar measure on $G$.
\end{theorem}

This is just a reformulation of Theorem~\ref{T:groups}. Indeed, 
since the action of $G$ on $Z$ is transitive in the algebraic sense by compactness of $G$,
it follows that $Z$ is homeomorphic to a quotient group $G/H$ of $G$  
and, in fact, 
the action of $G$ on $Z$ is equivalent to the natural action of $G$ on $G/H$.
Thus, by Theorem~\ref{T:groups} 
(applied to the group $G/H$), there exists a residual subset $Z'\subseteq Z$ of full Haar measure
such that for every $z\in Z'$ the rotation $R_z$ on $Z$ has $T\times R_z$  minimal on $X\times Z$.
So we could now prove the theorem by setting $G'=p^{-1}(Z')$, where $p\colon G\to Z$ is the quotient morphism. (Notice that $G'$ is residual of full Haar measure, since $p$ is open and measure-preserving.) We wish to present also the following elementary proof.

\begin{proof}
	Let $\mu$ be the Haar measure on $G$.
	For $g\in G$ let $R_g$ denote the rotation of $G$ by
	$g$. By Theorem~\ref{T:groups}, the product $T\times R_g$
	is minimal for every $g$
	from a residual subset $G'\subseteq G$ of full Haar measure on $G$. We show that
	for such $g$ the product $T\times g$ is minimal. So fix
	$g\in G'$, $(x,z)\in X\times Z$
	and nonempty open sets $U\subseteq X$, $V\subseteq Z$; we show that there is a positive integer $n$ with
	$(T\times g)^n(x,z)\in U\times V$. Since the natural
	action of $G$ on $Z$ is minimal, there is $h_0\in G$
	with $h_0(z)\in V$. Choose a neighborhood $W_0$ of
	$h_0$ with the property that $h(z)\in V$ for every
	$h\in W_0$. Since $T\times R_g$ is minimal, there is
	a positive integer $n$ such that $(T\times R_g)^n(x,\Id_Z)
	\in U\times W_0$, that is, $T^n(x)\in U$ and $g^n\in W_0$.
	Consequently, $g^n(z)\in V$ and hence $(T\times g)^n(x,z)
	\in U\times V$.
\end{proof}

We wish to mention that an analogue of Theorem~\ref{T:groups} does not hold for disconnected compact abelian groups $G$; see Example~\ref{E:disc.groups} below. It follows, in particular, that to verify the homeo-product-minimality of the Cantor set $C$ (a fact that we shall prove in Section~\ref{S:CtrCtrd}), it is not sufficient to know that $C$ carries the structure of a compact abelian group and use the corresponding minimal rotations of $C$.

\begin{example}\label{E:disc.groups}
There exist a compact metrizable space $X$ and a minimal homeomorphism $T\colon X\to X$ such that for every compact disconnected abelian group $G$ and every rotation $R$ of $G$, the product $(X\times G,T\times R)$ is not minimal.
\end{example}
\begin{proof}
Let $\tor(\mathbb S^1)$ be the torsion subgroup of $\mathbb S^1$ equipped with the discrete topology and let $X=\widehat{\tor(\mathbb S^1)}$ be the dual group of $\tor(\mathbb S^1)$. Being a dual group of a discrete countable abelian group, the group $X$ is compact, metrizable and abelian. Moreover, since $\widehat{X}\cong\tor(\mathbb S^1)$ is isomorphic to a subgroup of $\mathbb T^1$, $X$ is monothetic \cite[Theorem~6]{AnzKak}. Fix a topological generator $x_0$ of $X$ and let $T$ be the (minimal) rotation of $X$ by $x_0$.

Now let $G$ be a compact disconnected abelian group, $g_0\in G$ and $R$ be the rotation of $G$ by $g_0$; we show that the product $(X\times G,T\times R)$ is not minimal. Denote by $G_0$ the identity component of $G$ and by $p$ the canonical quotient morphism $G\to G/G_0$. Then $Y=G/G_0$ is a nontrivial totally disconnected compact abelian group. Write $y_0=p(g_0)$ and let $S$ be the rotation of $Y$ by $y_0$. Clearly, the morphism $p$ is a factor map $(G,R)\to(Y,S)$. Thus, to show that the product $(X\times G,T\times R)$ is not minimal, it is sufficient to show that the product $(X\times Y,T\times S)$ is not minimal. In view of Remark~\ref{R:groups.rot.base}, it is sufficient to find characters $\delta\in \widehat{X}$ and $\gamma\in \widehat{Y}$ with $\delta(x_0)=\gamma(y_0)\neq1$.

Since the characters of $Y$ separate points, there is $\gamma\in\widehat{Y}$ with $\gamma(y_0)\neq1$. Being a dual group of a totally disconnected group $Y$, the group $\widehat{Y}$ is a torsion group. Consequently, the image $\im(\gamma)$ of $\gamma$ is a finite subgroup of $\mathbb S^1$, that is, $\im(\gamma)=\mathbb Z_k\subseteq\mathbb S^1$ for some integer $k\geq2$. Now, since $\mathbb Z_k$ is a subgroup of $\tor(\mathbb S^1)\cong\widehat{X}$, the group $X$ factors onto $\widehat{\mathbb Z_k}\cong\mathbb Z_k$. Thus, there is a character $\eta\in\widehat{X}$ with $\im(\eta)=\mathbb Z_k$. The set $U=\eta^{-1}(\gamma(y_0))$ is nonempty and open in $X$, hence $x_0^n\in U$ for some $n\in\mathbb N$. Thus, $\eta^n(x_0)=\eta(x_0^n)=\chi(y_0)$ and we may finish the proof by letting $\delta=\eta^n$.
\end{proof}


\section{Products of product-minimal spaces with suitable spaces}\label{S:new.from.old}

The purpose of this section is to prove the following theorem. Its proof is based on ideas from \cite{GW} and \cite{Di}.

\begin{theorem}\label{T:prod.PM.arc}
Let $Z$ be a compact metrizable space admitting a minimal action of an arc-wise connected topological group $G_a$. Then for every nondegenerate (homeo-)product-minimal space $Y$, the space $Y\times Z$ is also (homeo-)product-minimal.
\end{theorem}
\begin{remark}\label{R:prod.PM.arc}
To prove Theorem~\ref{T:prod.PM.arc}, we may additionally assume that $G_a$ is a topological subgroup of the group $\Homeo(Z)$ equipped with the topology of uniform convergence.\footnote{Indeed, let $\Phi\colon G_a\times Z\to Z$ be the considered action, with the acting homeomorphisms  $\varphi_g=\Phi(g,\cdot)$.
Consider the group $H_a\subseteq \Homeo(Z)$ of all acting homeomorphisms and the map $\Psi\colon H_a\times Z\to Z$,
$(h,z)\mapsto h(z)$. 
Since $Z$ is a compact metrizable space and $\Homeo(Z)$ is equipped with the topology of uniform convergence, 
the map $g\mapsto\varphi_g$ is continuous. Therefore,  since $G_a$ is arc-wise connected, so is $H_a$. 
Moreover, $\Psi$ is continuous due to the topology of uniform convergence on $H_a$ and, clearly, it is an action of $H_a$ on $Z$. 
The orbits of $\Psi$ are the same as those of $\Phi$, therefore $\Psi$ is a minimal action.}
By letting $G$ be the closure of $G_a$ in $\Homeo(Z)$, we see that $G_a$ is a dense arc-wise connected subgroup of a Polish group $G$.\footnote{Recall that if $\varrho$ is the supremum metric on $\Homeo(Z)$ then the metric $\varrho'$, given by the rule $\varrho'(g,h)=\varrho(g,h)+\varrho(g^{-1},h^{-1})$, is complete and equivalent to $\varrho$ on $\Homeo(Z)$.} Clearly, the natural action of $G$ on $Z$ is minimal.

Let us mention that Theorem~\ref{T:prod.PM.arc} applies, for instance, to the following spaces $Z$:
\begin{itemize}
\item compact connected manifolds without boundary (see \cite{GW}),
\item compact connected Hilbert cube manifolds (see \cite{GW}),
\item homogeneous spaces of compact connected groups (see \cite{DM}),
\item countably infinite products of compact connected manifolds, infinitely many of which have nonempty boundary (see \cite{DM}).
\end{itemize}
We also recall from \cite{DM} that the class of spaces $Z$ satisfying the assumptions of Theorem~\ref{T:prod.PM.arc} is closed with respect to at most countable products.

It can be seen from the above examples that the space $Z$ itself need not be minimal. Further notice that in fact we assume that the space $Y$ is infinite, since, as we already know from Proposition~\ref{P:PM.con.comp}, a PM space is either infinite or a singleton.
\end{remark}

\begin{corollary}\label{C:prod.PM.arc}
Let $Z$ be a compact metrizable space admitting a minimal homeomorphism isotopic to the identity. Then for every nondegenerate (homeo-)product-minimal space $Y$, the space $Y\times Z$ is also (homeo-)product-minimal.
\end{corollary}
\begin{proof}
Let $S$ be a minimal homeomorphism on $Z$ isotopic to the identity via isotopy $S_t$ ($0\leq t\leq 1$) and let $G_a$ be the subgroup of $\Homeo(Z)$ generated by the set $\{S_t\colon 0\leq t\leq1\}$. Then $G_a$ is an arc-wise connected topological group, whose natural action on $Z$ is minimal. Thus, the statement of the corollary follows from Theorem~\ref{T:prod.PM.arc}.
\end{proof}

\begin{example}\label{e:M-not-PM}
  In Corollary~\ref{C:prod.PM.arc} it is essential that the minimal homeomorphism on $Z$
  is isotopic to the identity; it is not sufficient to assume that $Z$ admits an arbitrary
   minimal homeomorphism. Indeed, let $Y$ be the circle and $Z$ be a DST space.
   By Theorem~\ref{T:groups}, $Y$ is an HPM space but the product $Y\times Z$ is not even a PM space. 
   In fact, $Z\times (Y\times Z)$  
   does not admit a minimal map by Theorem~\ref{T:XxXxY}.
\end{example}

Now we turn to a proof of Theorem~\ref{T:prod.PM.arc}. Let us begin by fixing some necessary notation, which we keep throughout this section. Let $X,Y,Z$ be metrizable spaces, $Y$ infinite, $Y$ and $Z$ compact, and let $T\colon X\to X$, $S\colon Y\to Y$ be continuous maps. We shall assume that the product $(X\times Y,T\times S)$ is minimal. Further, assume that $G$ is a Polish group with a dense arc-wise connected subgroup $G_a$ and suppose that $\phi=(\varphi_g)_{g\in G}$ is a minimal $G$-flow on $Z$. Denote by $e$ the neutral element of $G$.
Given a continuous map $f\colon Y\to G$, we consider the continuous maps
\begin{equation}\label{Eq:def.R}
R\colon X\times Y\times G \to X\times Y\times G, \qquad (x,y,g)\mapsto(T(x),S(y),f(y)g)
\end{equation}
and
\begin{equation}\label{Eq:def.Q}
Q\colon X\times Y\times Z\to  X\times Y\times Z, \qquad (x,y,z)\mapsto(T(x),S(y),\varphi_{f(y)}(z)).
\end{equation}
For $n\in\mathbb N$ and $y\in Y$ write
\begin{equation*}
f^{(n)}(y)=f(S^{n-1}(y))f(S^{n-2}(y))\dots f(S(y))f(y).
\end{equation*} Then
\begin{equation*}
R^n(x,y,g)=(T^n(x),S^n(y),f^{(n)}(y)g)
\end{equation*}
and
\begin{equation*}
Q^n(x,y,z)=(T^n(x),S^n(y),\varphi_{f^{(n)}(y)}(z))
\end{equation*}
for all $x\in X$, $y\in Y$, $g\in G$, $z\in Z$ and $n\in\mathbb N$.
Notice that the system $(X\times Y\times Z, Q)$ is in fact a direct product of $(X,T)$ with a (skew product) system on $Y\times Z$. The map $f$ is traditionally referred to as a \emph{cocycle} over the product system $(X\times Y,T\times S)$, but we require that it depend only on $y\in Y$. We recall the cocycle identity
\begin{equation*}
f^{(k+n)}(y)=f^{(n)}(S^k(y))f^{(k)}(y),
\end{equation*}
which holds for all $y\in Y$ and $k,n\in\mathbb N$.

We call a cocycle $f$ a \emph{coboundary} if there exists a continuous map $\xi\colon Y\to G$, called a \emph{transfer function} for $f$, such that $f(y)=\xi(S(y))\xi(y)^{-1}$ for every $y\in Y$; in this case we write $f=\cob(\xi)$, so
$$
  \cob(\xi)(y) = \xi(S(y))\xi(y)^{-1}.
$$
(Thus, a transfer function is again required to depend only on $y\in Y$.) We shall use the symbol $\overline{\Cob}$ to denote the closure of the set of all coboundaries in the space of all cocycles $f\colon Y\to G$ equipped with the topology of uniform convergence. Thus $\overline{\Cob}$ becomes a completely metrizable space under the supremum metric
\begin{equation*}
d_{\sup}(f,f')=\sup_{y\in Y}d(f(y),f'(y)),
\end{equation*}
where $d$ is a complete metric for the topology of $G$. Notice that for $f=\cob(\xi)$ we have
\begin{equation*}
f^{(n)}(y)=\cob(\xi)^{(n)}(y)=\xi(S^n(y))\xi(y)^{-1}
\end{equation*}
for all $y\in Y$ and $n\in\mathbb N$.

\begin{lemma}\label{L:approx.lemma}
Fix a continuous map $\xi\colon Y\to G$, $\varepsilon>0$, $(x_0,y_0)\in X\times Y$, a pair of nonempty open sets $W\subseteq X\times Y$, $\mathcal N\subseteq G$ and a positive integer $k$ with $(T^k(x_0),S^k(y_0))\in W$. Then there exist a continuous map $\vartheta\colon Y\to G$ and a positive integer $n$ such that
\begin{enumerate}
\item $(T^{k+n}(x_0),S^{k+n}(y_0))\in W$,
\item $\cob(\xi\vartheta)^{(n)}(S^k(y_0))\in\mathcal N$,
\item $d_{\sup}(\cob(\xi\vartheta),\cob(\xi))<\varepsilon$.
\end{enumerate}
\end{lemma}
\begin{proof}
Consider the map
\begin{equation*}
\Phi\colon Y\times G\to \mathbb R, \qquad (y,g)\mapsto d\left(\xi(S(y))g\xi(y)^{-1},\xi(S(y))\xi(y)^{-1}\right),
\end{equation*}
where $d$ is the metric on $G$. Since $\Phi$ is continuous and takes the value $0$ on the compact set $Y\times\{e\}$, by the tube lemma there is a neighborhood $\mathcal N_e$ of $e$ in $G$ such that $\Phi(y,g)<\varepsilon$ for all $y\in Y$ and $g\in\mathcal N_e$.

Set $g_0=\xi(S^k(y_0))$. Since $G_a$ is a dense subgroup of $G$, we may choose $g\in (g_0^{-1}\mathcal Ng_0)\cap G_a$. Fix a neighborhood $\mathcal N_0$ of $g_0$ in $G$ with $\mathcal N_0g\mathcal N_0^{-1}\subseteq\mathcal N$. 
By the assumption, $S^k(y_0)\in \pr_2(W)$ and this point is mapped by $\xi$ to $g_0\in\mathcal N_0$.
By continuity of $\xi$, we may assume that $W$ is small enough so that $\xi(y)\in\mathcal N_0$ for every $y\in\pr_2(W)$.

Now, since $G_a$ is arc-wise connected and $e,g\in G_a$, there is a path $\eta\colon[0,1]\to G_a$ with $\eta(0)=e$ and $\eta(1)=g$. By compactness of $[0,1]$, the map $\eta$ is uniformly continuous and so there is $\delta>0$ such that $\eta(s)\eta(t)^{-1}\in\mathcal N_e$ for all $s,t\in[0,1]$ with $|s-t|<\delta$. Fix a positive integer $N>1/\delta$. By minimality of $T\times S$, there is $n\geq N$ such that $(T^{k+n}(x_0),S^{k+n}(y_0))\in W$. Since the space $X\times Y$ is infinite by our assumptions, all the orbits of the minimal map $T\times S$ are infinite. Thus, there is a continuous map $\kappa\colon Y\to[0,1]$ with $\kappa(S^{k+i}(y_0))=0$ and $\kappa(S^{k+n+i}(y_0))=1$ for $i=0,1,\dots,N-1$.

Consider the map $\varrho\colon Y\to [0,1]$ defined by
\begin{equation*}
\varrho=\frac{1}{N}\sum_{i=0}^{N-1}\kappa\circ S^i
\end{equation*}
and put $\vartheta=\eta\circ\varrho$. We show that the map $\vartheta\colon Y\to G_a\subseteq G$ satisfies the conditions from the lemma.

First, since $\varrho(S^k(y_0))=0$ and $\varrho(S^{k+n}(y_0))=1$, we get $\vartheta(S^k(y_0))=\eta(0)=e$ and $\vartheta(S^{k+n}(y_0))=\eta(1)=g$. Consequently,
\begin{align*}
\cob(\xi\vartheta)^{(n)}(S^k(y_0))&=\xi(S^{k+n}(y_0))\vartheta(S^{k+n}(y_0))\vartheta(S^k(y_0))^{-1}\xi(S^k(y_0))^{-1}\\
&=\xi(S^{k+n}(y_0))g\xi(S^k(y_0))^{-1}\in\mathcal N_0g\mathcal N_0^{-1}\subseteq\mathcal N,
\end{align*}
which verifies condition (2).

Second, since $|\varrho(S(y))-\varrho(y)| = (1/N) \left|\kappa(S^N(y)) - \kappa(y)\right| \leq 1/N<\delta$ for every $y\in Y$ by definition of $\varrho$, we get 
\begin{equation*}
\cob(\vartheta)(y)=\vartheta(S(y))\vartheta(y)^{-1}
=\eta(\varrho(S(y))\eta(\varrho(y))^{-1}\in\mathcal N_e
\end{equation*}
for every $y\in Y$ by our choice of $\delta$. Thus, by definition of $\mathcal N_e$, 
\begin{equation*}
d(\cob(\xi\vartheta)(y),\cob(\xi)(y))=\Phi(y,\cob(\vartheta)(y))\in\Phi(\{y\}\times\mathcal N_e)\subseteq[0,\varepsilon)
\end{equation*}
for every $y\in Y$, which verifies condition (3). Finally, condition (1) holds by our choice of $k$ and $n$ and the proof of the lemma is thus finished.
\end{proof}

\begin{lemma}\label{L:trans.coc.exist}
Given $(x_0,y_0)\in X\times Y$, there is a cocycle $f\colon Y\to G$ such that the system $(X\times Y\times G,R)$ defined by \eqref{Eq:def.R} is point-transitive with a transitive point $(x_0,y_0,e)$.
\end{lemma}
\begin{proof}
The space $X\times Y$ is separable because $T\times S$ is minimal, and $G$ is separable by assumption. So we may fix bases $(W_p)_{p\in\mathbb N}$ and $(\mathcal N_q)_{q\in\mathbb N}$ for $X\times Y$ and $G$, respectively. 
For every $p\in\mathbb N$, fix a positive integer $k_p$ with $(T^{k_p}(x_0),S^{k_p}(y_0))\in W_p$. Further, given $p,q\in\mathbb N$, let $\mathcal C_{p,q}$ denote the set of all cocycles $f\in\overline{\Cob}$, such that there is a positive integer $n$ with $(T^{k_p+n}(x_0),S^{k_p+n}(y_0))\in W_p$ and $f^{(n)}(S^{k_p}(y_0))\in\mathcal N_q$. By virtue of Lemma~\ref{L:approx.lemma}, all the sets $\mathcal C_{p,q}$ are dense in $\overline{\Cob}$ and they are also clearly open in $\overline{\Cob}$. Consequently, the intersection $\bigcap_{p,q=1}^{\infty}\mathcal C_{p,q}$ is a residual subset of (the Polish space) $\overline{\Cob}$. Thus, to finish the proof, we need only to show that for every $f\in\bigcap_{p,q=1}^{\infty}\mathcal C_{p,q}$, the corresponding map $R$ defined by \eqref{Eq:def.R} is point-transitive with a transitive point $(x_0,y_0,e)$.

So fix $f\in\bigcap_{p,q=1}^{\infty}\mathcal C_{p,q}$ and nonempty open sets $W\subseteq X\times Y$, $\mathcal N\subseteq G$. Choose $p,q\in\mathbb N$ so that $W_p\subseteq W$ and $\mathcal N_q\subseteq\mathcal N(f^{(k_p)}(y_0))^{-1}$. Since $f\in\mathcal C_{p,q}$, there is a positive integer $n$ with $(T^{k_p+n}(x_0),S^{k_p+n}(y_0))\in W_p$ and $f^{(n)}(S^{k_p}(y_0))\in\mathcal N_q$. Consequently, by the cocycle identity,
\begin{align*}
R^{k_p+n}(x_0,y_0,e)&=(T^{k_p+n}(x_0),S^{k_p+n}(y_0),f^{(k_p+n)}(y_0))\\
&=(T^{k_p+n}(x_0),S^{k_p+n}(y_0),f^{(n)}(S^{k_p}(y_0))f^{(k_p)}(y_0))\\
&\in W_p\times (\mathcal N_qf^{(k_p)}(y_0))\subseteq W\times\mathcal N.
\end{align*}
This shows that the point $(x_0,y_0,e)$ is transitive for $R$.
\end{proof}

\begin{lemma}\label{L:trans.coc.min}
Let $f\colon Y\to G$ be a cocycle. If the system $(X\times Y\times G,R)$ defined by \eqref{Eq:def.R} is point-transitive, then the system $(X\times Y\times Z,Q)$ defined by \eqref{Eq:def.Q} is minimal.
\end{lemma}
\begin{proof}
The vertical right translations of $X\times Y\times G$ by elements of $G$ are self-conjugacies of $R$ and they yield an (algebraically) transitive action of $G$ on the fibers $\{(x,y)\}\times G$ ($x\in X$, $y\in Y$). Thus, the map $R$ possesses a transitive point of the form $(x,y,e)$.

We show that the map $Q$ is minimal. Since the system $(X\times Y\times Z,Q)$ is a skew product over the minimal system $(X\times Y,T\times S)$ and the fiber $Z$ is compact, by Lemma~\ref{L:fiber} it suffices to show that the point $(x,y,z)$ is transitive for $Q$ for every $z\in Z$. So fix $z\in Z$ and nonempty open sets $W\subseteq X\times Y$, $\mathcal O\subseteq Z$. Since $\phi=(\varphi_g)_{g\in G}$ is a minimal $G$-flow on $Z$, there is a nonempty open set $\mathcal N\subseteq G$ with $\varphi_g(z)\in\mathcal O$ for every $g\in\mathcal N$. By transitivity of the point $(x,y,e)$ for $R$, there is $n\in\mathbb N$ with $R^n(x,y,e)\in W\times\mathcal N$. That is, $(T^n(x),S^n(y))\in W$ and $f^{(n)}(y)\in\mathcal N$. Hence $\varphi_{f^{(n)}(y)}(z)\in\mathcal O$, whence it follows that
\begin{equation*}
Q^n(x,y,z)=(T^n(x),S^n(y),\varphi_{f^{(n)}(y)}(z))\in W\times\mathcal O.
\end{equation*}
This shows that the point $(x,y,z)$ is transitive for $Q$.
\end{proof}

\begin{proof}[Proof of Theorem~\ref{T:prod.PM.arc}]
Let $Y$ be a (homeo-)product-minimal space. To show that the product $Y\times Z$ is also (homeo-)product-minimal, fix a metrizable space $X$ and a minimal continuous map $T\colon X\to X$. By our assumptions, the space $Y$ admits a minimal continuous map (a minimal homeomorphism) $S\colon Y\to Y$ such that the product $(X\times Y,T\times S)$ is minimal. By virtue of Lemma~\ref{L:trans.coc.exist}, there is a cocycle $f\colon Y\to G$ such that the corresponding system $(X\times Y\times G,R)$ is point-transitive. By Lemma~\ref{L:trans.coc.min} it follows that the system $(X\times Y\times Z,Q)$ is minimal. The theorem now follows from the obvious fact that the system $(X\times Y\times Z,Q)$ is a direct product of $(X,T)$ with a (skew product) system on $Y\times Z$ and the latter is a homeomorphism if $S$ is a homeomorphism.
\end{proof}


\section{Quotient spaces of products and the Klein bottle}

Our purpose in this section is to prove Theorem~\ref{T:fact.space} below, which is based on \cite[Theorem~11]{Di}. Before formulating it, we introduce some notation. Let $Z$ be a compact metrizable space and $G_a$ be a subgroup of $\Homeo(Z)$. We shall use the symbol $N(G_a)$ to denote the \emph{normalizer} of $G_a$ in $\Homeo(Z)$; recall that $h\in\Homeo(Z)$ belongs to $N(G_a)$ if and only if for every $g\in G_a$, $hgh^{-1}\in G_a$. We have obvious inclusions of groups $G_a\subseteq N(G_a)\subseteq\Homeo(Z)$.

\begin{theorem}\label{T:fact.space}
Let $\Gamma$ be an infinite compact connected metrizable abelian group and $\Lambda$ be a finite subgroup of $\Gamma$. Let $Z$ be a compact metrizable space and $G_a$ be an arc-wise connected subgroup of $\Homeo(Z)$ with a minimal natural action on $Z$. If $q\colon\Lambda\to N(G_a)$ is a morphism of groups then the orbit space $(\Gamma\times Z)/\Lambda$, obtained from $\Gamma\times Z$ by applying the diagonal action of $\Lambda$, is homeo-product-minimal.
\end{theorem}
\begin{remark}
Let us recall that the diagonal action of $\Lambda$ on $\Gamma\times Z$ is by means of the product homeomorphisms $R_{\lambda}\times q(\lambda)$, where $R_{\lambda}$ denotes the rotation of $\Gamma$ by $\lambda\in\Lambda$. Further, the space $Z$ is automatically connected, for it admits a minimal action of a connected group $G_a$. Finally, recall that the orbit space $(\Gamma\times Z)/\Lambda$ is a compact connected metrizable space, since the diagonal action of (the finite group) $\Lambda$ on $\Gamma\times Z$ is fixed-point free (hence also properly discontinuous), see \cite[p.~494]{Mun}.
\end{remark}

Before turning to the proof of Theorem~\ref{T:fact.space}, we discuss some corollaries of it.

\begin{theorem}\label{T:fact.groups}
Let $\Gamma$ be an infinite compact connected metrizable abelian group and $\Lambda$ be a finite subgroup of $\Gamma$. Let $Z$ be a compact connected (not necessarily abelian) metrizable group, on which the group $\Lambda$ acts by automorphisms. Then the orbit space $(\Gamma\times Z)/\Lambda$, obtained from $\Gamma\times Z$ by applying the diagonal action of $\Lambda$, is homeo-product-minimal.
\end{theorem}
\begin{remark}
In this case the space $(\Gamma\times Z)/\Lambda$ is the orbit space of $\Gamma\times Z$ subject to the group of homeomorphism $R_{\lambda}\times A_{\lambda}$ ($\lambda\in\Lambda$), where $A_{\lambda}$ is the acting automorphism of $Z$ corresponding to $\lambda\in\Lambda$.
\end{remark}
\begin{proof}[Proof of Theorem~\ref{T:fact.groups}]
Let $G$ be the group of the left rotations on $Z$; recall that $G$ and $Z$ are topologically isomorphic. Let $Z_a$ denote the identity arc-component of $Z$. By \cite[Theorem~9.60(v),~p.~501]{HofMor}, $Z_a$ is dense in $Z$. Consequently, the identity arc-component $G_a$ of $G$, which consists of the left rotations of $Z$ by elements of $Z_a$, is dense in $G$. Thus, since the natural action of $G$ on $Z$ is minimal, it follows that the natural action of $G_a$ on $Z$ is also minimal.

The action of $\Lambda$ on $Z$ by automorphisms $A_{\lambda}$ ($\lambda\in\Lambda$) yields a morphism of groups $q\colon\Lambda\to\Homeo(Z)$. Clearly, if $z\in Z_a$ and $L_z$ is the left rotation of $Z$ by $z$, then for every $\lambda\in\Lambda$, $A_{\lambda}L_zA_{\lambda}^{-1}$ is the left rotation of $Z$ by the element $A_{\lambda}(z)$. Consequently, since $A_{\lambda}(z)\in Z_a$ for every $z\in Z_a$, $q$ takes its values in the normalizer $N(G_a)$ of $G_a$ in $\Homeo(Z)$. Thus, the orbit space $(\Gamma\times Z)/\Lambda$ is homeo-product-minimal by virtue of Theorem~\ref{T:fact.space}.
\end{proof}

As a special case of Theorem~\ref{T:fact.groups}, we get the following result.

\begin{theorem}\label{T:Klein}
Klein bottle $\mathbb K^2$ is a homeo-product-minimal space.
\end{theorem}
\begin{proof}
We keep our notation introduced in (the proof of) Theorem~\ref{T:fact.groups}. Let $\Gamma=Z=\mathbb S^1$ and $\Lambda=\{-1,1\}$. The group $\Lambda$ acts on $\mathbb S^1$ by means of the involution automorphism $A_{-1}\colon\mathbb S^1\to\mathbb S^1$, $A_{-1}(z)=z^{-1}$. Since the quotient space $(\mathbb S^1\times\mathbb S^1)/\Lambda$ is homeomorphic to the Klein bottle (notice that the diagonal action of $\Lambda$ is now by means of the homeomorphism $(R_{-1}\times A_{-1})(\gamma,z)=(-\gamma,z^{-1})$), the latter is a homeo-product-minimal space by virtue of Theorem~\ref{T:fact.groups}.
\end{proof}

Notice that this implies that $\mathbb K^2$ admits a minimal homeomorphism, which is however well known.

\medskip

Now we turn to a proof of Theorem~\ref{T:fact.space}. Throughout the rest of this section, we shall assume that the assumptions of Theorem~\ref{T:fact.space} are fulfilled. We let $\alpha$ be a topological generator of $\Gamma$ (notice that $\Gamma$ is monothetic) and write $R_{\alpha}$ for the (minimal) rotation of $\Gamma$ by $\alpha$.
We assume that $X$ is a metrizable space and $T\colon X\to X$ is a continuous map such that the product $(X\times\Gamma,T\times R_{\alpha})$ is minimal. We also keep the notation and terminology introduced in Section~\ref{S:new.from.old} and emphasize that the role of $(Y,S)$ from Section~\ref{S:new.from.old} is now played by $(\Gamma,R_{\alpha})$. A continuous map $f\colon\Gamma\to G_a$ will be called \emph{invariant}, if 
\begin{equation}
f(\lambda\gamma)=q(\lambda)f(\gamma)q(\lambda)^{-1}
\end{equation}
for all $\gamma\in\Gamma$ and $\lambda\in\Lambda$. Notice that a (point-wise) product of invariant maps is invariant. Also, a coboundary $\cob(\xi)$ with a transfer function $\xi\colon\Gamma\to G_a$ is invariant if the map $\xi$ is invariant.

\begin{lemma}\label{L:approx.lemma.inv}
Fix a continuous invariant map $\xi\colon\Gamma\to G_a$, $\varepsilon>0$, $(x_0,\gamma_0)\in X\times\Gamma$, a pair of nonempty open sets $W\subseteq X\times\Gamma$, $\mathcal N\subseteq G_a$ and a positive integer $k$ with $(T^k(x_0),R_{\alpha}^k(\gamma_0))\in W$. Then there exist a continuous invariant map $\vartheta\colon\Gamma\to G_a$ and a positive integer $n$ such that
\begin{enumerate}
\item $(T^{k+n}(x_0),R_{\alpha}^{k+n}(\gamma_0))\in W$,
\item $\cob(\xi\vartheta)^{(n)}(R_{\alpha}^k(\gamma_0))\in\mathcal N$,
\item $d_{\sup}(\cob(\xi\vartheta),\cob(\xi))<\varepsilon$.
\end{enumerate}
\end{lemma}
\begin{proof}
Let us begin, similarly as in the proof of Lemma~\ref{L:approx.lemma}, by choosing a neighborhood $\mathcal N_e$ of $e$ in $G_a$ with
\begin{equation*}
d(\xi(\alpha\gamma)g\xi(\gamma)^{-1},\xi(\alpha\gamma)\xi(\gamma)^{-1})<\varepsilon
\end{equation*}
for all $g\in\mathcal N_e$ and $\gamma\in\Gamma$. Set $\mathcal N_e^{inv}=\bigcap_{\lambda\in\Lambda}q(\lambda)\mathcal N_eq(\lambda)^{-1}$. Since $q$ takes its values in the normalizer $N(G_a)$ of $G_a$ in $\Homeo(Z)$, $\mathcal N_e^{inv}$ is an identity neighborhood in $G_a$. Moreover, $\mathcal N_e^{inv}$ is contained in $\mathcal N_e$ and is invariant with respect to the conjugation by the elements $q(\lambda)$ ($\lambda\in\Lambda$) (that is, we have $q(\lambda)\mathcal N_e^{inv}q(\lambda)^{-1}=\mathcal N_e^{inv}$ for every $\lambda\in\Lambda$).

Fix an arbitrary injective function $\Lambda\ni\lambda\mapsto z_{\lambda}\in\mathbb S^1$ mapping the identity of $\Lambda$ to the identity of $\mathbb S^1$. For each $\lambda\in\Lambda$ set $I_{\lambda}=\{tz_{\lambda}\colon t\in[0,1]\}\subseteq\mathbb C$ and write $S=\bigcup_{\lambda\in\Lambda}I_{\lambda}$. We shall view $S$ as a subspace of $\mathbb C$ equipped with the  standard (that is, euclidean) metric.

Fix $g\in G_a$ and a neighborhood $\mathcal N_0$ of $\xi(\alpha^k\gamma_0)$ in $G_a$ so that $\mathcal N_0g\mathcal N_0^{-1}\subseteq\mathcal N$. Without loss of generality, we may assume that $\xi(\gamma)\in\mathcal N_0$ for every $\gamma\in\Pr_2(W)$. Now choose a path $\varrho\colon[0,1]\to G_a$ with $\varrho(0)=e$ and $\varrho(1)=g$. Given $t\in[0,1]$ and $\lambda\in\Lambda$, put $\eta(tz_{\lambda})=q(\lambda)\varrho(t)q(\lambda)^{-1}$. The map $\eta\colon S\to G_a$ thus defined is continuous and, having a compact domain, it is uniformly continuous. Consequently, there is $\delta>0$ such that $\eta(z)\eta(z')^{-1}\in\mathcal N_e^{inv}$ for all $z,z'\in S$ with $|z-z'|<\delta$. Fix a positive integer $N>1/\delta$. Since the map $T\times R_{\alpha}$ is minimal, we may fix an integer $n\geq N$ with $(T^{k+n}(x_0),R_{\alpha}^{k+n}(\gamma_0))\in W$ in order to fulfill condition (1).

Since the group $\Gamma$ is infinite and $\alpha$ is its topological generator, all the powers $\alpha^i$ ($i\in\mathbb Z$) are distinct. Consequently, all the orbits of the map $R_{\alpha}$ are infinite. Further, since $\Lambda$ is a finite subgroup of $\Gamma$, the points $\lambda\alpha^i\gamma_0$ ($\lambda\in\Lambda$, $i\in\mathbb Z$) are mutually distinct. Thus, there is a closed neighborhood $B$ of $\alpha^k\gamma_0$ such that the sets  $\lambda\alpha^{ i}B$ ($\lambda\in\Lambda$, $i=0,\dots,n+N$) are mutually disjoint.

Now fix a continuous map $\sigma\colon B\to[0,1/N]$, which takes value $1/N$ at $\alpha^k\gamma_0$ and value $0$ on the 
boundary $\Bd(B)$ of $B$. Define a map $\kappa\colon\Gamma\to S$ as follows:
\begin{align*}
\kappa(\gamma) =
\begin{cases}
 l_i\sigma(\lambda^{-1}\alpha^{-i}\gamma)z_{\lambda}; & \gamma\in\lambda\alpha^iB\text{ and }i=0,\dots,n+N, \\
 0; & \text{otherwise},
\end{cases}
\end{align*}
where
\begin{align*}
l_i=
\begin{cases}
 i; & i=0,\dots,N, \\
 N; & i=N+1,\dots,n, \\
 n+N-i; & i=n+1,\dots,n+N. 
\end{cases}
\end{align*}
Since the map $\sigma$ is continuous and takes the value $0$ on the boundary of $B$, it follows that $\kappa$ is continuous. 

Now set $\vartheta=\eta\circ\kappa$. Then $\vartheta\colon\Gamma\to G_a$ is a continuous map. Since $|\kappa(\alpha\gamma)-\kappa(\gamma)|\leq1/N<\delta$ for every $\gamma\in\Gamma$, it follows from our choice of $\delta$ that $\vartheta(\alpha\gamma)\vartheta(\gamma)^{-1}=\eta(\kappa(\alpha\gamma))\eta(\kappa(\gamma))^{-1}\in\mathcal N_e^{inv}\subseteq\mathcal N_e$ for every $\gamma\in\Gamma$. Consequently,
\begin{equation*}
d(\cob(\xi\vartheta)(\gamma),\cob(\xi)(\gamma))=
d(\xi(\alpha\gamma)(\vartheta(\alpha\gamma)\vartheta(\gamma)^{-1})\xi(\gamma)^{-1},\xi(\alpha\gamma)\xi(\gamma)^{-1})<\varepsilon
\end{equation*}
for every $\gamma\in\Gamma$ by our choice of $\mathcal N_e$. This verifies that $\vartheta$ satisfies condition (3).

We show that $\vartheta$ satisfies condition (2). Indeed, since $\vartheta(\alpha^k\gamma_0)=\eta(\kappa(\alpha^k\gamma_0))=\eta(0)=e$ and $\vartheta(\alpha^{k+n}\gamma_0)=
\eta(\kappa(\alpha^{k+n}\gamma_0))=
\eta(N\sigma(\alpha^k\gamma_0))=\eta(1)=g$, we obtain
\begin{align*}
\cob(\xi\vartheta)^{(n)}(R_{\alpha}^k(\gamma_0))&=(\xi\vartheta)(R_{\alpha}^{k+n}(\gamma_0))(\xi\vartheta)(R_{\alpha}^k(\gamma_0))^{-1}\\
&=\xi(\alpha^{k+n}\gamma_0)(\vartheta(\alpha^{k+n}\gamma_0)
\vartheta(\alpha^k\gamma_0)^{-1})\xi(\alpha^k\gamma_0)^{-1}\\
&=\xi(\alpha^{k+n}\gamma_0)(ge^{-1})\xi(\alpha^k\gamma_0)^{-1}\\
&\in\mathcal N_0g\mathcal N_0^{-1}\subseteq\mathcal N
\end{align*}
by our choice of $\mathcal N_0$.

To finish the proof, it remains to verify that $\vartheta$ is an invariant map. To this end, fix $\gamma\in\Gamma$ and $\lambda\in\Lambda$; we need to check that $\vartheta(\lambda\gamma)
=q(\lambda)\vartheta(\gamma)q(\lambda)^{-1}$. It will be convenient to distinguish two cases.

\emph{Case 1.} We have $\gamma\in\lambda_0\alpha^iB$ for some $\lambda_0\in\Lambda$ and $i\in\{0,\dots,n+N\}$. Then, by definition of $\kappa$, 
\begin{equation*}
\kappa(\gamma)=l_i\sigma(\lambda_0^{-1}\alpha^{-i}\gamma)z_{\lambda_0}.
\end{equation*}
Also, $\lambda\gamma\in\lambda\lambda_0\alpha^iB$ and so 
\begin{equation*}
\kappa(\lambda\gamma)
=l_i\sigma((\lambda\lambda_0)^{-1}\alpha^{-i}\lambda\gamma)z_{\lambda\lambda_0}
=l_i\sigma(\lambda_0^{-1}\alpha^{-i}\gamma)z_{\lambda\lambda_0}.
\end{equation*}
Therefore,
\begin{equation*}
\vartheta(\gamma)=\eta(\kappa(\gamma))
=\eta(l_i\sigma(\lambda_0^{-1}\alpha^{-i}\gamma)z_{\lambda_0})
=q(\lambda_0)\varrho(l_i\sigma(\lambda_0^{-1}\alpha^{-i}\gamma))q(\lambda_0)^{-1}
\end{equation*}
and, similarly,
\begin{equation*}
\vartheta(\lambda\gamma)=\eta(\kappa(\lambda\gamma))
=\eta(l_i\sigma(\lambda_0^{-1}\alpha^{-i}\gamma)z_{\lambda\lambda_0})
=q(\lambda\lambda_0)\varrho(l_i\sigma(\lambda_0^{-1}\alpha^{-i}\gamma))q(\lambda\lambda_0)^{-1}.
\end{equation*}
Finally, since $q\colon\Lambda\to N(G_a)$ is a morphism of groups, we obtain the desired equality
\begin{equation*}
\vartheta(\lambda\gamma)
=q(\lambda\lambda_0)q(\lambda_0)^{-1}\vartheta(\gamma)
q(\lambda_0)q(\lambda\lambda_0)^{-1}
=q(\lambda)\vartheta(\gamma)q(\lambda)^{-1}.
\end{equation*}

\emph{Case 2.} The point $\gamma$ lies outside the sets $\lambda_0\alpha^iB$ ($\lambda_0\in\Lambda$, $i=0,\dots,n+N$). Then so does the point $\lambda\gamma$ and so, by definition of $\kappa$, $\kappa(\gamma)=\kappa(\lambda\gamma)=0$. Then $\vartheta(\lambda\gamma)=\vartheta(\gamma)=e$ and the desired equality 
$\vartheta(\lambda\gamma)
=q(\lambda)\vartheta(\gamma)q(\lambda)^{-1}$ is thus immediate.
\end{proof}

Set
\begin{equation*}
\Cob^{inv}=\{\cob(\xi)\colon \xi\colon\Gamma\to G_a\text{ is continuous and invariant}\}.
\end{equation*}
Write $G$ for the closure of $G_a$ in $\Homeo(Z)$ and consider the closure $\overline{\Cob^{inv}}$ of $\Cob^{inv}$ in the space of all cocycles $\Gamma\to G$ with the topology of uniform convergence (that is, with the compact-open topology). Since $G$ is a completely metrizable group, it follows that $\overline{\Cob^{inv}}$ is a completely metrizable space.

\begin{lemma}\label{L:inv.trans.coc}
Given $(x_0,\gamma_0)\in X\times \Gamma$, there is a cocycle $f\in\overline{\Cob^{inv}}$ such that the system $(X\times \Gamma\times G,R)$ defined by \eqref{Eq:def.R} is point-transitive with a transitive point $(x_0,\gamma_0,e)$.
\end{lemma}
\begin{proof}
The lemma follows from Lemma~\ref{L:approx.lemma.inv} above in the same way as Lemma~\ref{L:trans.coc.exist} in Section~\ref{S:new.from.old} and so we omit the proof.
\end{proof}

\begin{proof}[Proof of Theorem~\ref{T:fact.space}]
Fix a minimal system $(X,T)$. By virtue of Theorem~\ref{T:groups}, there is $\alpha\in\Gamma$ such that the product $(X\times\Gamma,T\times R_{\alpha})$ is minimal. By Lemma~\ref{L:inv.trans.coc}, there is an invariant cocycle $f\colon\Gamma\to G$ such that the system $(X\times \Gamma\times G,R)$ defined by \eqref{Eq:def.R} is point-transitive. Consequently, the underlying system $(X\times\Gamma\times Z,Q)$ defined by \eqref{Eq:def.Q} is minimal by virtue of Lemma~\ref{L:trans.coc.min}. Recall that the latter is (topologically conjugate to) a direct product of $(X,T)$ with $(\Gamma\times Z,P)$, where $P$ is a skew product over $R_{\alpha}$ corresponding to $f$; that is
\begin{equation*}
P(\gamma,z)=(R_{\alpha}(\gamma),f(\gamma)(z))
\end{equation*}
for all $\gamma\in\Gamma$ and $z\in Z$.

We claim that $P$ commutes with the diagonal action of $\Lambda$ on $\Gamma\times Z$. To see this, fix $\lambda\in\Lambda$. Then, for $(\gamma,z)\in\Gamma\times Z$,
\begin{equation*}
\big[P\circ(R_{\lambda}\times q(\lambda))\big](\gamma,z)
=P(\lambda\gamma,q(\lambda)(z))=\left(\alpha\lambda\gamma,\big[f(\lambda\gamma)q(\lambda)\big](z)\right)
\end{equation*}
and
\begin{equation*}
\big[(R_{\lambda}\times q(\lambda))\circ P\big](\gamma,z)=(R_{\lambda}\times q(\lambda))(\alpha\gamma,f(\gamma)(z))=\left(\lambda\alpha\gamma,\big[q(\lambda)f(\gamma)\big](z)\right).
\end{equation*}
Now, since $\Gamma$ is abelian, we have $\alpha\lambda=\lambda\alpha$ and, by invariance of $f$, we get $f(\lambda\gamma)q(\lambda)=q(\lambda)f(\gamma)$. Consequently, $P\circ(R_{\lambda}\times q(\lambda))=(R_{\lambda}\times q(\lambda))\circ P$, as was to be shown.

Since $P$ commutes with the action of $\Lambda$ on $\Gamma\times Z$, it factors onto a continuous map $P/\Lambda$ on $(\Gamma\times Z)/\Lambda$. Since $P$ is a homeomorphism, it follows that so is $P/\Lambda$. Finally, since the map $T\times P$ is minimal and factors onto $T\times (P/\Lambda)$, it follows that $T\times (P/\Lambda)$ is also minimal. This shows that the space $(\Gamma\times Z)/\Lambda$ is homeo-product-minimal indeed.
\end{proof}


\section{Cantor space and cantoroids}\label{S:CtrCtrd}

Let $C$ denote the Cantor space. The first purpose of this section is to show that $C$ is a  homeo-product-minimal space. In particular, this answers in affirmative a question proposed to us by J. Kwiatkowski, whether every minimal system $(X,T)$ extends to a minimal skew product on $X\times C$.  

Recall that a map $\pi\colon X\to Y$ is called \emph{almost 1-to-1} if the points $x\in X$ for which the set $\pi^{-1}(\pi(x))$ is a singleton form a dense subset of~$X$. In this situation, if $\pi$ is surjective, we say that the space $X$ is an \emph{almost 1-to-1 extension} of $Y$. Further, if $(X,T)$ and $(Y,S)$ are dynamical systems given by continuous maps on metrizable spaces and $\pi\colon X\to Y$ is a continuous surjection with $\pi \circ T = S\circ \pi$, then $\pi$ is called a \emph{semiconjugacy} or a \emph{factor map}, the system $(Y,S)$ is a \emph{factor} of $(X,T)$ and the system $(X,T)$ is an \emph{extension} of $(Y,S)$. If $\pi$ is almost 1-to-1 then $(X,T)$ is an \emph{almost 1-to-1 extension} of $(Y,S)$.

Every factor of a minimal system is minimal. As for extensions, we have the following lemma.

\begin{lemma}\label{L:minim.ext}
	Let $X,Y$ be metrizable spaces, $(Y,S)$ be a minimal system and $(X,T)$
	be its extension via a semiconjugacy $\pi:X\to Y$.
	Assume that the following two conditions hold:
	\begin{enumerate}
		\item[(1)] the map $\pi$ is almost 1-to-1,
		\item[(2)] the map $\pi$ is closed.
	\end{enumerate}
	Then the system $(X,T)$ is minimal.
\end{lemma}

Let us mention that for systems with compact phase spaces the lemma is well known, see e.g.~\cite[Lemma~19]{BDHSS}. Notice also that in the general case we use that $\pi$ is closed but the proof does not need the fact that the semiconjugacy is continuous.

\begin{proof}
	Fix $x\in X$ and a nonempty open set $V\subseteq X$; we
	show that $T^n(x)\in V$ for some $n\geq1$. By condition (1),
	there is $x'\in V$ such that $\pi^{-1}(\pi(x'))=\{x'\}$. Since
	$\pi$ is a closed map, there exists a neighborhood $U$ of
	$\pi(x')$ in $Y$ (e.g.~$U=Y\setminus\pi(X\setminus V)$) with $\pi^{-1}(U)\subseteq V$. By minimality
	of $(Y,S)$, there is $n\geq1$ such that $S^n(\pi(x))
	\in U$. Therefore, $\pi(T^n(x))=S^n(\pi(x))
	\in U$ and so $T^n(x)\in\pi^{-1}(U)\subseteq V$, as desired.
\end{proof}

\begin{lemma}\label{L:closed}
	Let $X,Y,Z$ be metrizable spaces, $Y$ compact, and let $p\colon Y\to Z$ be a continuous map. Then $\pi=\Id_X\times p\colon X\times Y\to X\times Z$ is a closed map.
\end{lemma}
\begin{proof}
	Let $F\subseteq X\times Y$ be a closed set and let $((x_n,y_n))_{n=1}^{\infty}$ be a sequence in $F$, such that $\pi((x_n,y_n))=(x_n,p(y_n))\to(x,z)$ in $X\times Z$. Then $x_n\to x$ in $X$ and, due to compactness of $Y$, we may assume that $y_n\to y$ in $Y$. Since $F$ is closed, we get $(x,y)\in F$. Then $\pi((x_n,y_n))\to\pi(x,y)$ and so $(x,z)=\pi(x,y)\in\pi(F)$.
\end{proof}

\begin{proposition}\label{P:general}
	Let $Y$ and $Z$ be compact metrizable spaces. Assume that one of the following holds.
	\begin{enumerate}
		\item [(i)]  $Z$ is product-minimal and for every minimal map $h$ on $Z$ there exists a continuous map $S$ on $Y$ such that $(Y,S)$ is an almost 1-to-1 extension of $(Z,h)$.
		\item [(ii)]  $Z$ is homeo-product-minimal and for every minimal homeomorphism $h$ on $Z$ there exists a continuous map $S$ on $Y$ such that $(Y,S)$ is an almost 1-to-1 extension of $(Z,h)$.
	\end{enumerate}
	Then $Y$ is a product-minimal space. If (ii) is true and, moreover, such a map $S$ always exists in the class of homeomorphisms, then $Y$ is even a homeo-product-minimal space.
\end{proposition}

\begin{proof}
	Let $(X,T)$ be a minimal dynamical system given by a continuous map $T$ on a metrizable space $X$. By (i) or (ii) there exists, respectively, a map or a homeomorphism $h$ on $Z$ such that the product $(X\times Z, T\times h)$ is minimal. Clearly, $h$ is minimal. So, there exists a continuous map $S$ on $Y$ such that $(Y,S)$ is an almost 1-to-1 extension of $(Z,h)$ via a closed semiconjugacy $p$.\footnote{Notice that $Y$ is an almost 1-to-1 extension of $Z$ and that, by Lemma~\ref{L:minim.ext}, $(Y,S)$ is a minimal system.}	Since $p$ is almost 1-to-1, so is $\pi=\Id_X\times p$. Clearly, $\pi$ is a semiconjugacy $(X\times Y, T\times S)\to (X\times Z, T\times h)$ and, by Lemma~\ref{L:closed}, it is a closed map. Thus, by virtue of Lemma~\ref{L:minim.ext}, the product $T\times S$ is minimal on $X\times Y$. We have thus proved that $Y$ is a product-minimal space. The rest is obvious.  
\end{proof}

\begin{remark}
	In connection with Proposition~\ref{P:general}, let us mention that an almost 1-to-1 extension of a (homeo-)product-minimal space need not be product-minimal, even if it admits a minimal homeomorphism. To show an example, start with an arbitrary solenoid $Y$. By Theorem~\ref{T:groups}, $Y$ is product-minimal. However, if $X$ is a DST space constructed as an almost 1-to-1 extension of $Y$, see \cite{DST}, then $X$ admits a minimal homeomorphism but, by Proposition~\ref{P:Slovak.not.HPM}(b), $X$ is not product-minimal.
\end{remark}

\begin{theorem}\label{T:Cantor}
The Cantor space is homeo-product-minimal.
\end{theorem}
\begin{proof}
	We use Proposition~\ref{P:general} with $Y$ being the Cantor space and $Z$ being the circle. By Theorem~\ref{T:groups}, $Z$ is homeo-product-minimal.
	Now fix a minimal homeomorphism $h$ on the circle $Z$. We may assume that it is an irrational rotation. Now apply the Denjoy blow-up technique to the rotation $h$ to obtain a minimal homeomorphism $S$ on a Cantor subset $K$ of the circle $Z$ and denote the corresponding semiconjugacy $(K,S)\to (Z,h)$ by $p$. Recall that $p$ is almost 1-to-1. Since $K$ is homeomorphic to the Cantor space $Y$, Proposition~\ref{P:general} gives that $Y$ is homeo-product-minimal.
\end{proof}

Given a compact metrizable space $Y$, consider the set
\[\Ydeg=\left\{y\in Y\colon \{y\} \text{ is a
	component of }Y \right\}
\]
of all degenerate components of $Y$. By \cite{BDHSS}, $Y$ is called a \emph{cantoroid} if it has no isolated points and $\Ydeg$ is dense in $Y$.

\begin{theorem}\label{T:cantoroid}
	Every cantoroid is product-minimal.
\end{theorem}
\begin{proof}
	We use Proposition~\ref{P:general} with $Y$ being a cantoroid and $Z$ being the Cantor space. Then $Z$ is homeo-product-minimal by Theorem~\ref{T:Cantor}. Fix a minimal homeomorphism $h$ on the Cantor space $Z$. By \cite[Theorem~24]{BDHSS}, the cantoroid $Y$ admits a minimal continuous map $S$ which is an almost 1-to-1 extension of $h$. Then Proposition~\ref{P:general} gives that the cantoroid $Y$ is product-minimal.
\end{proof}

\begin{remark}\label{R:PMnotHPM}
	Let us mention that a cantoroid need not be homeo-product-minimal. Say, if a cantoroid has a unique nondegenerate component then it does not admit any minimal homeomorphism whatsoever.
\end{remark}


\section{Sierpi\'nski curves on minimal connected 2-manifolds}\label{S:Sierp}

Among connected 2-manifolds with or without boundary, only the 2-torus and the Klein bottle admit minimal maps, see \cite{BOT}. By our Theorems~\ref{T:groups} and~\ref{T:Klein}, both of them are homeo-product-minimal. We are going to show that also the Sierpi\'nski curves on these 
2-manifolds are homeo-product-minimal. Let us recall definitions of these curves.

Let $M$ be a compact connected 2-manifold without boundary and 
$A \subseteq M$ be a \emph{curve}, i.e., a one-dimensional continuum. 
Then $A$ is said to be an \emph{$S$-curve on $M$} (see~\cite{Bor}) if it is 
locally connected and there exists a sequence $(D_i)_{i=1}^\infty$ of mutually 
disjoint closed discs in $M$ such that $A = M \setminus \bigcup_{i=1}^\infty \Int
D_i$. As $A$ is one-dimensional, $\bigcup_{i=1}^{\infty} \Int
D_i$ is necessarily dense in $M$. As observed in \cite[p.~82]{Bor}, the assumptions
imply that $\diam D_i \to 0$.

On the other hand, if  $(D_i)_{i=1}^\infty$ is a sequence of
mutually disjoint closed discs in $M$ with $\diam D_i \to 0$
and if the set $\bigcup_{i=1}^{\infty} D_i$ is dense in $M$ then
the set $A = M \setminus \bigcup_{i=1}^\infty \Int D_i$ is a locally connected curve
\cite[Lemma~4.1]{Bor},
hence an $S$-curve on $M$.

By~\cite{Why}, any two $S$-curves on the 2-sphere are homeomorphic; they
are in fact homeomorphic to the Sierpi\'nski carpet (Sierpi\'nski plane 
universal curve). By~\cite{Bor}, this can be generalized
to any compact connected 2-manifold without boundary: two $S$-curves $A$
on $M$ and $A'$ on $M'$ are homeomorphic if
and only if $M$ and $M'$ are homeomorphic. Due to the above facts, 
any $S$-curve on $M$ is called \emph{the Sierpi\'nski curve on $M$}.

Let $M$ be the 2-torus or the Klein bottle. To show that the Sierpi\'nski curve 
on $M$ admits a minimal homeomorphism, one can proceed as in~\cite{Nor} or~\cite{BKS}. 
This means that we start with a minimal homeomorphism $h$ on $M$, we choose one full orbit $x_i$, $i\in \mathbb Z$, and we blow up 
all points of this orbit to closed round discs $D_i$, $i\in \mathbb Z$, whose diameters tend to zero. We obtain a new homeomorphism 
on $M$ for which these discs are wandering. The interior of $D_i$ is mapped onto the interior of $D_{i+1}$.
By removing all those interiors we finally obtain the Sierpi\'nski curve $Y$ on $M$
with a minimal homeomorphism $S$ on $Y$. The system $(Y,S)$ is an almost 1-to-1 extension of the system 
$(M,h)$. The corresponding almost 1-to-1 factor map $p\colon (Y, S)\to (M, h)$ collapses the boundary circles of the discs $D_i$
to the points $x_i$, $i\in \mathbb Z$.

\begin{theorem}\label{T:Sierp}
	The Sierpi\'nski curve on the 2-torus and the Sierpi\'nski curve on the Klein bottle are homeo-product-minimal spaces.
\end{theorem}

\begin{proof}
Let	$Z$ be either the 2-torus or the Klein bottle. We use Proposition~\ref{P:general} with this space $Z$ and with the Sierpi\'nski curve $Y$ on $Z$. Then $Z$ is homeo-product-minimal by Theorem~\ref{T:groups} or by Theorem~\ref{T:Klein}, respectively. Fix a minimal homeomorphism $h$ on the manifold $Z$. The blowing-up construction described above gives a minimal homeomorphism $S\colon Y\to Y$ (recall that all the Sierpi\'nski curves on $Z$ are homeomorphic) such that it is an almost 1-to-1 extension of $h$.
Then Proposition~\ref{P:general} gives that the Sierpi\'nski curve $Y$ is homeo-product-minimal. 
\end{proof}


\section{Spaces admitting minimal continuous flows}

If $(X,T)$ is a dynamical system and $A,A'\subseteq X$ put
\begin{equation*}
\Hit_T(A,A')=\{n\in\mathbb N\colon T^n(A)\cap A'\neq\emptyset\}.
\end{equation*}
Recall that a system $(X,T)$ is called \emph{topologically transitive} if the set $\Hit_T(U,U')$
is nonempty for each pair of nonempty open sets $U,U'\subseteq X$. Notice that in such a case all the sets $\Hit_T(U,U')$ are infinite. (Indeed, topological transitivity of $T$ implies that preimages of nonempty open sets are nonempty. Thus, for every $k\in\mathbb N$, $\Hit_T(U,T^{-k}(U'))\neq\emptyset$ and so $\Hit_T(U,U')$ contains an integer $n>k$.)

Analogously, given a flow $\phi=(\varphi_t)_{t\in\mathbb R}$ on a metrizable space $Y$ and sets $A,A'\subseteq Y$, we let
\begin{equation*}
\Hit_{\phi}(A,A')=\{t\in\mathbb R\colon \varphi_t(A)\cap A'\neq\emptyset\}.
\end{equation*}
Recall that if $\phi$ is minimal and $Y$ is compact then all the sets $\Hit_{\phi}(V,V')$ are syndetic 
(i.e., they have bounded gaps)
for each pair of nonempty open sets $V,V'$.

If $A=\{a\}$ is a singleton, instead of $\Hit_T(A,A')$ and $\Hit_\phi(A,A')$ we just write
$\Hit_T(a,A')$ and $\Hit_\phi(a,A')$.

\begin{proposition}\label{P:flow-trans}
	Let $Y$ be a compact metrizable space admitting a minimal continuous flow $\phi=(\varphi_t)_{t\in\mathbb R}$. Then
	\begin{enumerate}
		\item for every topologically transitive system $(X,T)$ on a second countable metrizable space $X$ and for residually many $t\in\mathbb R$, the product $(X\times Y,T\times \varphi_t)$ is topologically transitive,
		\item for every point-transitive system $(X,T)$ and for residually many $t\in\mathbb R$, the product $(X\times Y,T\times \varphi_t)$ is point-transitive.
	\end{enumerate} 
\end{proposition}
\begin{proof} 
	(1)
	Fix a topologically transitive system $(X,T)$ and a countable basis $(U_p)_{p\in\mathbb N}$ of $X$. Let $(V_r)_{r\in\mathbb N}$ be a countable basis of $Y$. Given $p,r\in\mathbb N$, write $W_{p}^r=U_p\times V_r$. Clearly, the sets $W_{p}^r$ form a (countable) basis of $X\times Y$. For $p,q,r,s\in\mathbb N$ let
\begin{equation*}
A_{pq}^{rs}=\{t\in\mathbb N\colon (T\times\varphi_t)^k(W_{p}^r)\cap W_q^s\neq\emptyset\text{ for infinitely many }k\in\mathbb N\}.
\end{equation*}
Notice that $(T\times\varphi_t)^k(W_{p}^r)\cap W_q^s\neq\emptyset$ if and only if $k\in\Hit_T(U_p,U_q)$ and $kt\in\Hit_{\phi}(V_r,V_s)$. Let $(k_n)_{n=1}^{\infty}$ be the list of all elements of the infinite set $\Hit_T(U_p,U_q)$. Then
\begin{equation*}
A_{pq}^{rs}=\bigcap_{m=1}^{\infty}\bigcup_{n=m}^{\infty}\frac{1}{k_n}\Hit_{\phi}(V_r,V_s).
\end{equation*}
Since the set $\Hit_{\phi}(V_r,V_s)$ is syndetic by compactness of $Y$, it follows that the set $\bigcup_{n=m}^{\infty}\frac{1}{k_n}\Hit_{\phi}(V_r,V_s)$ is dense in $\mathbb R$ for every $m\in\mathbb N$. Moreover, these unions are obviously open sets. Consequently, $A_{pq}^{rs}$ is a dense $G_{\delta}$ set for all $p,q,r,s\in\mathbb N$. Finally, notice that the product $T\times\varphi_t$ is transitive if and only if $t\in\bigcap_{p,q,r,s\in\mathbb N}A_{pq}^{rs}$.

(2)
Fix a point-transitive system $(X,T)$ with a transitive point $x$ and a countable basis $(U_p)_{p\in\mathbb N}$ of $X$. Let $(V_r)_{r\in\mathbb N}$ be a countable basis of $Y$. Given $p,r\in\mathbb N$, write $W_{p}^r=U_p\times V_r$. Clearly, the sets $W_p^r$ form a (countable) basis of $X\times Y$. Fix $y\in Y$ and for $p,r\in\mathbb N$ let
\begin{equation*}
A_{p}^{r}=\{t\in\mathbb N\colon (T\times\varphi_t)^k(x,y)\in W_{p}^r\text{ for infinitely many }k\in\mathbb N\}.
\end{equation*}
Notice that $(T\times\varphi_t)^k(x,y)\in W_{p}^r$ if and only if $k\in\Hit_T(x,U_p)$ and $kt\in\Hit_{\phi}(y,V_r)$. Let $(k_n)_{n=1}^{\infty}$ be the list of all elements of the infinite set $\Hit_T(x,U_p)$. Then
\begin{equation*}
A_{p}^{r}=\bigcap_{m=1}^{\infty}\bigcup_{n=m}^{\infty}\frac{1}{k_n}\Hit_{\phi}(y,V_r).
\end{equation*}
Since the set $\Hit_{\phi}(y,V_r)$ is syndetic by compactness of $Y$, it follows that the set $\bigcup_{n=m}^{\infty}\frac{1}{k_n}\Hit_{\phi}(y,V_r)$ is dense in $\mathbb R$ for every $m\in\mathbb N$. Moreover, these unions are obviously open sets. Consequently, $A_{p}^{r}$ is a dense $G_{\delta}$ set for all $p,r\in\mathbb N$. Finally, notice that $(x,y)$ is a transitive point for the product $T\times\varphi_t$ if and only if $t\in\bigcap_{p,r\in\mathbb N}A_{p}^{r}$.
\end{proof}

Given a flow $\phi=(\varphi_t)_{t\in\mathbb R}$ on a compact metrizable space $Y$, we shall denote by $Z(\phi)$ the centralizer of the set $\{\varphi_t\colon t\in\mathbb R\}$ in the group $\Homeo(Y)$ and call it the \emph{centralizer} of $\phi$. Thus
\begin{equation*}
Z(\phi)=\{\psi\in\Homeo(Y)\colon \psi\circ\varphi_t=\varphi_t\circ\psi\text{ for every }t\in\mathbb R\}.
\end{equation*}

\begin{theorem}\label{T:flow.centr}
	Let $Y$ be a compact metrizable space. Assume that $Y$ admits a minimal continuous flow $\phi=(\varphi_t)_{t\in\mathbb R}$, whose centralizer $Z(\phi)$ in $\Homeo(Y)$ acts transitively on $Y$ in the algebraic sense. Then for every minimal system $(X,T)$ and residually many $t\in\mathbb R$, the product $(X\times Y,T\times \varphi_t)$ is minimal. Consequently, $Y$ is homeo-product-minimal.
\end{theorem}
\begin{proof}
	By Proposition~\ref{P:flow-trans}(2) we have a residual set $R$ of times $s\in\mathbb R$ with
	$T\times \varphi_s$ point-transitive. We show that for all $s\in R$ the map $T\times \varphi_s$ is 
	even minimal (so the situation is similar to that when one works with a minimal flow and its transitive resp. minimal time-$t$ maps; see e.g.~\cite{Eg} or \cite{Fa}). So fix $s\in R$ along with a nonempty closed $(T\times\varphi_s)$-invariant set $M\subseteq X\times Y$; we need to show that $M=X\times Y$.
	For every $\psi\in Z(\phi)$ put $M_{\psi}=(\Id_X\times\psi)(M)$. Since $\Id_X\times\psi$ is a conjugacy 	$(M,T\times \varphi_s)\to(M_{\psi},T\times \varphi_s)$, it follows that $M_{\psi}$ is a nonempty closed $(T\times\varphi_s)$-invariant subset of $X\times Y$ for every $\psi\in Z(\phi)$. By compactness of $Y$, the projection $X\times Y\to X$ is a closed map and so all the sets $M_{\psi}$ have full projections onto $X$ by minimality of $T$. Since $Z(\phi)$ acts transitively on $Y$, it follows that $\bigcup_{\psi\in Z(\phi)}M_{\psi}=X\times Y$. Now choose a transitive point $(x,y)$ for $T\times\varphi_s$ and take $\psi_0\in Z(\phi)$ with $(x,y)\in M_{\psi_0}$. Then, by $(T\times\varphi_s)$-invariance of $M_{\psi_0}$, the dense orbit of $(x,y)$ is contained in the closed set $M_{\psi_0}$. Thus $M_{\psi_0}=X\times Y$, whence it follows that $M=X\times Y$.
\end{proof}

\begin{remark}\label{R:homogeneous}
	Notice that every space $Y$ satisfying the assumptions of Theorem~\ref{T:flow.centr} is a homogeneous continuum.
\end{remark}

\begin{theorem}\label{T:flow.centr.gp}
	Let $G$ be a compact connected metrizable abelian group and let $q\colon\mathbb R\to G$ be a topological morphism with a dense image. Then for every minimal system $(X,T)$ and residually many $t\in\mathbb R$, the product of $(X,T)$ with the rotation of $G$ by $q(t)$ is minimal.
\end{theorem}
\begin{proof}
Consider the standard equicontinuous flow $\phi=(\varphi_t)_{t\in\mathbb R}$ on $G$ generated by $q$; thus, $\varphi_t$ is the rotation of $G$ by $q(t)$ for every $t\in\mathbb R$. Since $q$ has a dense image, the flow $\phi$ is minimal. Moreover, the centralizer $Z(\phi)$ of $\phi$ contains the group of all rotations of $G$ and so it acts on $G$ transitively in the algebraic sense. Thus, the theorem follows from Theorem~\ref{T:flow.centr}.
\end{proof}

\begin{remark}
	Since every compact connected metrizable abelian group $G$ admits a topological morphism  $q\colon \mathbb R\to G$ with a dense image (i.e., is solenoidal, see, e.g., \cite[Theorem~16]{AnzKak}), Theorem~\ref{T:flow.centr.gp} gives yet another proof of the fact that such groups $G$ are homeo-product-minimal spaces. It shows that for every minimal system $(X,T)$ there exists a rotation $R_g$ of $G$ by $g$ such that the product $T\times R_g$ is minimal and, in fact, such $g$ can be chosen in the set $q(\mathbb R)$, which is a subset of the identity path-component of $G$.
\end{remark}


\section{Topological manifolds}\label{S:mnflds}

Our aim in this section is to prove Theorems~\ref{T:manif.homeo-PM} and~\ref{T:manif.homeo-PM.Lie} below. They are based on ideas from \cite{FH}. 

\begin{theorem}\label{T:manif.homeo-PM}
Let $Y$ be a compact connected manifold without boundary admitting a free action of $\mathbb S^1$. Then for every minimal system $(X,T)$ there is a homeomorphism $S\colon Y\to Y$ isotopic to the identity such that the product $(X\times Y,T\times S)$ is minimal. Consequently, $Y$ is a homeo-product-minimal space.
\end{theorem}

Before proving this theorem, we discuss some consequences of it.

\begin{remark}\label{R:manif.homeo-PM}
It follows from the theorem that all odd-dimensional spheres are homeo-product-minimal spaces. Indeed, recall that $\mathbb S^{2n-1}=\{(z_1,\dots,z_n)\in\mathbb C^n\colon \sum_{j=1}^n|z_j|^2=1\}$ admits a free action of $\mathbb S^1$, given by the rule $z(z_1,\dots,z_n)=(zz_1,\dots,zz_n)$.
\end{remark}

The following theorem easily follows from Theorem~\ref{T:manif.homeo-PM}.

\begin{theorem}\label{T:manif.homeo-PM.Lie}
	Let $Y$ be a compact connected manifold without boundary and $G$ be a nontrivial compact connected Lie group. Assume that $Y$ admits a free action of $G$. Then for every minimal system $(X,T)$ there is a homeomorphism $S\colon Y\to Y$ isotopic to the identity such that the product $(X\times Y,T\times S)$ is minimal. Consequently, $Y$ is a homeo-product-minimal space.
\end{theorem}
\begin{proof}[Proof of Theorem~\ref{T:manif.homeo-PM.Lie}]
	Since every nontrivial compact connected Lie group has a nontrivial maximal torus (see e.g.~\cite[Theorem~6.30,~p.~212]{HofMor}), $G$ contains a subgroup $H$ topologically isomorphic to $\mathbb S^1$. Moreover, $H$ acts (continuously and) freely on $Y$, since $G$ does. Therefore, the theorem follows from Theorem~\ref{T:manif.homeo-PM} above.
\end{proof}

\begin{remark}\label{R:manif.homeo-PM.Lie2}
Theorems~\ref{T:manif.homeo-PM} and \ref{T:manif.homeo-PM.Lie} speak about the same class of spaces $Y$. Indeed, a compact connected manifold $Y$ without boundary admits a free action of $\mathbb S^1$ if and only if it admits a free action of a nontrivial compact connected Lie group $G$. (One implication is trivial, the other one follows from the proof of Theorem~\ref{T:manif.homeo-PM.Lie}.)
\end{remark}

\begin{remark}\label{R:manif.homeo-PM.Lie}
As an immediate corollary of Theorem~\ref{T:manif.homeo-PM.Lie} we see that all compact connected Lie groups are homeo-product-minimal spaces (and hence they admit minimal homeomorphisms, as already noticed in \cite[Example~3.9(b)]{FH}). In particular, the following matrix groups are homeo-product-minimal: special orthogonal groups $\operatorname{SO}(n)$, spin groups  $\operatorname{Spin}(n)$, compact symplectic groups $\operatorname{Sp}(n)$, unitary groups $\operatorname{U}(n)$ and special unitary groups $\operatorname{SU}(n)$. In connection with Remark~\ref{R:manif.homeo-PM}, recall also that $\mathbb S^{2n-1}$ admits the structure of a Lie group if and only if $n=1$ or $n=2$ (see \cite[Corollary~9.59,~p.~497]{HofMor}).
\end{remark}

In the remaining part of this section we are going to prove Theorem~\ref{T:manif.homeo-PM}.

\subsection{Auxiliary results}

Given topological spaces $Y,Z$, a continuous map $F\colon[0,1]\times Y\to Z$ and $t\in[0,1]$, we shall write $F_t$ for the map $Y\to Z$, given by the rule $F_t(y)=F(t,y)$ ($y\in Y$). We recall that the map $F$ is called a homotopy between $F_0$ and $F_1$. If all the maps $F_t$ ($t\in[0,1]$) are topological embeddings then we call $F$ a homotopy of embeddings.
If $Y=Z$ and all the maps $F_t$ ($t\in[0,1]$) are homeomorphisms then $F$ is called an isotopy between $F_0$ and $F_1$. Finally, if $Y$ is a subspace of $Z$ and $z_0\in Y$ is such that $F_t(z_0)=z_0$ for every $t\in[0,1]$ then we say that the homotopy $F$ fixes $z_0$.

In the proof of Lemma~\ref{L:aux.hit} below we shall use the following result, which is a special case of \cite[Corollary~1.2]{EdwKir}.

\begin{theorem}[\cite{EdwKir}]\label{T:Edw.Kir.ext}
	Let $Z$ be a compact connected manifold without boundary, $C\subseteq Z$ be a closed set and $U\subseteq Z$ be an open set containing $C$. Let $\Phi\colon[0,1]\times U\to Z$ be a homotopy of embeddings with $\Phi_0=\Id_U$. Then there is an isotopy $\Psi\colon[0,1]\times Z\to Z$ such that $\Psi_0=\Id_Z$ and $\Psi(t,z)=\Phi(t,z)$ for all $t\in[0,1]$ and $z\in C$.
\end{theorem}
\begin{remark}
	In the mentioned result from \cite{EdwKir}, the resulting map $\Psi$ is a homotopy of embeddings; that is, $\Psi_t\colon Z\to Z$ is a topological embedding for every $t\in[0,1]$. However, it follows from our assumptions on $Z$ that all $\Psi_t$ are homeomorphisms. Indeed, the set $\Psi_t(Z)$ is closed in $Z$ by compactness of $Z$ and it is also open in $Z$ by the invariance of domain theorem. Hence $\Psi_t(Z)=Z$ by connectedness of $Z$ and $\Psi_t\colon Z\to Z$ is thus a homeomorphism.
\end{remark}

We shall make use of the following well known result on homogeneity of connected manifolds without boundary
(cf.~\cite[p.~142]{GP}).

\begin{lemma}[Isotopy lemma]\label{L:manif.hom}
	Let $M$ be a (not necessarily compact) connected manifold without boundary. Given $x,y\in M$, there is an isotopy $F\colon[0,1]\times M\to M$ with $F_0=\Id_M$ and $F_1(x)=y$.
\end{lemma}

\begin{lemma}\label{L:aux.hit}
	Let $Z$ be a compact connected manifold without boundary admitting a free action of $\mathbb S^1$; fix such an action. Then for every $z_0\in Z$ and every nonempty open set $W\subseteq Z$ there is a homeomorphism $\sigma\colon Z\to Z$ with the following properties:
	\begin{enumerate}
		\item[(a)] $\sigma$ is isotopic to the identity via an isotopy fixing $z_0$,
		\item[(b)] $\sigma(W)$ intersects each orbit of the action of $\mathbb S^1$ in $Z$.
	\end{enumerate}
\end{lemma}
\begin{proof}
	We may assume that $Z$ has topological dimension $n\geq2$ (otherwise $Z$ is homeomorphic to $\mathbb S^1$ and we may take $\sigma=\Id_Z$). Write $B=Z/\mathbb S^1$ and denote by $p$ the quotient map $Z\to B$. It follows from the assumptions of the lemma that $Z$ is a fiber bundle with the base $B$, projection $p$ and fiber $\mathbb S^1$ (see e.g.~\cite[Theorem~5.8,~p.~88]{Bre}). That is, there is a finite open cover $V_1,\dots,V_N$ of $B$ such that for every $i=1,\dots,N$ there exists a homeomorphism $\varphi_i\colon V_i\times\mathbb S^1\to p^{-1}(V_i)$ with $p\varphi_i(b,w)=b$ for all $b\in V_i$ and $w\in\mathbb S^1$. It follows, in particular, that $B$ is a compact connected manifold without boundary of topological dimension $n-1$.
	
	\emph{1st step.} We show that there exist a positive integer $m$, indices $i_1,\dots,i_m\in\{1,\dots,N\}$, closed topological $(n-1)$-dimensional discs $K_1,\dots,K_m\subseteq B$, open topological $(n-1)$-dimensional discs $Y_1,\dots,Y_m\subseteq B$ and points $w_1,\dots,w_m\in\mathbb S^1$ such that
	\begin{itemize}
		\item $K_j\subseteq Y_j\subseteq\overline{Y_j}\subseteq V_{i_j}$ for every $j=1,\dots,m$,
		\item $B=\bigcup_{j=1}^mK_j$,
		\item the sets $D_j=\varphi_{i_j}(\overline{Y_j}\times\{w_j\})$ ($j=1,\dots,m$) are mutually disjoint and contained in $Z\setminus\{z_0\}$.
	\end{itemize}
	
	Let us begin by choosing, for every $i=1,\dots,N$, a closed set $Q_i\subseteq V_i$ in such a way that $B=\bigcup_{i=1}^NQ_i$. In our construction of the required objects we shall proceed inductively, using the discs $K_j$ to cover step-by-step each of the closed sets $Q_i$.
	
	The 0-th step of the induction is ``empty''; just put $l_0=0$. In the induction step, fix 
	$k\in\{1,\dots,N\}$ and assume that the discs $Q_1,\dots,Q_{k-1}$ are already covered, meaning that
	 there are a nonnegative integer $l=l_{k-1}$, indices $i_1,\dots,i_l\in\{1,\dots,N\}$, closed topological $(n-1)$-dimensional discs $K_1,\dots,K_l\subseteq B$, open topological $(n-1)$-dimensional discs $Y_1,\dots,Y_l\subseteq B$ and points $w_1,\dots,w_l\in\mathbb S^1$ such that
	\begin{itemize}
		\item $K_j\subseteq Y_j\subseteq\overline{Y_j}\subseteq V_{i_j}$ for every $j=1,\dots,l$,
		\item $\bigcup_{j=1}^lK_j\supseteq\bigcup_{i=1}^{k-1}Q_i$,
		\item the sets $D_j=\varphi_{i_j}(\overline{Y_j}\times\{w_j\})$ ($j=1,\dots,l$) are mutually disjoint and contained in $Z\setminus\{z_0\}$.
	\end{itemize}
	
	Since each $D_j$ ($j=1,\dots,l$) intersects every orbit of $\mathbb S^1$ in at most one point, the open subset $U=Z\setminus(\bigcup_{j=1}^lD_j\cup\{z_0\})$ of $Z$ projects onto $B$ via $p$. Consequently, the open subset $\varphi_k^{-1}(p^{-1}(V_k)\cap U)$ of $V_k\times\mathbb S^1$ projects onto $V_k$ via $\pr_1$. It follows that for every $b\in Q_k$ there exist a closed topological $(n-1)$-dimensional disc $K^{(b)}$, an open topological $(n-1)$-dimensional disc $Y^{(b)}$ and a point $w^{(b)}\in\mathbb S^1$ such that
	\begin{itemize}
		\item $b\in\Int(K^{(b)})\subseteq K^{(b)}\subseteq Y^{(b)}\subseteq\overline{Y^{(b)}}\subseteq V_k$,
		\item $\overline{Y^{(b)}}\times\{w^{(b)}\}\subseteq\varphi_k^{-1}(p^{-1}(V_k)\cap U)$.
	\end{itemize}
	
	Since the set $Q_k$ is compact, there is a finite family of points $b_1,\dots,b_p$ such that $Q_k\subseteq\bigcup_{t=1}^pK^{(b_t)}$. We may clearly assume that the points $w^{(b_t)}$ ($t=1,\dots,p$) are chosen in such a way that the sets $\overline{Y^{(b_t)}}\times\{w^{(b_t)}\}$ are mutually disjoint. Now put $l_k=l+p$ and for every $j=l+1,\dots,l+p$ put $i_j=k$, $K_j=K^{(b_{j-l})}$, $Y_j=Y^{(b_{j-l})}$ and $w_j=w^{(b_{j-l})}$. Then
	\begin{itemize}
		\item $K_j\subseteq Y_j\subseteq\overline{Y_j}\subseteq V_{i_j}$ for every $j=1,\dots,l+p$,
		\item $\bigcup_{j=1}^{l+p}K_j\supseteq\bigcup_{i=1}^kQ_i$,
		\item the sets $D_j=\varphi_{i_j}(\overline{Y_j}\times\{w_j\})$ ($j=1,\dots,l+p$) are mutually disjoint and contained in $Z\setminus\{z_0\}$.
	\end{itemize}
	This verifies the induction step and finishes the first step of the proof.
	
	\emph{2nd step.} Set $C_j=\varphi_{i_j}(K_j\times\{w_j\})$ for $j=1,\dots,m$. By the first step of the proof, the sets $C_j$ are mutually disjoint closed topological $(n-1)$-dimensional discs in $Z$ contained in $Z\setminus\{z_0\}$. We show that there is a family of mutually disjoint open sets (in fact, open topological $n$-dimensional discs) $U_1,\dots,U_m\subseteq Z$ such that $C_j\subseteq U_j$ for $j=1,\dots,m$ and $\dist(z_0,\bigcup_{j=1}^mU_j)>0$.
	
	It follows from the properties of the objects constructed in the first step that we may let $U_j=\varphi_{i_j}(Y_j\times J_j)$ for $j=1,\dots,m$, where $J_j$ is a sufficiently small open interval in $\mathbb S^1$ containing $w_j$.
	
	\emph{3rd step.} Let $U=\bigcup_{j=1}^mU_j$. We show that there exist an open neighborhood $U_0$ of $z_0$ in $Z$, disjoint from $U$, and a homotopy of embeddings $F_t\colon U_0\cup U\to Z$ ($t\in[0,1]$) such that $F_0=\Id_{U_0\cup U}$, $F_1(U)\subseteq W$ and $F_t|_{U_0}=\Id_{U_0}$ for every $t\in[0,1]$.
	
	Fix $z_j\in C_j$ for every $j=1,\dots,m$. Since $z_1,\dots,z_m\in Z$ are distinct points in a connected manifold without boundary $Z\setminus\{z_0\}$ of dimension at least $2$, we can join each $z_j$ with an element of $W$ by an arc $A_j\subseteq Z\setminus\{z_0\}$. We may clearly assume that the arcs $A_j$ ($j=1,\dots,m$) are mutually disjoint.
	Consequently, there exist disjoint connected open sets $A_j\subseteq M_j\subseteq Z$ and a neighborhood $U_0$ of $z_0$ in $Z$ disjoint from $\bigcup_{j=1}^mM_j$.
	
	We shall now describe the desired homotopy of embeddings $F_t$ ($t\in[0,1]$). First, as required, we let $U_0$ remain fixed during the homotopy. Given $j\in\{1,\dots,m\}$, we begin by squeezing $U_j$ within itself into a subset $U_j'$ of $M_j$ (this is possible, since $U_j$ is an open topological $n$-dimensional disc intersecting $M_j$). Further, using that $M_j$ is a connected manifold (being a connected open subset of a manifold $Z$), we may apply the Isotopy lemma and let $U_j'$ slide within $M_j$ to a set $U_j''$ intersecting $W$. Finally, we squeeze (the open topological $n$-dimensional disc) $U_j''$ within itself into a subset $U_j'''$ of $W$. The homotopy thus described is well defined, since the sets $U_0,U_1,\dots,U_m$ are disjoint and open in $Z$, and it is a homotopy of embeddings by disjointness of the sets $\bigcup_{t\in[0,1]}F_t(U_j)$ ($j=0,\dots,m$).
	
	\emph{4th step.} We prove the statement of the lemma.
	
	Set $U^*=\bigcup_{j=1}^mU_j'''$. Then $U^*\subseteq W$ and $U^*=\bigcup_{j=1}^mF_1(U_j)=F_1(U)$. Since $F_1\colon U_0\cup U\to Z$ is an embedding with the image $U_0\cup U^*$, its restriction $f\colon U_0\cup U\to U_0\cup U^*$ is a homeomorphism. Set $\Phi_t=F_{1-t}\circ f^{-1}$ for $t\in[0,1]$. The maps $\Phi_t$ thus defined constitute a homotopy of embeddings of $U_0\cup U^*$ into $Z$ and $\Phi_0=F_1\circ f^{-1}=\Id_{U_0\cup U^*}$. Set $C=\bigcup_{j=1}^mC_j$ and $C^*=\bigcup_{j=1}^m F_1(C_j)$; then $C\subseteq U$ and $C^*=F_1(C)=f(C)\subseteq f(U)=U^*$. By virtue of Theorem~\ref{T:Edw.Kir.ext}, the embeddings $\Phi_t|_{\{z_0\}\cup C^*}$ can be extended to homeomorphisms $\Psi_t\colon Z\to Z$, which constitute an isotopy with $\Psi_0=\Id_Z$.
	
	We finish the proof by showing that the map $\sigma=\Psi_1$ satisfies conditions (a) and (b) from the lemma. First, since $F_t(z_0)=z_0$ for every $t\in[0,1]$, condition (a) follows from our definition of $\sigma$ and the isotopy $\Psi_t$ ($t\in[0,1]$). To verify condition (b), notice that
	\begin{equation*}
	\sigma(W)\supseteq\sigma(U^*)\supseteq\sigma(C^*)=\Psi_1(C^*)=\Phi_1(C^*)
	=F_0(f^{-1}(C^*))=F_0(C)=C=\bigcup_{j=1}^mC_j.
	\end{equation*}
	Consequently,
	\begin{equation*}
	p(\sigma(W))\supseteq p\Big(\bigcup_{j=1}^mC_j\Big)=\bigcup_{j=1}^mp(C_j)
	=\bigcup_{j=1}^mK_j=B,
	\end{equation*}
	as was to be shown.
\end{proof}

Before turning to the next auxiliary result, let us recall some basic facts concerning fundamental groups. If $Z$ is a connected manifold with a base point $z_0$, we denote its fundamental group by $\pi_1(Z,z_0)$ or, briefly, by $\pi_1(Z)$. Further, if $X$ is a connected manifold with a base point $x_0$ and $f\colon X\to Z$ is a continuous map with $f(x_0)=z_0$, we write $f_*$ for the induced morphism $\pi_1(X)\to\pi_1(Z)$. Given a connected manifold $Y$ with a base point $y_0$ and a covering map $p\colon Y\to Z$ with $p(y_0)=z_0$, we recall that, by the Lifting lemma (cf. \cite[Lemma~79.1, p.~478]{Mun}), the following conditions are equivalent:
\begin{itemize}
\item there is a continuous map $g\colon X\to Y$ with $g(x_0)=y_0$ and $f=p\circ g$,
\item $f_*\pi_1(X)\subseteq p_*\pi_1(Y)$.
\end{itemize}
We also recall that such a map $g$, if it exists, is unique.

\begin{lemma}\label{L:aux.hit.inv}
Let $Y$ be a compact connected manifold without boundary admitting a free action $A$ of $\mathbb S^1$. Then for every finite subgroup $H$ of $\mathbb S^1$ and every nonempty open set $V\subseteq Y$ there is a homeomorphism $\psi\colon Y\to Y$ with the following properties:
\begin{enumerate}
\item $\psi$ is isotopic to the identity,
\item $\psi$ commutes with the induced action of $H$ on $Y$,
\item $\psi(V)$ intersects each $A$-orbit of $\mathbb S^1$ in $Y$.
\end{enumerate}
\end{lemma}
\begin{proof}
If $\zeta\in\mathbb S^1$, we will denote by $\varphi_\zeta$ the corresponding acting homeomorphism of $A$ on $Y$. However, symbols $A(\zeta,y)=\varphi_\zeta(y)$ and $\zeta y$ will be used interchangeably.
Let $\pi\colon Y\to Y/\mathbb S^1$, $p\colon Y\to Y/H$ be the canonical quotient maps and $q\colon\mathbb S^1\to\mathbb S^1$ be the quotient morphism with kernel $\ker(q)=H$ (that is, $q(z)=z^{\card H}$). Since the induced action of the finite group $H$ on $Y$ is free, the quotient map $p\colon Y\to Y/H$ is a covering map \cite[Theorem~81.5 and Exercise~4 on p.~493]{Mun}. The original action $A$ of $\mathbb S^1$ on $Y$ then descends via $p$ and $q$ to an action $A'$ of $\mathbb S^1$ on $Y/H$
\begin{equation*}
A'\colon\mathbb S^1\times Y/H\to Y/H,\quad (q(\zeta),p(y))\mapsto p(\zeta y)
\end{equation*}
for $y\in Y$ and $\zeta\in\mathbb S^1$, see Figure~\ref{Fig:cd-action-A'}. 
\begin{figure}[!htb]\centering
	\begin{minipage}{0.49\textwidth}
		\begin{center}
			$\begin{CD}
				\mathbb S^1\times Y     @>A>>  Y\\
				@Vq\times pVV        @VVpV\\
				\mathbb S^1\times Y/H     @>A'>>  Y/H
			\end{CD}$
			\caption{Action $A'$}\label{Fig:cd-action-A'}		
		\end{center}
	\end{minipage}
		\begin {minipage}{0.49\textwidth}
		\begin{center}
			$\begin{CD}
			Y     @>p>>  Z=Y/H\\
			@V\pi{}VV        @VV\pi'V\\
			Y/\mathbb S^1     @>g>>  Z/\mathbb S^1
			\end{CD}$
			\caption{Homeomorphism $g$}\label{Fig:cd-homeo-g}		
		\end{center}
\end{minipage}
\end{figure}
We shall now proceed in five steps.

\emph{1st step.} Write $Z=Y/H$. We show that $Z$ and $A'$ satisfy the assumptions of Lemma~\ref{L:aux.hit}.

First, since $p\colon Y\to Z$ is a covering map, the space $Z$ is a compact connected manifold without boundary. Since the original action $A$ of $\mathbb S^1$ on $Y$ is free and $\ker(q)=H$, it follows that the descended action $A'$ of $\mathbb S^1$ on $Z$ is also free. Moreover, if $\pi'\colon Z\to Z/\mathbb S^1$ denotes the canonical quotient map from $Z$ to the orbit space of $A'$, then there is a homeomorphism $g\colon Y/\mathbb S^1\to Z/\mathbb S^1$ with $g\circ\pi=\pi'\circ p$, see Figure~\ref{Fig:cd-homeo-g}.

\emph{2nd step.} We define the desired homeomorphism $\psi\colon Y\to Y$.

Fix $y_0\in Y$ and write $z_0=p(y_0)$; we shall use $y_0$ and $z_0$ as base points of the spaces $Y$ and $Z$, respectively. Set $W=p(V)$. Since the map $p$, being a covering map, is open, it follows that $W$ is a nonempty open subset of $Z$. Thus, by using the first step of the proof along with Lemma~\ref{L:aux.hit}, we find a homeomorphism $\sigma\colon Z\to Z$ with the following properties:
\begin{enumerate}
\item[(a)] $\sigma$ is isotopic to the identity via an isotopy $F$ fixing $z_0$,
\item[(b)] $\sigma(W)$ intersects each $A'$-orbit of $\mathbb S^1$ in $Z$.
\end{enumerate}

Consider the morphisms $p_*\colon\pi_1(Y)\to\pi_1(Z)$ and $\sigma_*\colon\pi_1(Z)\to\pi_1(Z)$, induced by $p$ and $\sigma$, respectively. By virtue of (a), we have $\sigma_*=\Id_{\pi_1(Z)}$ and, by functoriality of the induced morphism, $(p \circ \sigma)_*=p_*\circ \sigma_*=p_*$. So, 
by the Lifting lemma, $\sigma$ lifts uniquely across $p$ to a homeomorphism $\psi\colon Y\to Y$ with $\psi(y_0)=y_0$; that is, we have $p\circ\psi=\sigma\circ p$, see Figure~\ref{Fig:cd-psi}. 
\begin{figure}[!htb]\centering
	\begin{minipage}{0.49\textwidth}
		\begin{center}
			\begin{tikzcd}
			(Y,y_0)\ \arrow[r,dashed]{}{\psi}
			  \arrow[d,"p" left] 
	        & \ (Y,y_0)
	          \arrow[d,"p"]
			\\
			(Z,z_0)\ 
			  \arrow[r]{}{\sigma=F_1} 
			& \ (Z,z_0)
			\end{tikzcd}	
			\caption{$\sigma\circ p$ lifts to $\psi$}\label{Fig:cd-psi}
		\end{center}
	\end{minipage}
	\begin {minipage}{0.49\textwidth}
	\begin{center}
			\begin{tikzcd}
				([0,1]\times Y, (0,y_0))\ \arrow[r,dashed]{}{G}
				  \arrow[d,"\Id_{[0,1]}\times p" left] 
		        & \ (Y,y_0)
		          \arrow[d,"p"]
				\\
				([0,1]\times Z, (0,z_0))\ 
				  \arrow[r]{}{F} 
				& \ (Z,z_0)
			\end{tikzcd}
			\caption{$F\circ (\Id_{[0,1]}\times p)$ lifts to $G$}\label{Fig:cd-G}		
	\end{center}
	\end{minipage}
\end{figure}

\emph{3rd step.} We show that $\psi$ satisfies condition (1).

We shall use $(0,z_0)$ and $(0,y_0)$ as base points of $[0,1]\times Z$ and $[0,1]\times Y$, respectively. Under the natural identifications $\pi_1([0,1]\times Y)=\pi_1(Y)$ and $\pi_1([0,1]\times Z)=\pi_1(Z)$, we have $F_*=\Id_{\pi_1(Z)}$ and $(\Id_{[0,1]}\times p)_*=p_*$. It follows from the Lifting lemma that $F$ lifts uniquely across $p$ to a continuous map $G\colon[0,1]\times Y\to Y$ with $G(0,y_0)=y_0$; that is, we have $F\circ(\Id_{[0,1]}\times p)=p\circ G$, see Figure~\ref{Fig:cd-G}. 
\begin{figure}[!htb]\centering
	\begin {minipage}{0.49\textwidth}
	\begin{center}
			\begin{tikzcd}
			(Y,y_0)\ \arrow[r,dashed]{}{G_t}
			  \arrow[d,"p" left] 
	        & \ (Y,y_0)
	          \arrow[d,"p"]
			\\
			(Z,z_0)\ 
			  \arrow[r]{}{F_t} 
			& \ (Z,z_0)
			\end{tikzcd}	
			\caption{$F_t\circ p$ lifts to $G_t$}\label{Fig:cd-Gt}		
	\end{center}
	\end{minipage}
	\begin {minipage}{0.49\textwidth}
	\begin{center}
			\begin{tikzcd}
			(Y,y_0)\ \arrow[r,dashed]{}{\Gamma_t}
			  \arrow[d,"p" left] 
	        & \ (Y,y_0)
	          \arrow[d,"p"]
			\\
			(Z,z_0)\ 
			  \arrow[r]{}{F_t^{-1}} 
			& \ (Z,z_0)
			\end{tikzcd}	
			\caption{$F_t^{-1}\circ p$ lifts to $\Gamma_t$}\label{Fig:cd-Gt-inv}		
	\end{center}
	\end{minipage}
\end{figure}

We are going to show that $G$ is an isotopy from $\Id_Y$ to $\psi$. 

Since $F(t,z_0)=z_0$ for every $t\in[0,1]$ and $G(0,y_0)=y_0$, we also have $G(t,y_0)=y_0$ for every $t\in[0,1]$. It follows, using commutative diagram in Figure~\ref{Fig:cd-G}, that the diagram in Figure~\ref{Fig:cd-Gt} commutes for every $t$.
Now we use the uniqueness part of the Lifting lemma. For $t=0$, $F_0\circ p$ lifts to $G_0$, but trivially also
to $\Id_Y$. Hence $G_0=\Id_Y$. For $t=1$, $F_1\circ p$ lifts to $G_1$ and, by Figure~\ref{Fig:cd-psi}, 
also to $\psi$. Hence $G_1=\psi$.

Since the maps $F_t$ ($t\in[0,1]$) are homeomorphisms on $Z$ isotopic to the identity, we have $(F_t)_*=\Id_{\pi_1(Z)}$ for every $t$. By functoriality of the induced morphism, $(F_t^{-1})_*=\Id_{\pi_1(Z)}$.
Hence, by the Lifting lemma, $F_t^{-1}\circ p$ lifts to some continuous map $\Gamma_t$, see Figure~\ref{Fig:cd-Gt-inv}. Gluing together the diagrams in Figures~\ref{Fig:cd-Gt}	and \ref{Fig:cd-Gt-inv} in two ways, we get that $p$ lifts to both $\Gamma_t\circ G_t$ and $G_t\circ\Gamma_t$, and trivially also to $\Id_Y$. By  uniqueness, $\Gamma_t\circ G_t = G_t\circ\Gamma_t = \Id_Y$. Thus every $G_t$ is a homeomorphism.

\emph{4th step.} We show that $\psi$ satisfies condition (2).

First we show that $\psi(\zeta y_0)=\zeta y_0$ for every $\zeta\in H$. So fix $\zeta\in H$ along with a path $\gamma\colon[0,1]\to Y$ from $y_0$ to $\zeta y_0$. Then $p\circ(\psi\circ\gamma)=\sigma\circ(p\circ\gamma)$ by definition of $\psi$. Since $\sigma_*=\Id_{\pi_1(Z)}$, it follows that the paths $p\circ(\psi\circ\gamma)$ and $p\circ\gamma$ are path-homotopic. Moreover, $(\psi\circ\gamma)(0)=\psi(y_0)=y_0=\gamma(0)$, hence the paths $\psi\circ\gamma$ and $\gamma$ are also path-homotopic. In particular, $\psi\circ\gamma$ and $\gamma$ have the same endpoint. Thus, $\psi(\zeta y_0)=(\psi\circ\gamma)(1)=\gamma(1)=\zeta y_0$, as was to be shown.

We show that $\psi$ commutes with the induced action of $H$ on $Y$. To this end, fix $\zeta\in H$ and consider the corresponding acting homeomorphism $\varphi_{\zeta}$ on $Y$. Then, by definition of $p$, $p\circ\varphi_{\zeta}^{-1}=p\circ\varphi_{\zeta}=p$ and so, using also the definition of $\psi$,
\begin{equation*}
p\circ(\varphi_{\zeta}^{-1}\circ\psi\circ\varphi_{\zeta})
=p\circ\psi\circ\varphi_{\zeta}=\sigma\circ p\circ\varphi_{\zeta}=\sigma\circ p.
\end{equation*}
Moreover, by using our claim from the preceding paragraph, we obtain
\begin{equation*}
(\varphi_{\zeta}^{-1}\circ\psi\circ\varphi_{\zeta})(y_0)=\varphi_{\zeta}^{-1}(\psi(\zeta y_0))=\zeta^{-1}(\zeta y_0)=y_0.
\end{equation*}
In summary, $\varphi_{\zeta}^{-1}\circ\psi\circ\varphi_{\zeta}$ is a continuous lift of $\sigma$ across $p$ mapping $y_0$ to $y_0$. Hence $\varphi_{\zeta}^{-1}\circ\psi\circ\varphi_{\zeta}=\psi$ and so $\psi$ commutes with $\varphi_{\zeta}$ indeed.

\emph{5th step.} We finish the proof by showing that $\psi$ satisfies condition (3).

Fix $y\in Y$ and set $z=p(y)$. Since the set $\sigma(W)$ intersects the $A'$-orbit of $z$, there is $\xi\in\mathbb S^1$ with $\xi z\in\sigma(W)$. Choose $\varrho\in\mathbb S^1$ with $\xi=q(\varrho)$. Then
\begin{equation*}
p(\varrho y)=A'(q(\varrho), p(y))=A'(\xi,z)=\xi z\in\sigma(W)=\sigma(p(V))=p(\psi(V))
\end{equation*}
and so there is $\zeta\in H$ with $\zeta(\varrho y)\in\psi(V)$. It follows that $(\zeta\varrho)y\in\psi(V)$ and $\psi(V)$ thus intersects the orbit of $y$ under the action of $\mathbb S^1$ on $Y$.
\end{proof}

\subsection{Proof of Theorem~\ref{T:manif.homeo-PM}}
First, we recall that $\Homeo(Y)$ is a Polish group with the topology of uniform convergence (of homeomorphisms and their inverses). Since the group $\Homeo(Y)$ is locally contractible \cite[Theorem~1]{Cer}, hence locally arc-wise connected, its identity arc-component $\Homeo_a(Y)$ is an open (hence closed) subgroup of $\Homeo(Y)$ and so it is also a Polish group.

\begin{proof}[Proof of Theorem~\ref{T:manif.homeo-PM}]
Fix a minimal system $(X,T)$. Fix a free action of $\mathbb S^1$ on $Y$ and denote the acting homeomorphisms of this action by $\varphi_z$ ($z\in\mathbb S^1$). Set
\begin{equation*}
\mathcal H=\{\psi^{-1}\circ\varphi_z\circ\psi\colon\psi\in\Homeo_a(Y)\text{ and }z\in\mathbb S^1\}.
\end{equation*}
Since $\varphi_z\in\Homeo_a(Y)$ for every $z\in\mathbb S^1$ by arc-connectedness of $\mathbb S^1$, we have $\mathcal H\subseteq\Homeo_a(Y)$. Consequently, $\overline{\mathcal H}\subseteq\Homeo_a(Y)$. Given nonempty open sets $U\subseteq X$ and $V\subseteq Y$, we set
\begin{equation*}
\mathcal H_{U,V}=\left\{\sigma\in\Homeo_a(Y)\colon\bigcup_{n=1}^{\infty}(T\times\sigma)^{-n}(U\times V)=X\times Y\right\}.
\end{equation*}
We shall now proceed in eight steps.

\emph{1st step.} Given $\psi\in\Homeo_a(Y)$ and nonempty open sets $U\subseteq X$ and $V\subseteq Y$, we show that
\begin{equation*}
\psi\mathcal H\psi^{-1}=\mathcal H\hspace{3mm}\text{and}\hspace{3mm}\psi\mathcal H_{U,V}\psi^{-1}=\mathcal H_{U,\psi(V)}.
\end{equation*}

The first equality follows immediately by definition of $\mathcal H$. To verify the second equality, fix $\sigma\in\Homeo_a(Y)$. Then
\begin{align*}
\bigcup_{n=1}^{\infty}(T\times(\psi\circ\sigma\circ\psi^{-1}))^{-n}(U\times \psi(V))&=\bigcup_{n=1}^{\infty}T^{-n}(U)\times\psi(\sigma^{-n}(V))\\
&=\bigcup_{n=1}^{\infty}(\Id_X\times\psi)\left((T\times\sigma)^{-n}(U\times V)\right)\\
&=(\Id_X\times\psi)\left(\bigcup_{n=1}^{\infty}(T\times\sigma)^{-n}(U\times V)\right).
\end{align*}
Since $\Id_X\times\psi$ is a homeomorphism on $X\times Y$, it follows that $\psi\circ\sigma\circ\psi^{-1}\in\mathcal H_{U,\psi(V)}$ if and only if $\sigma\in\mathcal H_{U,V}$. Thus, $\psi\mathcal H_{U,V}\psi^{-1}=\mathcal H_{U,\psi(V)}$, as was to be shown.

\emph{2nd step.} Let $U\subseteq X$ and $V\subseteq Y$ be nonempty open sets. We show that the set $\mathcal H_{U,V}$ is open in $\Homeo_a(Y)$.

Fix $\sigma_0\in\mathcal H_{U,V}$ and choose $x\in X$. By compactness of $Y$, there is $N\in\mathbb N$ with $\{x\}\times Y\subseteq\bigcup_{n=1}^N(T\times\sigma_0)^{-n}(U\times V)$. Let $\mathcal U$ be the set of all $\sigma\in\Homeo_a(Y)$ with $\{x\}\times Y\subseteq\bigcup_{n=1}^N(T\times\sigma)^{-n}(U\times V)$. Then $\sigma_0\in\mathcal U$ and $\mathcal U$ is an open subset of $\Homeo_a(Y)$ due to compactness of $Y$. Therefore, it suffices to show that $\mathcal U\subseteq\mathcal H_{U,V}$.

So let $\sigma\in\mathcal U$. By the tube lemma there is a neighborhood $U'$ of $x$ in $X$ with $U'\times Y\subseteq\bigcup_{n=1}^N(T\times\sigma)^{-n}(U\times V)$. By minimality of $T$, we have $X=\bigcup_{m=1}^{\infty}T^{-m}(U')$. Consequently,
\begin{align*}
\bigcup_{n=1}^{\infty}(T\times\sigma)^{-n}(U\times V)&\supseteq\bigcup_{m=1}^{\infty}(T\times\sigma)^{-m}\left(\bigcup_{n=1}^N(T\times\sigma)^{-n}(U\times V)\right)
\supseteq\bigcup_{m=1}^{\infty}(T\times\sigma)^{-m}(U'\times Y)\\
&=\left(\bigcup_{m=1}^{\infty}T^{-m}(U')\right)\times Y=X\times Y,
\end{align*}
whence it follows that $\sigma\in\mathcal H_{U,V}$.

\emph{3rd step.} Let $V'\subseteq Y$ be an open subset of $Y$ intersecting each orbit of $\mathbb S^1$ in $Y$. Assume that $\xi\in\mathbb S^1$ and an increasing sequence $(k_n)_{n=1}^{\infty}$ of positive integers are such that the set $\{\xi^{k_n}\colon n\in\mathbb N\}$ is dense in $\mathbb S^1$. We show that there exists $N\in\mathbb N$ with $Y=\bigcup_{n=1}^N(\varphi_{\xi})^{-k_n}(V')$.

By compactness of $Y$, it suffices to show that $Y=\bigcup_{n=1}^{\infty}(\varphi_{\xi})^{-k_n}(V')$. So let $y\in Y$. Since $V'$ intersects the orbit of $y$ under the action of $\mathbb S^1$,  there is $\zeta\in\mathbb S^1$ with $\varphi_{\zeta}(y)\in V'$. Since the set $\{\xi^{k_n}\colon n\in\mathbb N\}$ is dense in $\mathbb S^1$, there is $n\in\mathbb N$ with $\xi^{k_n}$ close enough to $\zeta$ so that $\varphi_{\xi^{k_n}}(y)\in V'$. Then $y\in(\varphi_{\xi^{k_n}})^{-1}(V')=(\varphi_{\xi})^{-k_n}(V')$. Thus, $\bigcup_{n=1}^{\infty}(\varphi_{\xi})^{-k_n}(V')=Y$, as was to be shown.

\emph{4th step.} Let $U\subseteq X$, $V'\subseteq Y$ be open sets. Let $x\in U$ and suppose that $V'$ intersects each orbit of $\mathbb S^1$ in $Y$. Let $\xi\in\mathbb S^1$ and assume that there is an increasing sequence of positive integers $(k_n)_{n=1}^{\infty}$ such that
\begin{itemize}
\item $T^{k_n}(x)\in U$ for every $n\in\mathbb N$, and
\item the set $\{\xi^{k_n}\colon n\in\mathbb N\}$ is dense in $\mathbb S^1$.
\end{itemize}
We show that $\varphi_{\xi}\in\mathcal H_{U,V'}\cap\mathcal H$.

First, we have $\varphi_{\xi}\in\mathcal H$ by definition of $\mathcal H$. To show that $\varphi_{\xi}\in\mathcal H_{U,V'}$, use the third step of the proof to find $N\in\mathbb N$ with $Y=\bigcup_{n=1}^N(\varphi_{\xi})^{-k_n}(V')$ and set $U'=\bigcap_{n=1}^NT^{-k_n}(U)$. Then $x\in U'$ and so $U'$ is a nonempty open subset of $X$. Thus, by minimality of $T$, $X=\bigcup_{m=1}^{\infty}T^{-m}(U')$. Consequently,
\begin{align*}
\bigcup_{n=1}^{\infty}(T\times\varphi_{\xi})^{-n}(U\times V')&\supseteq\bigcup_{m=1}^{\infty}(T\times\varphi_{\xi})^{-m}\left(\bigcup_{n=1}^N(T\times\varphi_{\xi})^{-k_n}(U\times V')\right)\\
&\supseteq\bigcup_{m=1}^{\infty}(T\times\varphi_{\xi})^{-m}\left(\bigcup_{n=1}^NU'\times(\varphi_{\xi})^{-k_n}(V')\right)\\
&=\bigcup_{m=1}^{\infty}(T\times\varphi_{\xi})^{-m}(U'\times Y)\\
&=\left(\bigcup_{m=1}^{\infty}T^{-m}(U')\right)\times Y=X\times Y,
\end{align*}
which shows that $\varphi_{\xi}\in\mathcal H_{U,V'}$ indeed.

\emph{5th step.} Let $U\subseteq X$, $V\subseteq Y$ be nonempty open sets and let 
$\psi\mathcal H_a(Y)$. Then 
\[
  \overline{\mathcal H_{U,\psi(V)}\cap\mathcal H}
  =\psi\overline{\mathcal H_{U,V}\cap\mathcal H}\psi^{-1}.
\]
Indeed, by the first step we have  
$\overline{\mathcal H_{U,\psi(V)}\cap\mathcal H} 
= \overline{\left(\psi\mathcal H_{U,V}\psi^{-1}\right)\cap\mathcal (\psi H\psi^{-1})}$.
Since the map $h\mapsto \psi h\psi^{-1}$ is a homeomorphism of $\mathcal H(Y)$, this set is equal to
$\overline{\psi\left(\mathcal H_{U,V}\cap\mathcal H\right)\psi^{-1}}
=\psi\overline{\mathcal H_{U,V}\cap\mathcal H}\psi^{-1}$.

\emph{6th step.} Let $U\subseteq X$, $V\subseteq Y$ be nonempty open sets. We show that $\varphi_z\in\overline{\mathcal H_{U,V}\cap\mathcal H}$ for every $z\in\mathbb S^1$.

Since $\tor(\mathbb S^1)$ is dense in $\mathbb S^1$, it is sufficient to verify that $\varphi_z\in\overline{\mathcal H_{U,V}\cap\mathcal H}$ for every torsion element $z\in\mathbb S^1$. So fix $z\in\tor(\mathbb S^1)$. By virtue of Lemma~\ref{L:aux.hit.inv}, there is a homeomorphism  $\psi\in\Homeo_a(Y)$ with the following properties:
\begin{itemize}
\item $\psi$ commutes with $\varphi_z$,
\item the set $\psi(V)$ intersects each orbit of $\mathbb S^1$ in $Y$.
\end{itemize}

Fix $x\in U$ and use minimality of $T$ to find an increasing sequence of positive integers $(k_n)_{n=1}^{\infty}$ with $T^{k_n}(x)\in U$ for every $n\in\mathbb N$. Denote by $A$ the set of all $\xi\in\mathbb S^1$ such that the sequence $(\xi^{k_n})_{n=1}^{\infty}$ is uniformly distributed in $\mathbb S^1$. By the fourth step of the proof (applied to $V'=\psi(V)$), we have $\varphi_{\xi}\in\mathcal H_{U,\psi(V)}\cap\mathcal H$ for every $\xi\in A$. Since the set $A$ is dense in $\mathbb S^1$ by virtue of Lemma~\ref{unif.distr}, we get $\varphi_z\in\overline{\mathcal H_{U,\psi(V)}\cap\mathcal H}$. By the fifth step of the proof, this means that
$\varphi_z\in\psi\overline{\mathcal H_{U,V}\cap\mathcal H}\psi^{-1}.$
Consequently, since $\psi$ commutes with $\varphi_z$,
\begin{equation*}
\varphi_z=\psi^{-1}\circ\varphi_z\circ\psi\in\overline{\mathcal H_{U,V}\cap\mathcal H},
\end{equation*}
as was to be shown.

\emph{7th step.} Given nonempty open sets $U\subseteq X$ and $V\subseteq Y$, set $\mathcal M_{U,V}=\mathcal H_{U,V}\cap\overline{\mathcal H}$. We show that the set $\mathcal M_{U,V}$ is open and dense in $\overline{\mathcal H}$.

The openness follows from the second step of the proof. We verify the density by showing that $\mathcal H\subseteq\overline{\mathcal M_{U,V}}$. To this end, fix $\psi\in\Homeo_a(Y)$ and $z\in\mathbb S^1$; we verify that $\psi^{-1}\circ\varphi_z\circ\psi\in\overline{\mathcal M_{U,V}}$. By using the sixth and the fifth steps of the proof, we obtain
\begin{equation*}
\varphi_z\in\overline{\mathcal H_{U,\psi(V)}\cap\mathcal H}=\psi\overline{\mathcal H_{U,V}\cap\mathcal H}\psi^{-1}\subseteq\psi\overline{\mathcal M_{U,V}}\psi^{-1},
\end{equation*}
whence it follows that $\psi^{-1}\circ\varphi_z\circ\psi\in\overline{\mathcal M_{U,V}}$ indeed.

\emph{8th step.} We prove the theorem.

Fix countable bases $(U_n)_{n=1}^{\infty}$ and $(V_m)_{m=1}^{\infty}$ of $X$ and $Y$, respectively,
and set $\mathcal M=\bigcap_{n,m=1}^{\infty}\mathcal M_{U_n,V_m}$. By the preceding step of the proof, all the sets $\mathcal M_{U_n,V_m}$ are open and dense in $\overline{\mathcal H}$ and so $\mathcal M$ is a dense $G_{\delta}$ subset of $\overline{\mathcal H}$. Since $\overline{\mathcal H}$ is a completely metrizable space, we infer that $\mathcal M\neq\emptyset$. Finally, since the product $T\times S$ is obviously minimal for every $S\in\mathcal M$ and $\mathcal M\subseteq\Homeo_a(Y)$, the proof is finished.
\end{proof}

\section{Smooth manifolds}\label{S:thmB-smooth-manifolds}

As we mentioned in the previous section, our Theorem~\ref{T:manif.homeo-PM} is based on (and its proof basically follows the line of that of) \cite[Th\'eor\`eme~1]{FH}. Of course, there are some essential distinctions in the proofs of these two theorems. This is caused, firstly, by the fact that we work in the topological category rather than in the smooth one and, secondly, by the fact that instead of skew products we consider direct products, so we need our minimal homeomorphism be disjoint from a given minimal map in the base. Although our interest in this paper is in the topological category, we deem it opportune to mention that smooth analogues of our Theorems~\ref{T:manif.homeo-PM} and~\ref{T:manif.homeo-PM.Lie} are true. In fact, we have the following theorem.

\begin{theorem}\label{T:manif.diff-PM.Lie}
Let $Y$ be a smooth compact connected manifold without boundary and $G$ be a nontrivial compact connected Lie group. Assume that $G$ acts smoothly and freely on $Y$. Then for every minimal system $(X,T)$ there is a smooth diffeomorphism $S\colon Y\to Y$ isotopic to the identity such that the product $(X\times Y,T\times S)$ is minimal.
\end{theorem}
\begin{proof}[Sketch of the proof]
As explained in the proof of Theorem~\ref{T:manif.homeo-PM.Lie}, we may restrict our attention to the case $G=\mathbb S^1$. To prove the theorem in this particular case, it is sufficient to combine the proofs of Theorem~\ref{T:manif.homeo-PM} and \cite[Th\'eor\`eme~1]{FH}. We shall therefore omit the details and present only a sketch of the proof.

Fix a minimal system $(X,T)$. Let $\mathcal D(Y)$ be the Polish group of all smooth diffeomorphisms $Y\to Y$ equipped with the $C^{\infty}$-topology and $\mathcal D_a(Y)$ be the open (hence closed, hence Polish) subgroup of $\mathcal D(Y)$ formed by the diffeomorphisms isotopic to the identity. Denoting the acting diffeomorphisms of $\mathbb S^1$ by $\varphi_z$ ($z\in\mathbb S^1$), we set
\begin{equation*}
\mathcal D=\left\{\psi^{-1}\circ\varphi_z\circ\psi\colon\psi\in\mathcal D_a(Y)\text{ and }z\in\mathbb S^1\right\}\subseteq\mathcal D_a(Y).
\end{equation*}
Further, for each pair of nonempty open sets $U\subseteq X$ and $V\subseteq Y$, write
\begin{equation*}
\mathcal D_{U,V}=\left\{\sigma\in\mathcal D_a(Y)\colon\bigcup_{n=1}^{\infty}(T\times\sigma)^{-n}(U\times V)=X\times Y\right\}.
\end{equation*}
Then all the sets $\mathcal D_{U,V}$ are open in $\mathcal D_a(Y)$. By using \cite[Corollaire~4.12]{FH} in place of our Lemma~\ref{L:aux.hit.inv}, we find that $\varphi_z\in\overline{\mathcal D_{U,V}\cap\mathcal D}$ for every $z\in\mathbb S^1$ and infer from this observation that all the sets $\mathcal M_{U,V}=\mathcal D_{U,V}\cap\overline{\mathcal D}$ are open and dense in $\overline{\mathcal D}\subseteq\mathcal D_a(Y)$. Now fix countable bases $(U_n)_{n=1}^{\infty}$ and $(V_m)_{m=1}^{\infty}$ of $X$ and $Y$, respectively, and set $\mathcal M=\bigcap_{n,m=1}^{\infty}\mathcal M_{U_n,V_m}$. Being a dense $G_{\delta}$ subset of the completely metrizable space $\overline{\mathcal D}$, the set $\mathcal M$ is nonempty. We may therefore finish the proof by choosing $S\in\mathcal M$.
\end{proof}


\section{Examples and counterexamples related to minimal direct products}\label{S:examples}

In this section we discuss examples mentioned in the introduction.

\begin{example}\label{e:1}
	There exist nonminimal continua $X,Z$ such that
	the product $X\times Z$ is minimal.
\end{example}
\begin{proof}
	We show that the cylinder $X=\mathbb S^1\times I$ and
	the Hilbert cube $Z=\mathbb H$ satisfy the required conditions.
	
	First, by \cite{BOT}, the only compact connected two-manifolds (with
	or without boundary) which admit a minimal map are torus
	$\mathbb T^2$ and Klein bottle $\mathbb K^2$. Hence the
	cylinder $\mathbb S^1\times I$ is a nonminimal space. Further, the Hilbert cube has the fixed point property and so it is also a nonminimal space. Finally, the product $\left(\mathbb S^1\times I\right)\times\mathbb H$
	admits a minimal homeomorphism by virtue of \cite[p.~323]{GW}.
\end{proof}

\begin{example}\label{e:2}
	There exists a minimal compact metrizable space $X$ such that
	the product $X\times X$ admits a minimal skew product but
	does not admit any minimal direct product.
\end{example}
\begin{proof}
    It is sufficient to take $X$ consisting of two points. As a less trivial example,
	let $X=\mathbb S^1_0\cup \mathbb S^1_1$ be a union of two disjoint circles.
	Then $X$ admits a minimal map. Moreover, every minimal (in fact,
	every surjective) self-map of $X$ permutes the two
	building circles of $X$. Therefore the product $T\times S$
	of any pair of minimal maps $T,S$ on $X$ has
	$(\mathbb S^1_0\times\mathbb S^1_0)\cup (\mathbb S^1_1\times\mathbb S^1_1)$
	as an invariant subset. Consequently, such $T\times S$ can not
	be minimal.
	
	We show that the product $X\times X$ admits a minimal
	skew-product $F=(T,S_x)$. To simplify the
	description of $F$ let $\mathbb S^1_0=\{0\}\times\mathbb S^1$
	and $\mathbb S^1_1=\{1\}\times\mathbb S^1$. Now take
	a rotation $R_{\alpha}$ of $\mathbb S^1$ by an irrational
	$\alpha$ and let
	$T:X\to X$ be given by $T(0,z)=(1,z)$ and $T(1,z)=
	(0,R_\alpha(z))$ for $z\in\mathbb S^1$. Obviously, $T$ is
	a minimal map. Further, take a real number $\beta$
	such that $1,\alpha,\beta$ are linearly independent
	over the rationals and set $S_x(i,\xi)=(i,\xi)$
	for $x\in\mathbb S^1_0$, $\xi\in\mathbb S^1$ and
	$i=0,1$, and $S_x(i,\xi)=(1-i,R_\beta(\xi))$ for
	$x\in\mathbb S^1_1$, $\xi\in\mathbb S^1$ and $i=0,1$.
	One checks easily that $F=(T,S_x)$ is minimal.
\end{proof}

\begin{example}\label{e:3}
	There exist minimal continua $X,Z$ such that $X\times Z$
	admits a minimal direct product and every skew-product on $X\times Z$ is a direct product.
\end{example}
\begin{proof}
	Let $X=\mathbb S^1$ be the circle and let $Z$ be the
	pseudo-circle. Clearly, $X$ is a minimal space. It
	follows from \cite{Han} that $Z$ is also a minimal space.
	The product $X\times Z$ thus admits a minimal direct
	product by Theorem~\ref{T:groups}. Now we show that
	every skew product $F=(T,S_x)$ on $X\times Z$
	is in fact a direct product. To see this observe that
	the components of arc-wise connectedness of $Z$ are
	the singletons. Therefore, the components of arc-wise
	connectedness of $C(Z)$ are also singletons. Since $X$
	is arc-wise connected, 
	the set $\{S_x\colon x\in X\}$ is arc-wise connected in $C(Z)$. Hence
	all the maps $S_x\in C(Z)$ are equal. It follows that $F$
	is a direct product.
\end{proof}

\begin{remark}\label{R:e:3}
	In the previous example, if we do not insist that $Z$ be a continuum, one can use the Cantor set
	instead and apply a similar argument using
	connectedness instead of arc-wise connectedness.
\end{remark}

\begin{remark}
	Although the product space $X\times Z$ from Example~
	\ref{e:3} or Remark~\ref{R:e:3} does not admit a minimal skew-product which is not a direct product,
	one can show that the space $Z\times X$
	\emph{does} admit a skew-product $F=(g,f_z)$ which is not
	a direct product (this follows from \cite[Theorem~1]{GW}).
\end{remark}

\begin{example}\label{e:factorwise-rigid}
	There exist continua $X,Y$ such that all the spaces $X,Y,X\times Y$ admit minimal homeomorphisms and every homeomorphism on $X\times Y$ takes the form of a direct product.
\end{example}

\begin{proof}
	Let $Y$ be an arbitrary continuum supporting a minimal homeomorphism $S$ and not containing any arc; say, $Y$ can be the pseudo-circle. Let $Z$ be a solenoid. By Theorem~\ref{T:groups}, there is a homeomorphism (in fact, a rotation) $R\colon Z\to Z$ such that the product $(Z\times Y,R\times S)$ is minimal. Fix a DST space $X$ derived from $(Z,R)$ (for details, see Section~\ref{Sec:the.Slovak}) and let $T\colon X\to X$ be the homeomorphism, which is an almost 1-1 extension of $R$. The product $T\times S$ is a minimal homeomorphism, being an almost 1-1 extension of the minimal homeomorphism $R\times S$. We show that each homeomorphism on $X\times Y$ has the form of a direct product.
	
	So let $F\colon X\times Y\to X\times Y$ be a homeomorphism. Recall that $X$ is not path-connected, but it contains a dense path-component $\alpha$. Given $y\in Y$, the set $F(\alpha\times\{y\})\subseteq X\times Y$ is path-connected. Since $Y$ has degenerate path-components, it follows that there is $y'\in Y$ with $F(\alpha\times\{y\})\subseteq X\times\{y'\}$. Since $\alpha$ is dense in $X$, we infer that $F(X\times\{y\})\subseteq X\times\{y'\}$. Consequently, $F$ is a skew product over $Y$. That is, there exist a homeomorphism $h\colon Y\to Y$ and a family of homeomorphisms $g_y\colon X\to X$ with $F(x,y)=(g_y(x),h(y))$ for all $x\in X$ and $y\in Y$. Recall that $Y$ is connected, $\Homeo(X)$ is discrete and $g_y\in\mathcal{H}(X)$ depend continuously on $y\in Y$. Hence all the maps $g_y$ coincide and so $F$ is indeed a direct product on $X\times Y$.
\end{proof}

\begin{example}\label{E:skew.direct.fiber}
	There exists a compact connected manifold $Y$ such that for every compact minimal system $(X,T)$, the product $X\times Y$ admits a minimal skew product $(T,g_x)$ and admits a minimal direct product $T\times S$, but all the direct products $T\times g_x$ ($x\in X$) are nonminimal.
\end{example}
\begin{proof}
	Let $Y=\mathbb S^3$. Recall that $Y$ carries the structure of a (compact connected) nonabelian Lie group. Let $G$ be the group of the rotations on $Y$. Then, by \cite[p.~323]{GW}, $X\times Y$ admits a minimal skew product $(T,g_x)$ with $g_x\in G$ for every $x\in X$ and it admits a minimal direct product $T\times S$ by virtue of Theorem~B(9). However, since $Y$ is nonabelian, all the rotations $g\in G$ are nonminimal (if $g$ were minimal, then the full orbit $\{g^n(e)\colon n\in\mathbb Z\}$ of the neutral element $e$ of $Y$ would be a dense abelian subgroup of $Y$, hence also $Y$ would be abelian) and hence the products $T\times g$ are not minimal.
\end{proof}


\section{DST spaces}\label{Sec:the.Slovak}

As mentioned in the introduction, those Slovak spaces 
which have been constructed in \cite[Section~4]{DST}
are said to be DST spaces. Fix such a space $X$.
To show that $X$ can serve as a counterexample required by
Theorem~A(1), we are going first to describe its topological structure.
We also introduce notation which will be used throughout 
the rest of this  paper.\footnote{We warn the reader that our notation differs from  that used in \cite{DST}.
	Instead of the notation $X,\tilde{F}, T, \tilde{T}$ used in \cite{DST}
	we are going to write $X_d,X,T_d,T$, respectively. Here the lower index $d$
	can be read as ``down'', since the system $(X_d,T_d)$ will be a factor of
	$(X,T)$.}

\subsection{Description of the topology of the space $X$}\label{SS:Slovak}
We are going to describe some properties of the DST space $X$
(for more details the reader is referred to \cite{DST}).

First basic fact is that $X$ is a subset of
$X_d\times[0,1]$, where $X_d$ is a
\emph{generalized solenoid}
\begin{equation}\label{EQ:Td}
X_d = (C\times [0,1]) / _{(y,1)\sim (h(y),0)},
\end{equation}
with $C$ being a Cantor set and $h:C\to C$ being a minimal homeomorphism.
The continuum $X_d$ has uncountably many composants, each of them is dense in $X_d$ and is
a continuous injective image of the real line.
The only nondegenerate proper subcontinua of $X_d$ 
are arcs.\footnote{The solenoids, as well as the circle, 
	are compact connected metrizable abelian, hence monothetic, groups. 
	On the other hand, by \cite{Ha}, 
	a nondegenerate continuum is a solenoid if and only if it is indecomposable, 
	homogeneous and all of its proper subcontinua are arcs.
	Therefore, if a generalized solenoid is not a solenoid, then it is not homogeneous
	and so it is not a topological group.\label{ftn:generalized-solenoid}}

The DST space $X$ is the closure of the graph of a (discontinuous) function
from $X_d$ to $[0,1]$. Denote by $\pi\colon X\to X_d$ the natural projection.
It is an almost $1$-$1$ map and the only nondegenerate point inverses
are arcs $W_n$ ($n\in\ZZZ$), with
\begin{equation}\label{EQ:Wn-diam}
\lim_{n\to\pm\infty}\diam(W_n) =  0.
\end{equation}

The space $X$ has uncountably many composants, each of them being dense in $X$.
A single one of them, denote it by $\gamma$,
is not path connected. Its path components are $C_n$ ($n\in\ZZZ$), where
each $C_n$ is homeomorphic to the graph of
$\sin(1/x)$, $x\in(0,1]$;
see Figure~\ref{Fig:gamma}.
Moreover, for every $n$, we have
\begin{equation}\label{EQ:Cn-closure}
\overline{C}_n = C_n\sqcup W_{n+1} \subseteq C_n\sqcup C_{n+1},
\qquad
\overline{C}_n \cap \overline{C}_{n+1} = W_{n+1},
\end{equation}
where $\sqcup$ denotes the disjoint union
and $\overline{C}_n$ stands for the closure of the set ${C_n}$ in $X$
(note that $\overline{C}_n\cap \overline{C}_m$ is nonempty if and only if
$\abs{m-n}\le 1$).
The family of all the other composants of $X$ will be denoted by $\AAa$;
every composant $\alpha\in\AAa$ is a continuous injective image of the real line.

\begin{figure}[ht!]
	\includegraphics{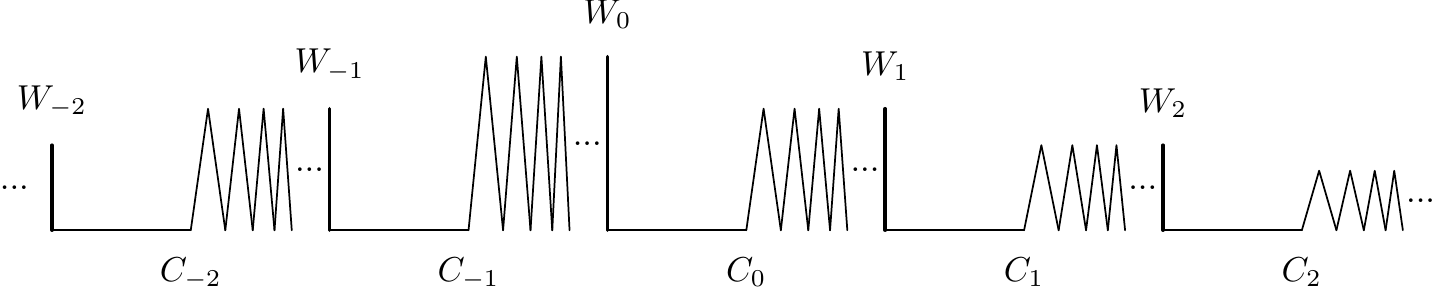}
	\caption{The composant $\gamma$}
	\label{Fig:gamma}
\end{figure}

There are special minimal homeomorphisms
$T_d\colon X_d\to X_d$ and $T\colon X\to X$ such that
$\pi\colon (X,T)\to (X_d,T_d)$ is a factor map (i.e., $\pi\circ T = T_d\circ \pi$)
and for every $n\in\ZZZ$ it holds that
\begin{equation}\label{EQ:F(Cn)}
T(C_n)=C_{n+1},\qquad
T(W_n)=W_{n+1}.
\end{equation}
Since both $T$ and $T^{-1}$ are minimal,
in view of \eqref{EQ:F(Cn)} we have that for every $n_0\in\ZZZ$, the disjoint unions
\begin{equation}\label{EQ:Wn-union-dense}
\bigsqcup_{n\ge n_0} W_{n}
\quad\text{and}\quad
\bigsqcup_{n\le n_0} W_{n}
\qquad\text{are dense in}\quad
X.
\end{equation}
Now we describe how $T_d$ has been constructed in \cite{DST}.
Start with
the suspension flow $\phi=(\phi_t)_{t\in\RRR}$ on $X_d$; i.e.
$\phi_t\colon X_d\to X_d$ is defined by
\begin{equation}\label{EQ:phit}
\phi_t(y,s) =
\left(
h^{\lfloor t+s \rfloor}(y), \{t+s\}
\right),
\end{equation}
where $\lfloor \cdot \rfloor$ and $\{\cdot\}$ denote the integer and the fractional part
of a real number, respectively.
Then $T_d=\phi_{t_0}$, where $t_0\ne 0$ is such that $\phi_{t_0}$
is a minimal homeomorphism on $X_d$;
since the suspension flow $\phi$ is minimal, such $t_0$ does exist
(see, e.g.,~\cite{Eg,Fa}).

\subsection{Some facts on $X\times X$}\label{SS:XxX}

For pairs of integers $m,n$
put
\begin{equation}\label{EQ:Cmn}
C_{m,n}=C_m\times C_n
\quad\text{and}\quad
W_{m,n}=W_m\times W_n.
\end{equation}
Notice that $C_{m,n}$ is homeomorphic to a (closed) quadrant of the plane and
$W_{m,n}$ is homeomorphic to the square.

The path components of $X\times X$, being products of the path components of $X$, are
of four types:
\begin{itemize}
	\item $\alpha\times\beta$, where $\alpha,\beta\in\AAa$;
	\item $C_m\times\alpha$, where $m\in\ZZZ$ and $\alpha\in\AAa$;
	\item $\alpha\times C_n$, where $n\in\ZZZ$ and $\alpha\in\AAa$;
	\item $C_m\times C_n$, where $m,n\in\ZZZ$.
\end{itemize}
Note that every path component of the first type is dense in $X\times X$,
while all the other path components are nowhere dense in $X\times X$.

Given maps $f,g\colon Y\to Z$, a point $y\in Y$ is called a \emph{point of coincidence} of the pair $f,g$ if $f(y)=g(y)$. We write
$\Coin(f,g)$ for the set of all points of coincidence of the pair $f,g$.

\begin{lemma}\label{L:point-of-coincidence}
	Let $Y$ and $Z$ be topological spaces, $Y$ having the fixed point property.
	Let $f, g\colon Y\to Z$ be continuous maps, $g$ being a homeomorphism.
	Then $f$ and $g$ have a point of coincidence in $Y$.
\end{lemma}

\begin{proof}
	The map $g^{-1} \circ f\colon Y\to Y$ is continuous and so it has a
	fixed point $y_0$. Thus $(g^{-1} \circ f) (y_0) = y_0$, whence
	$f(y_0) = g (y_0)$.
\end{proof}

For integers $a,b$ put $T_d^{a\times b} = T_d^a\times T_d^b$ and $T^{a\times b} = T^a\times T^b$.

\begin{lemma}\label{L:coin-F-Tab}
	Let $F\colon X\times X \to X\times X$ be a continuous map. If $m,n$ and $a,b$ are integers such that
	$F({W}_{m,n})\subseteq {W}_{m+a,n+b}$ then $\Coin(F, T^{a\times b})\cap {W}_{m,n} \ne\emptyset$.
\end{lemma}

\begin{proof}
	The space $W_{m,n}$, being homeomorphic to the square, has the fixed point property. Since 
	$T^{a\times b}({W}_{m,n})= {W}_{m+a,n+b}$ and $T^{a\times b}$ is a homeomorphism, it is sufficient to use 
	Lemma~\ref{L:point-of-coincidence}.
\end{proof}

\begin{remark}\label{R:Dyer}
	It is of some interest to mention that, analogously, $\Coin(F, T^{a\times b})\cap \overline{C}_{m,n} \ne\emptyset$,
	provided that $F({\overline{C}}_{m,n})\subseteq {\overline{C}}_{m+a,n+b}$. Indeed, the closure $\overline{C}_{m,n}$,
	being the product of arc-like continua $\overline{C}_m$ and $\overline{C}_n$, has the fixed point property by~\cite{Dy}.	
\end{remark}

\begin{lemma}\label{L:coincidence-of-maps}
	Let $Y,Z$ be compact metric spaces and $\varphi, \tau \colon Y\to Z$ be continuous maps.
	Assume that there are nowhere dense sets $A_n\subseteq Y$ ($n\in\NNN$)
	with diameters converging to zero
	such that their union is dense in $Y$
	and, for every $n\in\mathbb N$, $A_n\cap\Coin(\varphi,\tau)\ne\emptyset$. Then $\varphi=\tau$.
\end{lemma}

\begin{proof}
It is sufficient to show that each nonempty open set in $Y$ contains one of the sets $A_n$. Indeed, if this is the case then the set $\Coin(\varphi,\tau)$ is dense in $Y$, which means that $\varphi=\tau$.

So fix an open ball $B(y,\varepsilon)$ in $Y$ and choose $n_0\in\mathbb N$ so that $\diam(A_n)<\varepsilon/2$ for every $n> n_0$. Since $A_1,\dots,A_{n_0}$ are nowhere dense in $Y$, the open set $B(y,\varepsilon/2)\setminus\bigcup_{i=1}^{n_0}\overline{A_i}$ is nonempty, hence it intersects $A_n$ for some $n>n_0$. Then $A_n\subseteq B(y,\varepsilon)$, as was to be shown.
\end{proof}

Since $\pi\colon (X,T)\to (X_d,T_d)$ is an almost~$1$-$1$ factor map, we have the following lemma.

\begin{lemma}\label{L:pi-x-pi}
	For all $a,b\in\ZZZ$,
	$\pi\times\pi\colon (X\times X, T^{a\times b})\to(X_d\times X_d, T_d^{a\times b})$
	is an almost~$1$-$1$ factor map.
\end{lemma}

\subsection{Direct products $T^{a\times b}\colon X\times X \to X\times X$
	are not minimal}

Recall that our aim is to show that $X\times X$ does not admit a minimal
map.\footnote{Note that the space $X_d\times X_d$, though it is in general not a 
	topological group, does admit a minimal homeomorphism. In fact, $X_d$ admits a minimal 
	continuous flow defined by \eqref{EQ:phit}, and so  $X_d\times X_d$ also admits a 
	minimal continuous flow by 
	\cite[Theorem~25]{Di}. By passing to an appropriate time $t$-map,  we get a minimal 
	homeomorphism on $X_d\times X_d$ by~\cite{Eg,Fa}.}
In this subsection we show that the particular homeomorphisms $T^{a\times b}$ on $X\times X$
are not minimal.

\begin{lemma}\label{L:Td-ab-nonminimal}
	The map $T_d^{a\times b}$ is not minimal on $X_d\times X_d$ for any $a,b\in\ZZZ$.
\end{lemma}
\begin{proof}
	It follows from \eqref{EQ:phit} that the map $T_d=\phi_{t_0}\colon X_d\to X_d$ is
	an extension of the rotation $s\mapsto t_0+s$ of  circle $\RRR/\ZZZ$,
	the corresponding factor map being the projection onto the second coordinate.
	Consequently, the map $T_d^{a\times b}$ is not minimal, since its factor
	$(s,s')\mapsto (at_0+s,bt_0+s')$ on the torus is not minimal.
\end{proof}

\begin{proposition}\label{P:T-ab-nonminimal}
	The map $T^{a\times b}$ is not minimal on $X\times X$ for any $a,b\in\ZZZ$.
\end{proposition}
\begin{proof}
	Since minimality is preserved by passing to factors, it suffices to use
	Lemmas~\ref{L:pi-x-pi} and \ref{L:Td-ab-nonminimal}.
\end{proof}


\section{Prelude to the proof of Theorem~A: Homeomorphisms on $X\times X$}

We are going to prove a part of Theorem~A(1), namely that if $X$ is a DST space then
$X\times X$ does not admit minimal
homeomorphisms; this is done in Theorem~\ref{T:TA.homeo} below. This section could be omitted because in later sections we prove Theorem~A(1) 
in full generality, without referring to Theorem~\ref{T:TA.homeo}.
However, we have three reasons for considering first the special case of homeomorphisms. 
First, this particular case is much shorter than the general one. 
Second, we want to illustrate some of the methods we shall use in the proof of Theorem~A(1). 
Third, in order to prove the nonexistence of minimal homeomorphisms on $X\times X$, 
we describe in full details the elements of the homeomorphism group of $X\times X$,
which is a result of independent interest.

\begin{theorem}\label{T:TA.homeo}
	Let $X$ be a DST space and $T\colon X\to X$ be the minimal homeomorphism from Subsection~\ref{SS:Slovak}. Given a homeomorphism $F\colon X\times X\to X\times X$, there exist $a,b\in\mathbb Z$ such that
	\begin{equation}\label{EQ:homeo}
	F=T^{a\times b}
	\quad\text{or}\quad
	F=R\circ T^{b\times a},
	\end{equation}
	where $R$ is the reflection sending $(x_1,x_2)$ to $(x_2,x_1)$. Consequently, $F$ is not minimal.
\end{theorem}
\begin{remark}\label{R:TA.homeo}
Notice that for all $a,b\in\mathbb Z$, $T^{a\times b}\circ R=R\circ T^{b\times a}$, hence the group $\Homeo(X\times X)$ is not abelian. In fact, $\Homeo(X\times X)$ is isomorphic to a semi-direct product
\begin{equation}\label{Eq:H.iso.gp}
\mathcal H(X\times X)\cong\mathbb Z_2\ltimes\mathbb Z^2.
\end{equation}
Indeed, $N=\{T^{a\times b}\colon a,b\in\mathbb Z\}$ is a normal subgroup of $\mathcal H(X\times X)$ isomorphic to $\mathbb Z^2$, $H=\{\Id_{X\times X},R\}$ is a subgroup of $\mathcal H(X\times X)$ isomorphic to $\mathbb Z_2$, $H\cap N=\{\Id_{X\times X}\}$ and $\mathcal H(X\times X)=HN$. These facts together yield the isomorphism \eqref{Eq:H.iso.gp}.
\end{remark}
\begin{proof}
	The nonminimality of $F$ follows from the first statement of the theorem. Indeed, if $F=T^{a\times b}$ or $F=R\circ T^{b\times a}$ for some $a,b\in\mathbb Z$ then $F^2=T^{c\times d}$ for some $c,d\in\mathbb Z$ and, 
	by Proposition~\ref{P:T-ab-nonminimal},
	$F^2$ is not minimal. Consequently, $F$ is not minimal 
	by connectedness of $X\times X$.

	Recall that the space $X\times X$ has path components of four types:
	\begin{equation*}
	\alpha\times\beta,\ C_m\times\alpha,\ \alpha\times C_n,\ C_m\times C_n
	\end{equation*}
	with $\alpha,\beta\in\mathcal A$ and $m,n\in\mathbb Z$, and the closures of these path components are, in respective order,
	\begin{equation*}
	X\times X,\ \overline{C}_m\times X,\ X\times\overline{C}_n,\ \overline{C}_m\times\overline{C}_n.
	\end{equation*}
	We also recall the notation $C_{m,n}=C_m\times C_n$ from \eqref{EQ:Cmn}.
	
	It follows from our discussion in Remark~\ref{R:Dyer} that the path components of the fourth type, $C_{m,n}$ ($m,n\in\mathbb Z$), are the only ones with closures possessing the fixed point property. Since $F$ is a homeomorphism, it follows that there is a bijection $\varphi$ on $\mathbb Z\times\mathbb Z$ such that
	\begin{equation}\label{EQ:phi-homeo}
	F(C_{m,n})=C_{\varphi(m,n)}
	\qquad\text{for all }m,n\in\mathbb Z.
	\end{equation}
	We also notice that $F(C_{m,n})=C_{m',n'}$ is equivalent to $F(\overline{C}_{m,n})=\overline{C}_{m',n'}$.
	
	Now $F$ induces a permutation on the collection of the path components of $X\times X$ of the second and the third type; in particular, for all $m\in\mathbb Z$ and $\alpha\in\mathcal A$ there is a unique pair $n\in\mathbb Z$, $\beta\in\mathcal A$ such that $F(C_m\times\alpha)=C_n\times\beta$ or $F(C_m\times\alpha)=\beta\times C_n$. In the first case we have $F(\overline{C}_m\times X)=\overline{C}_n\times X$ and, in the second case, $F(\overline{C}_m\times X)=X\times\overline{C}_n$. This observation implies that $n$ depends only on $m$ and not on $\alpha$, allowing us to define a function $\psi\colon\mathbb Z\ni m\mapsto n\in\mathbb Z$. We claim that $F$ either preserves the type or reverses it; to be precise, we assert that there are the following two possibilities.
	\begin{enumerate}
		\item[(1)] For each pair $m\in\mathbb Z$, $\alpha\in\mathcal A$ there is $\beta\in\mathcal A$ with $F(C_m\times\alpha)=C_{\psi(m)}\times\beta$.
		\item[(2)] For each pair $m\in\mathbb Z$, $\alpha\in\mathcal A$ there is $\beta\in\mathcal A$ with $F(C_m\times\alpha)=\beta\times C_{\psi(m)}$.
	\end{enumerate}
	To see this, assume, on the contrary, that $F(C_m\times\alpha)=C_{\psi(m)}\times\beta$ and $F(C_{m'}\times\alpha')=\beta'\times C_{\psi(m')}$ for some $m,m'\in\mathbb Z$ and $\alpha,\beta,\alpha',\beta'\in\mathcal A$. Then, by passing to closures, we obtain
	\begin{equation*}
	\begin{split}
	F((\overline{C}_m\cap\overline{C}_{m'})\times X)&=F(\overline{C}_m\times X)\cap F(\overline{C}_{m'}\times X)=(\overline{C}_{\psi(m)}\times X)\cap(X\times\overline{C}_{\psi(m')})=\overline{C}_{(\psi(m),\psi(m'))}\\
	&=F(\overline{C}_{\varphi^{-1}(\psi(m),\psi(m'))}),
	\end{split}
	\end{equation*}
	which is in contradiction with the injectivity of $F$.
	
	We now handle cases (1) and (2) separately.
	
	{\bf Case (1).} Now the map $F$ induces a permutation on the family $C_m\times\alpha$ ($m\in\mathbb Z$, $\alpha\in\mathcal A$). Since $F(\overline{C}_m\times X)=\overline{C}_{\psi(m)}\times X$ for every $m\in\mathbb N$, the injectivity of $F$ yields that $\psi$ is injective. Given $m\in\mathbb Z$ and $\alpha\in\mathcal A$, we have $\beta\in\mathcal A$ with $F(C_{m+1}\times\alpha)=C_{\psi(m+1)}\times\beta$ and $F(\overline{C}_m\times X)=\overline{C}_{\psi(m)}\times X$. Since the sets $C_{m+1}\times\alpha$ and $\overline{C}_m\times X$ intersect, their intersection being the product $W_{m+1}\times\alpha$, it follows that the sets $C_{\psi(m+1)}\times\beta$ and $\overline{C}_{\psi(m)}\times X$ also intersect. Consequently, $C_{\psi(m+1)}\cap\overline{C}_{\psi(m)}\neq\emptyset$, which yields $\psi(m+1)=\psi(m)+1$ by injectivity of $\psi$. We thereby conclude by finding $a\in\mathbb Z$ such that $\psi(m)=m+a$ for every $m\in\mathbb Z$. 
	
	Further, $F$ induces a permutation on the family $\alpha\times C_n$ ($\alpha\in\mathcal A$, $n\in\mathbb Z$). By applying the same argument as in the preceding paragraph, we find $b\in\mathbb Z$ such that for each pair $\alpha\in\mathcal A$, $n\in\mathbb Z$, there is $\beta\in\mathcal A$ with $F(\alpha\times C_n)=\beta\times C_{n+b}$. Consequently, $F(X\times\overline{C}_n)=X\times\overline{C}_{n+b}$ for every $n\in\mathbb Z$.
	
	Combining our results from the preceding two paragraphs with the definition of the map $\varphi$, we obtain, for all $m,n\in\mathbb Z$,
	\begin{equation*}
	\begin{split}
	\overline{C}_{\varphi(m,n)}&=F(\overline{C}_{m,n})=F((\overline{C}_m\times X)\cap (X\times\overline{C}_n))=F(\overline{C}_m\times X)\cap F(X\times\overline{C}_n)\\
	&=(\overline{C}_{m+a}\times X)\cap(X\times\overline{C}_{n+b})=\overline{C}_{m+a,n+b}.
	\end{split}
	\end{equation*}
	Thus, $\varphi(m,n)=(m+a,n+b)$ and so $F(C_{m,n})=C_{m+a,n+b}$ for all $m,n\in\mathbb Z$.
	
	Now let $m,n\in\mathbb Z$. Then
	\begin{equation*}
	\begin{split}
	F(W_m\times X)&=F((\overline{C}_{m-1}\times X)\cap(\overline{C}_m\times X))=F(\overline{C}_{m-1}\times X)\cap F(\overline{C}_m\times X)\\
	&=(\overline{C}_{m+a-1}\times X)\cap(\overline{C}_{m+a}\times X)=W_{m+a}\times X
	\end{split}
	\end{equation*}
	and, similarly, $F(X\times W_n)=X\times W_{n+b}$. Consequently, with the notation $W_{m,n}=W_m\times W_n$ from \eqref{EQ:Cmn},
	\begin{equation}\label{Eq:F.W.m.n}
	F(W_{m,n})=W_{m+a,n+b}
	\qquad\text{ for all }m,n\in\mathbb Z.
	\end{equation}
	Now, since $X\times X$ has a countable basis, it follows from \eqref{EQ:Wn-union-dense} that there is a sequence $((m_i,n_i))_{i\geq 1}$ in $\ZZZ^2$ such that the union of $W_{m_i,n_i}$ is dense in $X\times X$ and the diameters of  $W_{m_i,n_i}$ converge to zero. Since the sets $W_{m_i,n_i}$ are nowhere dense in $X\times X$, we may use \eqref{Eq:F.W.m.n} and apply Lemmas~\ref{L:coin-F-Tab} and \ref{L:coincidence-of-maps} to obtain $F=T^{a\times b}$.
	
	{\bf Case (2).} Now the homeomorphism $R\circ F$ fits into case (1). Thus, $R\circ F=T^{b\times a}$ for some $a,b\in\mathbb Z$, which yields $F=R\circ T^{b\times a}$.
\end{proof}

\begin{remark}[Extension of Theorem~\ref{T:TA.homeo} to $X^N$]\label{R:homeo-case-XN}
	The method used in the proof of Theorem~\ref{T:TA.homeo} can be directly generalized to prove that none of the cubes $X^N$ ($N\ge 3$) admits a minimal homeomorphism. As a matter of fact, if $F$ is a homeomorphism on $X^N$ then an analogue of \eqref{EQ:homeo} holds true; that is, either $F$ is a direct product of iterates of $T$ or else it is a composition of such a product with a nontrivial permutation of coordinates on $X^N$. Consequently, an appropriate iterate of $F$ is a direct product of iterates of $T$, hence it is not minimal
	(otherwise, its factor onto the first two coordinates would also be minimal, contradicting Proposition~\ref{P:T-ab-nonminimal}).
	By connectedness of $X^N$ it follows that neither $F$ is minimal. Let us also mention that an analogue of the isomorphism \eqref{Eq:H.iso.gp} from Remark~\ref{R:TA.homeo} is true. In this case the group $\Homeo(X^N)$ is isomorphic to a semi-direct product $S_N\ltimes\mathbb Z^N$, where $S_N$ stands for the symmetric group of the set $\{1,\dots,N\}$.
\end{remark}


\section{Towards the proof of Theorem~A: Continuous surjections on $X\times X$}\label{SS:XxX-surjection}

Now we turn to the study of continuous surjective maps on $X\times X$. The results obtained in this section will be used in our proof of Theorem~A(1) in Section~\ref{S:XxX-minimal}. We would like to bring the reader's attention to the fact that the methods of this section can be used to obtain analogous results for continuous surjective maps on spaces of the form $X\times X\times Y$, where $Y$ is an arbitrary path connected compact metrizable space. In fact, the statements and the proofs of all the results in this section remain true, verbatim, after adding a factor $Y$ to all the subsets of $X\times X$. (This observation will be useful to us in the proof of Theorem~\ref{T:XxXxY} in Section~\ref{Sec:other.min}.) However, for the sake of simplicity of the notation, we will formulate and prove the results of this section for continuous surjective maps on $X\times X$.

Throughout this section we assume that $F\colon X\times X \to X\times X$ is
a fixed continuous surjective map. We keep the notation introduced in Section~\ref{Sec:the.Slovak}. Let us also recall, see Subsection~\ref{SS:XxX}, that the space $X\times X$ has path components of four types:
\begin{equation*}
\alpha\times\beta,\ C_m\times\alpha,\ \alpha\times C_n,\ C_m\times C_n
\end{equation*}
with $\alpha,\beta\in\mathcal A$ and $m,n\in\mathbb Z$.

\begin{lemma}\label{L:F(alphaxbeta)}
	For all $\alpha,\beta\in\AAa$ there are $\alpha',\beta'\in\AAa$ such that
	\begin{equation*}
	F(\alpha\times\beta) \subseteq \alpha'\times\beta'.
	\end{equation*}
Consequently, $F((X\setminus\gamma)\times(X\setminus\gamma))\subseteq(X\setminus\gamma)\times(X\setminus\gamma)$.
\end{lemma}

\begin{proof}
	To verify the first statement, observe that a continuous image of a path connected set is path connected, and
	every continuous surjection
	maps dense sets onto dense sets.
	Since only the path components of the first type are dense,
	such $\alpha'$ and $\beta'$ necessarily exist.
	
	The second statement follows from the first one, since $\bigcup_{\alpha,\beta\in\mathcal A}\alpha\times\beta=(X\setminus\gamma)\times(X\setminus\gamma)$.
\end{proof}

Put
\begin{equation}\label{EQ:Dij}
\begin{split}
	D_{11} &=\   \{m\in\ZZZ\colon\  F(C_m\times \alpha)\subseteq {C}_k\times X \text{ for some }k\in\ZZZ \text{ and } \alpha\in\AAa\},
	\\
	D_{12} &=\   \{m\in\ZZZ\colon\  F(C_m\times \alpha)\subseteq X\times {C}_k \text{ for some }k\in\ZZZ \text{ and } \alpha\in\AAa\},
	\\
	D_{21} &=\   \{m\in\ZZZ\colon\ F(\alpha\times C_m)\subseteq {C}_k\times X \text{ for some }k\in\ZZZ \text{ and } \alpha\in\AAa\},
	\\
	D_{22} &=\   \{m\in\ZZZ\colon\  F(\alpha\times C_m)\subseteq X\times {C}_k \text{ for some }k\in\ZZZ \text{ and } \alpha\in\AAa\}.
\end{split}
\end{equation}

We are going to show that if $m\in D_{ij}$ then the corresponding $k\in\ZZZ$ does not depend on $\alpha$ and so it is unique.

\begin{lemma}\label{L:F(Cmxalpha)}
	Let $m,k\in\ZZZ$ and $\alpha\in\AAa$.
	\begin{enumerate}
		\item\label{IT:L:F(Cmxalpha):1}
		If $F(C_m\times \alpha)\subseteq C_k\times X$,
		then $F(C_m\times \beta)\subseteq C_k\times X$ for every $\beta\in\AAa$.
		\item\label{IT:L:F(Cmxalpha):2}
		If $F(C_m\times \alpha)\subseteq X\times C_k$,
		then $F(C_m\times \beta)\subseteq X\times C_k$ for every $\beta\in\AAa$.
		\item\label{IT:L:F(Cmxalpha):3}
		If $F(\alpha\times C_m)\subseteq C_k\times X$,
		then $F(\beta\times C_m)\subseteq C_k\times X$ for every $\beta\in\AAa$.
		\item\label{IT:L:F(Cmxalpha):4}
		If $F(\alpha\times C_m)\subseteq X\times C_k$,
		then $F(\beta\times C_m)\subseteq X\times C_k$ for every $\beta\in\AAa$.
	\end{enumerate}
\end{lemma}
\begin{proof} We prove only \eqref{IT:L:F(Cmxalpha):1}; the other statements are proved similarly.
	So let $F(C_m\times \alpha)\subseteq C_k\times X$.
	The fact that $\alpha$ is dense
	gives $F(C_m\times X)\subseteq \overline{C}_k\times X$.
	Consequently, by \eqref{EQ:Cn-closure},
	$F(C_m\times X)\subseteq ({C}_k\times X) \sqcup ({C}_{k+1}\times X)$.
	Now let $\beta\in\AAa$. Then the first projection of
	$F(C_m\times \beta)$ is a path connected subset of $C_k\sqcup C_{k+1}$ and so it is contained either in $C_k$ or in $C_{k+1}$. Consequently,
	$F(C_m\times \beta)$ is a subset of either $C_k\times X$
	or $C_{k+1}\times X$.
	
	We claim that $F(C_m\times \beta)$ is in fact a subset of $C_k\times X$.
	Suppose, on the contrary, that $F(C_m\times \beta)\subseteq C_{k+1}\times X$.
	By switching the role of $\alpha$ and $\beta$, we infer from the preceding paragraph that
	$F(C_m\times \alpha)$ is a subset of either $C_{k+1}\times X$
	or $C_{k+2}\times X$. This contradicts our assumption on $\alpha$ in
	\eqref{IT:L:F(Cmxalpha):1}.
\end{proof}

For $i,j\in\{1,2\}$ and $m\in D_{ij}$, denote the corresponding $k$ from the definition of $D_{ij}$ by $\psi_{ij}(m)$. By Lemma~\ref{L:F(Cmxalpha)}, $\psi_{ij}$ is a well defined function $D_{ij}\to\mathbb Z$; we denote its range by $R_{ij}$.

\begin{lemma}\label{L:psi(m+1)}
	Let $i,j\in\{1,2\}$.
	\begin{enumerate}
		\item\label{IT:L:psi(m+1):2} If $m\in D_{ij}$ then
		also $m+1\in D_{ij}$ and
		$\psi_{ij}(m+1) - \psi_{ij}(m) \in \{0,1\}$;
		hence the function $\psi_{ij}$ is nondecreasing.
		\item\label{IT:L:psi(m+1):1} Either $D_{ij}=\ZZZ$ or $D_{ij}=\emptyset$
		or $D_{ij}=[m_0,\infty)\cap\ZZZ$ for some $m_0$.
	\end{enumerate}
\end{lemma}
\begin{proof}
	Let $i=j=1$; the other cases are similar.
	Let $m\in D_{11}$ and put
	$k=\psi_{11}(m)$; that is, $F(C_m\times\alpha)\subseteq C_k\times X$ for some
	$\alpha\in\AAa$.
	Since $W_{m+1}\subseteq \overline{C}_{m}$,
	we get $F(W_{m+1}\times \alpha)\subseteq \overline{C}_k\times X\subseteq (C_k\times X)\sqcup(C_{k+1}\times X)$.
	In view of the inclusion $W_{m+1}\subseteq C_{m+1}$, this means that there is a path component $P$ of $X$ such that $F(C_{m+1}\times\alpha)$ intersects one of the sets $C_k\times P$, $C_{k+1}\times P$.
	Consequently, due to path connectedness, $F(C_{m+1}\times\alpha)$ is a subset
	of either $C_k\times P\subseteq C_k\times X$ or $C_{k+1}\times P\subseteq C_{k+1}\times X$.
	Thus $m+1\in D_{11}$ and $\psi_{11}(m+1)\in\{k,k+1\}$. This verifies statement~\eqref{IT:L:psi(m+1):2}. Statement~\eqref{IT:L:psi(m+1):1} follows from \eqref{IT:L:psi(m+1):2} obviously.
\end{proof}

\begin{lemma}\label{L:psi-surjective}
	Let $j\in\{1,2\}$. Then		
	$R_{1j}\cup R_{2j}=\ZZZ$ and at least one of
	$D_{1j}$, $D_{2j}$ is equal to $\ZZZ$.
\end{lemma}
\begin{proof}
	We assume that $j=1$; the other case is analogous.
	To prove that $R_{11}\cup R_{21}=\ZZZ$, fix $k\in\ZZZ$.
	We need to show that there exist $m\in\ZZZ$ and $\alpha\in\AAa$ such that
	$F(C_m\times \alpha)\subseteq C_k\times X$ or
	$F(\alpha\times C_m)\subseteq C_k\times X$.
	We use a cardinality argument.
	The set $C_k\times X$ contains uncountably many path-components $C_k\times\beta$
	($\beta\in\AAa$). By surjectivity of $F$,
	for every $\beta\in\mathcal A$ there is a path component of $X\times X$
	which is mapped by $F$ into $C_k\times\beta$.
	Thus there are uncountably many path components of $X\times X$ which are mapped by $F$ into $C_k\times X$. In view of Lemma~\ref{L:F(alphaxbeta)}, none of them is of the form
	$\alpha\times{\beta}$ ($\alpha,{\beta}\in\AAa$).
	Further, there are only countably many path components of the form
	$C_m\times C_n$ ($m,n\in\ZZZ$).
	Hence, necessarily, a path component of the form $C_m\times \alpha$ or
	$\alpha\times C_m$ (with $\alpha\in\mathcal A$) is mapped to $C_k\times X$.
	
	If both $D_{11}$ and $D_{21}$ are different from $\ZZZ$ then, by
	Lemma~\ref{L:psi(m+1)}\eqref{IT:L:psi(m+1):1},
	their union is bounded from below. Hence, by
	Lemma~\ref{L:psi(m+1)}\eqref{IT:L:psi(m+1):2}, the union of
	$R_{11}$ and $R_{21}$ is also bounded from below, which is in contradiction with the first statement of the lemma.
\end{proof}

\begin{lemma}\label{L:psi-2possib}
	Let $j\in\{1,2\}$. Then		
	one of $D_{1j}$, $D_{2j}$ is $\ZZZ$ and the other one is empty.
\end{lemma}
\begin{proof}
	Again, we assume that $j=1$. 
	By Lemma~\ref{L:psi-surjective}, at least one of 
	$D_{11}$ and $D_{21}$ equals $\ZZZ$.
	Suppose that both are nonempty. To get a contradiction, fix $m\in D_{11}$ and $n\in D_{21}$. Put $k=\psi_{11}(m)$ and $l=\psi_{21}(n)$.
	Then $F(C_m\times X)\subseteq \overline{C}_k\times X$
	and $F(X\times C_n)\subseteq \overline{C}_l\times X$.
	Consequently, $F(C_m\times C_n) \subseteq (\overline{C}_k\cap\overline{C}_l)\times X$
	and so $\abs{k-l}\le 1$. Thus we have proved that
	\begin{equation*}
	\abs{\psi_{11}(m)-\psi_{21}(n)} \le 1
	\qquad\text{for all }
	m\in D_{11} \text{ and } n\in D_{21}.
	\end{equation*}
	It follows that both $\psi_{11}$ and $\psi_{21}$ are bounded, which contradicts the equality $R_{11}\cup R_{21}=\mathbb Z$ from Lemma~\ref{L:psi-surjective}.
\end{proof}

\begin{lemma}\label{L:psi-row-possib}
	Let $i\in\{1,2\}$. Then		
	one of $D_{i1}$, $D_{i2}$ is $\ZZZ$ and the other one is empty.
\end{lemma}
\begin{proof}
       By virtue of Lemma~\ref{L:psi-2possib}, it suffices to exclude the cases $D_{11}=D_{12}=\mathbb Z$ and $D_{21}=D_{22}=\mathbb Z$. We handle the first case; the other one is handled similarly. So assume that $D_{11}=D_{12}=\ZZZ$. Then, by Lemma~\ref{L:psi-2possib}, $D_{21}=D_{22}=\emptyset$.
	
	In view of Lemma~\ref{L:F(Cmxalpha)} and by definitions of $D_{11}$ and $D_{12}$,	
	for all $m\in\ZZZ$ and $\alpha\in \AAa$
	we have $F(C_m\times\alpha) \subseteq  C_{\psi_{11}(m)}\times C_{\psi_{12}(m)}$,
	and so
	$F(C_m\times X) \subseteq  \overline{C}_{\psi_{11}(m)}\times
	\overline{C}_{\psi_{12}(m)} \subseteq \gamma\times\gamma$. Since $m$ was arbitrary, it follows that
	\begin{equation}\label{EQ:P:Dpsi=Z:1}
	F(\gamma\times X) 
	=F\left(\bigcup_{m\in\mathbb Z}C_m\times X\right)=\bigcup_{m\in\mathbb Z}F(C_m\times X)
	\subseteq \gamma\times\gamma.
	\end{equation}
	
	Now fix $k\in\ZZZ$ and $\beta\in\AAa$.
	By surjectivity of $F$, there is a path component $P$ of $X\times X$, which is mapped into $C_k\times \beta$.
	In view of Lemma~\ref{L:F(alphaxbeta)} and \eqref{EQ:P:Dpsi=Z:1}, $P$ is of the form $P=\alpha\times C_m$ for some $m\in\ZZZ$ and $\alpha\in \AAa$. Thus $F(\alpha\times C_m)\subseteq C_k\times\beta$ and so $m\in D_{21}=\emptyset$, a contradiction.
\end{proof}

\begin{proposition}\label{P:Dpsi=Z}
	Exactly one of the following two possibilities is true:
	\begin{enumerate}[leftmargin=4\parindent]
		\item[\caseI]
		$D_{11}=D_{22}=\ZZZ$ and $D_{12}=D_{21}=\emptyset$,
		hence
		$R_{11}=R_{22}=\ZZZ$ and $R_{12}=R_{21}=\emptyset$;
		\item[\caseII]
		$D_{11}=D_{22}=\emptyset$ and $D_{12}=D_{21}=\ZZZ$,
		hence
		$R_{11}=R_{22}=\emptyset$ and $R_{12}=R_{21}=\ZZZ$.
	\end{enumerate}
\end{proposition}
\begin{proof}
	The claims on the domains follow from Lemmas~\ref{L:psi-2possib} and \ref{L:psi-row-possib}.
	For the claims on the ranges, use Lemma~\ref{L:psi-surjective}.
\end{proof}

\begin{lemma}\label{L:F(gammaxgamma)}
	$F(\gamma\times\gamma)=F^{-1}(\gamma\times\gamma)=\gamma\times\gamma$.
\end{lemma}
\begin{proof}
	Fix $\alpha\in\AAa$.
	If $F(\gamma\times\alpha)$ intersects $\gamma\times\gamma$ then, by a path connectedness argument, there are integers
	$m,k,l$ with $F(C_m\times\alpha)\subseteq C_k\times C_l$.
	Hence $m\in D_{11}\cap D_{12}$, which contradicts Proposition~\ref{P:Dpsi=Z}.
	Thus $F(\gamma\times\alpha)$ is disjoint from $\gamma\times\gamma$.
	Analogously, $F(\alpha\times\gamma)$ is disjoint from $\gamma\times\gamma$. Consequently, by taking unions over $\alpha\in\mathcal A$, we infer that
	\begin{equation}\label{Eq:gmxgm1}
	F\big(\gamma\times(X\setminus\gamma)\big)\cap(\gamma\times\gamma)=F\big((X\setminus\gamma)\times\gamma\big)\cap(\gamma\times\gamma)=\emptyset.
	\end{equation}
	Further, by Lemma~\ref{L:F(alphaxbeta)},
	\begin{equation}\label{Eq:gmxgm2}
	F((X\setminus\gamma)\times(X\setminus\gamma))\subseteq
	(X\setminus\gamma)\times(X\setminus\gamma).
	\end{equation}
	Thus, by virtue of \eqref{Eq:gmxgm1} and \eqref{Eq:gmxgm2}, $F^{-1}(\gamma\times\gamma)\subseteq\gamma\times\gamma$.
	
	We prove that $F(\gamma\times\gamma)\subseteq \gamma\times\gamma$.
	To this end, fix $m,n\in\ZZZ$.
	Assume that we are in Case~\caseI{} from Proposition~\ref{P:Dpsi=Z}. 
	Since $m\in D_{11}$ and $n\in D_{22}$, we have $F(C_m\times X)\subseteq\overline{C}_{\psi_{11}(m)}\times X$ and $F(X\times C_n)\subseteq X\times\overline{C}_{\psi_{22}(n)}$, hence $F(C_m\times C_n)\subseteq \overline{C}_{\psi_{11}(m)} \times
	\overline{C}_{\psi_{22}(n)}\subseteq \gamma\times\gamma$.
	In Case~\caseII{} we similarly get
	$F(C_m\times C_n)\subseteq \overline{C}_{\psi_{21}(n)} \times
	\overline{C}_{\psi_{12}(m)}\subseteq \gamma\times\gamma$.
	In any case, $F(C_m\times C_n)\subseteq\gamma\times\gamma$ for all $m,n\in\mathbb Z$, which verifies that $F(\gamma\times\gamma)\subseteq \gamma\times\gamma$.
	
	To summarize, for $A=\gamma\times\gamma$ we have proved $F^{-1}(A)\subseteq A$
	and $F(A)\subseteq A$. Consequently, $F^{-1}(A)=A$ and, by surjectivity of $F$, $F(A)=A$.
\end{proof}


\section{Proof of Theorem~A(1): Nonexistence of minimal maps on $X\times X$}\label{S:XxX-minimal}
We still assume that $F\colon X\times X\to X\times X$ is a continuous surjection. Later in this section we will suppose
that $F$ is a minimal map. 
(To get a contradiction, we will show that then $F^2$ is a direct product, 
$F^2=T^{c\times d}$
for some integers $c\ne 0\ne d$. 
Then Proposition~\ref{P:T-ab-nonminimal} will be used.)

We distinguish two cases.


\subsection{Case~\caseI{} from Proposition~\ref{P:Dpsi=Z}}\label{SS:XxX-minimal-11}
Assume that \caseI{} is true.
Then for all $m,n\in\ZZZ$ and $\alpha\in\AAa$,
\begin{equation}\label{EQ:case1a}
F(C_m\times\alpha)\subseteq C_{\psi_{11}(m)}\times X,
\quad
F(\alpha\times C_n)\subseteq X\times C_{\psi_{22}(n)}.
\end{equation}
By Proposition~\ref{P:Dpsi=Z}, $\varphi=\psi_{11}\times\psi_{22}$ is a surjective selfmap of
$\ZZZ\times\ZZZ$. It follows from \eqref{EQ:case1a}, by passing to closures,
that the map $\varphi$ has the property
\begin{equation}\label{EQ:case1b}
F(C_{m,n})\subseteq \overline{C}_{\varphi(m,n)}
\end{equation}
for all $m,n\in\ZZZ$ (we use the notation from \eqref{EQ:Cmn}).

\begin{lemma}\label{L:F(WxW)}
	Let a continuous surjection $F$ satisfy \caseI{} from Proposition~\ref{P:Dpsi=Z}.
	Let $m,n\in\ZZZ$ be such that
	\begin{equation}\label{EQ:phi(m,n)}
	\varphi(m+1,n+1)=\varphi(m,n) + (1,1).
	\end{equation}
	Then $F(W_{m+1,n+1}) \subseteq W_{\varphi(m+1,n+1)}$.
\end{lemma}
\begin{proof}
	By \eqref{EQ:Cn-closure} and \eqref{EQ:Cmn},
	$W_{m+1,n+1}=\overline{C}_{m,n}\cap
	\overline{C}_{m+1,n+1}$.
	Then the assumption and \eqref{EQ:case1b} yield
	$F(W_{m+1,n+1})
	\subseteq \overline{C}_{\varphi(m,n)} \cap \overline{C}_{\varphi(m+1,n+1)}
	= W_{\varphi(m+1,n+1)}$.
\end{proof}

For $m\in\mathbb Z$ set
\begin{equation*}
I_m^+=\mathbb Z\cap[m,\infty)\hspace{3mm}\text{ and }\hspace{3mm}I_m^-=\mathbb Z\cap(-\infty,m].
\end{equation*}
For a nonzero integer $a$, by $I^{\sgn(a)}_{m}$ we will mean $I_m^+$ or $I_m^-$ depending on whether $a>0$ or $a<0$.
Further, for nonzero integers $a,b$ denote
\begin{equation*}
E_{a,b} = \{(m,n)\in\ZZZ^2\colon \varphi(m,n)= (m,n) + (a,b)\}.
\end{equation*}
The reason why it is useful to consider the sets $E_{a,b}$
lies in the following two lemmas.

\begin{lemma}\label{L:psi-translation}
	Let a minimal map $F$ satisfy \caseI{} from Proposition~\ref{P:Dpsi=Z}.
	Then there are integers $a,m_a$ and $b,n_b$ such that $a\ne 0 \ne b$ and
	\begin{equation*}
	\psi_{11}(m)=m+a,
	\qquad
	\psi_{22}(n)=n+b
	\end{equation*}
	for all integers $m\in I^{\sgn(a)}_{m_a}$ and $n\in I^{\sgn(b)}_{n_b}$.
	Hence
	\begin{equation*}
	E_{a,b}\supseteq I^{\sgn(a)}_{m_a}\times I^{\sgn(b)}_{n_b}.
	\end{equation*}
\end{lemma}
\begin{proof}
	We prove only the claim for $\psi_{11}$; the claim for $\psi_{22}$ is proved analogously.
	First realize that  $\psi_{11}$ has no fixed point due to minimality of $F$
	(indeed, if $\psi_{11}(k)=k$ then the closed set $\overline{C}_k\times X$ would be $F$-invariant
	by \eqref{EQ:case1a}), hence $\psi_{11}-\Id_{\mathbb Z}$ does not vanish.
	Further, by
	Lemma~\ref{L:psi(m+1)}\eqref{IT:L:psi(m+1):2}, $\psi_{11}$ is nondecreasing,
	$(\psi_{11}-\Id_\ZZZ)$ is nonincreasing and
	\begin{equation}\label{Eq:jmp.1}
	(\psi_{11}-\Id_{\mathbb Z})(m+1)-(\psi_{11}-\Id_{\mathbb Z})(m) \ \in\  \{-1,0\}
	\end{equation}
	for every $m\in\mathbb Z$.
	
	Assume first that there is $m_1$ such that $\psi_{11}(m_1)>m_1$, i.e.~$(\psi_{11}-\Id_\ZZZ)(m_1)>0$.
	Since $\psi_{11}-\Id_\ZZZ$ is nonincreasing, satisfies \eqref{Eq:jmp.1} and does not vanish, we clearly have that
	there are integers $a>0$ and $m_a$ such that
	$(\psi_{11}-\Id_\ZZZ)(m)=a$, i.e.~$\psi_{11}(m)=m+a$,
	for every $m\ge m_a$.
	
	If there is no such $m_1$ then $\psi_{11}<\Id_\ZZZ$.
	Then $\psi_{11}-\Id_\ZZZ$ is negative and nonincreasing.
	Hence there are integers $a<0$ and $m_a$ such that 
	$(\psi_{11}-\Id_\ZZZ)(m)=a$, i.e.~$\psi_{11}(m)=m+a$,
	for all $m\le m_a$.
\end{proof}

\begin{lemma}\label{L:Eab}
	Let a continuous surjection $F$ satisfy \caseI{} from Proposition~\ref{P:Dpsi=Z}.
	Let $a,b$ be nonzero integers. 
	If $m,n$ are such that both $(m,n)$ and $(m+1,n+1)$ belong to $E_{a,b}$,
	then we have \eqref{EQ:phi(m,n)} and
	$\Coin(F,T^{a\times b}) \cap W_{m+1,n+1}\ne\emptyset$.
\end{lemma}
\begin{proof}
	The first claim follows from the definition of $E_{a,b}$.
	So, by Lemma~\ref{L:F(WxW)},
	$F(W_{m+1,n+1})\subseteq W_{\varphi(m+1,n+1)} = W_{(m+1,n+1)+(a,b)}$.
	Now use Lemma~\ref{L:coin-F-Tab}.
\end{proof}

\begin{lemma}\label{L:F-is-Tab}
	Let a minimal map $F$ satisfy \caseI{} from Proposition~\ref{P:Dpsi=Z}.
	Then $F=T^{a\times b}$ for some integers
	$a\ne 0 \ne b$.
\end{lemma}
\begin{proof}
	Take $a\ne 0 \ne b$ from Lemma~\ref{L:psi-translation}.
	Let $J$ be the set of all $(m,n)\in\ZZZ\times\ZZZ$ such that
	both $(m,n)$ and $(m+1,n+1)$ belong to $E_{a,b}$.
	By Lemma~\ref{L:psi-translation}, the set $J$ contains a quadrant
	in $\ZZZ\times\ZZZ$.
	Further, the sets $W_{m+1,n+1}$ ($(m,n)\in J$) are nowhere dense
	in $X\times X$ (because they are closed and are subsets of the set $\gamma\times\gamma$ with empty interior).	
	Since $X\times X$ has a countable basis, using \eqref{EQ:Wn-union-dense}
	we can inductively construct a sequence $((m_i,n_i))_{i\geq 1}$ from $J$
	such that the union of $W_{m_i+1,n_i+1}$ is dense and both $\abs{m_i}$ and $\abs{n_i}$ are increasing in $i$.
	By \eqref{EQ:Wn-diam}, the diameters of  $W_{m_i+1,n_i+1}$ converge to zero.
	Now Lemmas~\ref{L:Eab} and \ref{L:coincidence-of-maps} yield that $F=T^{a\times b}$.
\end{proof}

\subsection{Case~\caseII{} from Proposition~\ref{P:Dpsi=Z}}
Assume that \caseII{} is true.
Then, for all $m,n\in\ZZZ$ and $\alpha\in\AAa$,
\begin{equation*}
F(C_m\times\alpha)\subseteq X\times C_{\psi_{12}(m)},
\qquad
F(\alpha\times C_n)\subseteq C_{\psi_{21}(n)}\times X.
\end{equation*}
Due to Lemma~\ref{L:F(gammaxgamma)}, we in fact have
\begin{equation*}
F(C_m\times\alpha)\subseteq (X\setminus\gamma)\times C_{\psi_{12}(m)},
\qquad
F(\alpha\times C_n)\subseteq C_{\psi_{21}(n)}\times (X\setminus\gamma).
\end{equation*}
These two inclusions imply that for all $m,n\in\ZZZ$ and $\alpha\in\AAa$,
\begin{equation}\label{EQ:case2-for-F2}
F^2(C_m\times\alpha)\subseteq C_{\psi_{21}(\psi_{12}(m))}\times X
\quad\text{and}\quad
F^2(\alpha\times C_n)\subseteq X\times C_{\psi_{12}(\psi_{21}(n))}.
\end{equation}
So $F^2$ satisfies \caseI{} from Proposition~\ref{P:Dpsi=Z}.
Moreover, $F^2$ is also minimal since $F$ is minimal and $X\times X$ is a continuum.
Thus, applying Lemma~\ref{L:F-is-Tab} to $F^2$, we immediately obtain the following lemma.

\begin{lemma}\label{L:case2}
	Let a minimal map $F$ satisfy \caseII{} from Proposition~\ref{P:Dpsi=Z}.
	Then $F^2=T^{a'\times b'}$ for some integers $a'\ne 0 \ne b'$.
\end{lemma}


\subsection{Proof of Theorem~A(1)}

We are finally ready to prove that the product of minimal spaces need not be minimal. 

\begin{theorem}\label{T:thmA1}
Every DST space $X$ admits a minimal homeomorphism, but $X\times X$ does not admit any minimal continuous map.
\end{theorem}

\begin{proof}
	Let $X$ and $T$ be a DST space and a minimal homeomorphism $X\to X$
	from Subsection~\ref{SS:Slovak}.
	Assume that $X\times X$ admits a minimal map $F$.
	By Lemmas~\ref{L:F-is-Tab} and \ref{L:case2}, the minimal map $F^2$ is of the form
	$T^{c\times d}$ for some integers $c,d$.
	This contradicts Proposition~\ref{P:T-ab-nonminimal}.
\end{proof}

Contrary to Remark~\ref{R:homeo-case-XN}, the following question is open.

\begin{problem}
   Let $X$ be a DST space.
   Is it true that for $N\ge 3$ the space $X^N$ does not admit any minimal continuous map?
\end{problem}


\section{Other minimal spaces with nonminimal squares. Proof of Theorem~A(2,3)}\label{Sec:other.min}

The following result is a strengthening of Theorem~A(1). 

\begin{theorem}\label{T:XxXxY}
   Let $X$ be a DST space and let $Y$ be a path-connected continuum.  
   Then the space $X\times X\times Y$ is not minimal.   
\end{theorem}

The proof of this result is analogous to that of Theorem~\ref{T:thmA1}.
Basically one only needs to replace all products of two sets in Sections~\ref{SS:XxX-surjection} and \ref{S:XxX-minimal} by the corresponding products of three sets, the third factor being $Y$.
There are only few places in the proof where the argument is slightly different.
The main difference is in the proof of (generalization of) Lemma~\ref{L:Eab}; there we cannot
use Lemma~\ref{L:coin-F-Tab} (and Lemma~\ref{L:point-of-coincidence}) since the space $W_{m,n}\times Y$ need not have the fixed point property. In the rest of this section we outline the main steps of the proof of Theorem~\ref{T:XxXxY}.

Put $Z=X\times X\times Y$. 
Recall that the path components of $X\times X$   
were described in the beginning of Subsection~\ref{SS:XxX}.
Using that notation, the path components of $Z$ are of four types:
$\alpha\times \beta\times Y$, $C_m\times\alpha\times Y$, $\alpha\times C_n\times Y$, and 
$C_m\times C_n\times Y$ ($m,n\in\ZZZ$, $\alpha\in\AAa$). Only those of type $\alpha\times \beta\times Y$ are dense in $Z$.

Let $F\colon Z\to Z$ be a continuous surjection. Then the analogues of all the results from 
Section~\ref{SS:XxX-surjection} hold; the only difference is that 
if we had there a subset $A\times B\subseteq X\times X$ 
then now it is replaced by the subset $A\times B\times Y$ of $Z$.
In particular, this applies to the definitions of the sets $D_{ij}$
from \eqref{EQ:Dij} and we have the following.

\begin{proposition}[analogue of Proposition~\ref{P:Dpsi=Z}]\label{P:Dpsi=Z*}
	For a continuous surjection $F\colon Z\to Z$, 
	exactly one of the following two possibilities is true:
	\begin{enumerate}[leftmargin=4\parindent]
	\item[\caseIy]
		There are surjections $\psi_{11}^*,\psi_{22}^*\colon\ZZZ\to\ZZZ$ such that,
		for all $m,n\in\ZZZ$ and $\alpha\in\AAa$,
		\begin{equation*}
			F(C_m\times\alpha\times Y)  \subseteq  C_{\psi_{11}^*(m)}\times X\times Y,
			\quad
			F(\alpha\times C_n\times Y)\subseteq X\times C_{\psi_{22}^*(n)}\times Y.
		\end{equation*}
	\item[\caseIIy]
		There are surjections $\psi_{12}^*,\psi_{21}^*\colon\ZZZ\to\ZZZ$ such that,
		for all $m,n\in\ZZZ$ and $\alpha\in\AAa$,
		\begin{equation*}
			F(C_m\times\alpha\times Y)  \subseteq  X\times C_{\psi_{12}^*(m)}\times Y,
			\quad
			F(\alpha\times C_n\times Y)\subseteq C_{\psi_{21}^*(n)}\times X\times Y.
		\end{equation*}
	\end{enumerate}
\end{proposition}

\begin{lemma}[analogue of Lemma~\ref{L:F(gammaxgamma)}]\label{L:F(gammaxgamma)*}
	$F(\gamma\times\gamma\times Y)=F^{-1}(\gamma\times\gamma\times Y)=\gamma\times\gamma\times Y$.
\end{lemma}

\subsection{Case~\caseIy{} from Proposition~\ref{P:Dpsi=Z*}}
Assume that \caseIy{} is true and consider the
surjection $\varphi^*=\psi_{11}^*\times\psi_{22}^*$ of $\ZZZ\times\ZZZ$. 
Then, for all $m,n\in\ZZZ$,
\begin{equation*}
F(C_{m,n}\times Y)\subseteq \overline{C}_{\varphi^*(m,n)}\times Y.
\end{equation*}
Proofs of the following two lemmas are analogous to those given in 
Subsection~\ref{SS:XxX-minimal-11}.
Recall that $I_m^+=\mathbb Z\cap[m,\infty)$ and $I_m^-=\mathbb Z\cap(-\infty,m]$ for 
$m\in\ZZZ$; for integers $a\ne 0\ne b$ put
$$
  E_{a,b}^* = \{(m,n)\in\ZZZ^2\colon \varphi^*(m,n)= (m,n) + (a,b)\}.  
$$
\begin{lemma}[analogue of Lemma~\ref{L:F(WxW)}]\label{L:F(WxW)*}
	Let a continuous surjection $F$ satisfy \caseIy{} from Proposition~\ref{P:Dpsi=Z*}.
	Let $m,n\in\ZZZ$ be such that
	\begin{equation*}
	\varphi^*(m+1,n+1)=\varphi^*(m,n) + (1,1).
	\end{equation*}
	Then $F(W_{m+1,n+1}\times Y) \subseteq W_{\varphi^*(m+1,n+1)}\times Y$.
\end{lemma}
\begin{lemma}[analogue of Lemma~\ref{L:psi-translation}]\label{L:psi-translation*}
	Let a minimal map $F$ satisfy \caseIy{} from Proposition~\ref{P:Dpsi=Z*}.
	Then there are integers $a,m_a$ and $b,n_b$ such that $a\ne 0 \ne b$ and
	\begin{equation*}
	\psi_{11}^*(m)=m+a,
	\qquad
	\psi_{22}^*(n)=n+b
	\end{equation*}
	for all integers $m\in I^{\sgn(a)}_{m_a}$ and $n\in I^{\sgn(b)}_{n_b}$.
	Hence
	\begin{equation*}
	E_{a,b}^*\supseteq I^{\sgn(a)}_{m_a}\times I^{\sgn(b)}_{n_b}.
	\end{equation*}
\end{lemma}

Contrary to Subsection~\ref{SS:XxX-minimal-11}, now we cannot prove
that a minimal map $F$
satisfying \caseIy{} is a direct product. Instead, we prove that $F$ is 
a skew product over some $T^{a\times b}$.

\begin{lemma}[analogue of Lemma~\ref{L:F-is-Tab}]\label{L:F-is-Tab*}
	Let a minimal map $F$ satisfy \caseIy{} from Proposition~\ref{P:Dpsi=Z*}.
	Then there are integers $a\ne 0 \ne b$ such that $F$ is a skew product over
	$T^{a\times b}$; that is, for every $(x,x',y)\in Z$ we have
	$$
	  \pi_{12} \circ F(x,x',y) = T^{a\times b}(x,x'),
	$$
	where $\pi_{12}\colon Z\to X\times X$ is the projection onto the first two coordinates.
\end{lemma}
\begin{proof}
	Take $a\ne 0 \ne b$ from Lemma~\ref{L:psi-translation*}.
	As in the proof of Lemma~\ref{L:F-is-Tab} construct sequences $(m_i)_{i\ge 1}$ and 
	$(n_i)_{i\ge 1}$ in 
	$\ZZZ$ such that
	\begin{enumerate}
 	\item\label{IT:1} for every $i$, both $(m_i,n_i)$ and $(m_{i}+1,n_{i}+1)$ belong to $E_{a,b}^*$;
 	\item the union of $W_{m_i+1,n_i+1}$ is dense in $X\times X$;
 	\item the diameters of $W_{m_i+1,n_i+1}$ tend to zero.
	\end{enumerate}
	
	Fix $y\in Y$ and consider the map $F_y\colon X\times X\to X\times X$
	defined by $F_y(x,x')=\pi_{12}\circ F(x,x',y)$.	
	Then $F_y$ is a well-defined continuous map.
	Notice that Lemma~\ref{L:F(WxW)*} and the property (1) above imply that, for every $i$,
	\begin{equation*}
	  F_y(W_{m_i+1,n_i+1})\subseteq W_{\varphi^*(m_i+1,n_i+1)} = W_{m_i+1+a,n_i+1+b}
	  = T^{a\times b} (W_{m_i+1,n_i+1}).
	\end{equation*}
	Thus, by Lemma~\ref{L:coin-F-Tab},
	$\Coin(F_y,T^{a\times b})$ intersects every $W_{m_i+1,n_i+1}$.
	Now Lemma~\ref{L:coincidence-of-maps} applied to the maps $F_y,T^{a\times b}\colon X\times X \to X\times X$ and
	the nowhere dense sets $W_{m_i+1,n_i+1}$ 
	(which satisfy the conditions (2) and (3) above)
	shows that $F_y=T^{a\times b}$.
    Since this is true for
	every $y\in Y$ and $a,b$ do not depend on $y$, $F$ is a skew product over $T^{a\times b}$.
\end{proof}

\subsection{Case~\caseIIy{} from Proposition~\ref{P:Dpsi=Z*}}

\begin{lemma}[analogue of Lemma~\ref{L:case2}]\label{L:case2*}
	Let a minimal map $F\colon Z\to Z$ satisfy \caseIIy{} from Proposition~\ref{P:Dpsi=Z*}.
	Then there are integers $a'\ne 0 \ne b'$ such that $F^2$ is a skew product over 
	$T^{a'\times b'}$.
\end{lemma}
\begin{proof}
By \caseIIy{} and Lemma~\ref{L:F(gammaxgamma)*}, for every $m,n\in\ZZZ$ and $\alpha\in\AAa$ 
it holds that
\begin{equation*}
  F(C_m\times\alpha\times Y)
  \subseteq 
  (X\setminus\gamma)\times C_{\psi_{12}^*(m)}\times Y,
\qquad
  F(\alpha\times C_n\times Y)
  \subseteq 
  C_{\psi_{21}^*(n)}\times (X\setminus\gamma)\times Y.
\end{equation*}
These two inclusions imply that, for all $m,n\in\ZZZ$ and $\alpha\in\AAa$,
\begin{equation*}
  F^2(C_m\times\alpha\times Y)
  \subseteq 
  C_{\psi_{21}^*(\psi_{12}^*(m))}\times X\times Y
\quad\text{and}\quad
  F^2(\alpha\times C_n\times Y)
  \subseteq 
  X\times C_{\psi_{12}^*(\psi_{21}^*(n))}\times Y.
\end{equation*}
So $F^2$ satisfies \caseIy{} from Proposition~\ref{P:Dpsi=Z*}.
Moreover, $F^2$ is also minimal since $F$ is minimal and $Z$ is a continuum.
Thus, applying Lemma~\ref{L:F-is-Tab*} to $F^2$, we immediately obtain the assertion of the 
lemma.
\end{proof}

\subsection{Proof of Theorem~\ref{T:XxXxY}}
\begin{proof}
Suppose that $F\colon Z\to Z$ is minimal.
Since $F$ satisfies either \caseIy{} or \caseIIy{} from Proposition~\ref{P:Dpsi=Z*},
Lemmas~\ref{L:F-is-Tab*} and \ref{L:case2*} yield that $F^2$ is a skew product 
over $T^{c\times d}$ for some integers $c,d$. Since $F^2$ is minimal we get that also $T^{c\times d}$ is minimal, which contradicts Proposition~\ref{P:T-ab-nonminimal}.
\end{proof}

\subsection{Proof of Theorem~A(2,3)}

Theorem~\ref{T:XxXxY} enables us to find minimal spaces,
other than DST spaces, with nonminimal squares. In fact, we have the following theorem, giving (2) of Theorem~A.

\begin{theorem}\label{T:XxY-nonminimal-square}
   Let $X$ be a DST space and $Y$ be a product-minimal path-connected continuum.
   Then $X\times Y$ is a minimal space with nonminimal square.
\end{theorem}
\begin{proof}
   Since $X$ is minimal and $Y$ is product-minimal, the space $X\times Y$ is minimal. Further, $Y^2$ is a path-connected continuum, hence the space $X\times X\times Y^2$ is not minimal by Theorem~\ref{T:XxXxY}. Finally, since the square $(X\times Y)^2$ is homeomorphic to $X\times X\times Y^2$, we are done.
\end{proof}

Finally we prove Theorem~A(3).

\begin{theorem}\label{T:XxT-nonminimal-square}
   Let $X$ be a DST space and $n\ge 2$ be an integer.
   Then the space $X\times\mathbb T^n$ admits a minimal homeomorphism as well as a minimal noninvertible map,
   but its square is not minimal.
\end{theorem}
\begin{proof}
  Since the $n$-torus $\mathbb T^n$ is homeo-product-minimal by Theorem~\ref{T:groups}, the space $X\times\mathbb T^n$ admits a minimal homeomorphism. 
  By Theorem~\ref{T:XxY-nonminimal-square}, the square of $X\times\mathbb T^n$ is not minimal. So, to finish the proof, it remains to show that $X\times\mathbb T^n$ admits a minimal noninvertible map.
  
Fix a minimal homeomorphism $T$ on $X$ and a minimal irrational flow $\phi=(\varphi_t)_{t\in\mathbb R}$ on $\mathbb T^n$. Since the centralizer $Z(\phi)$ of $\phi$ in $\mathcal H(\mathbb T^n)$
contains all rotations on $\mathbb T^n$ and so it acts transitively on $\mathbb T^n$ in the algebraic sense, Theorem~\ref{T:flow.centr} yields a residual set $G_1\subseteq\mathbb R$ such that for every $t\in G_1$, the product $(X\times \mathbb T^n, T\times\varphi_t)$ is a minimal homeomorphism. By \cite[Theorem~2.1]{BCO},
  there is a residual set $G_2\subseteq\mathbb R$ such that for every $t\in G_2$, the homeomorphism $\varphi_t$ admits a noninvertible minimal map $S_t\colon \mathbb T^n\to\mathbb T^n$ as an almost 1-1 extension. Fix $t\in G_1\cap G_2$. Then $T\times\varphi_t$ is a minimal homeomorphism and, clearly, $T\times S_t$ is an almost 1-1 extension of $T\times\varphi_t$. Hence $T\times S_t$ is a minimal noninvertible map on $X\times\mathbb T^n$ by Lemma~\ref{L:minim.ext}.
\end{proof}


\end{large}

\begin{thebibliography}{9999999}

\bibitem[AK]{AnzKak} H.~Anzai, S.~Kakutani,
    \textit{Bohr compactifications of a locally compact
	Abelian group II}, Proc.~Imp.~Acad.~Tokyo \textbf{19}~(1943),
	533--539.
	
\bibitem[BDHSS]{BDHSS} F.~Balibrea, T.~Downarowicz, R.~Hric, \mL.~Snoha, V.~Špitalský,
	\textit{Almost totally disconnected minimal systems}, 
	Ergodic Theory Dynam.~Systems \textbf{29}~(2009), no.~3, 737--766.
	
\bibitem[BL]{BL} D.~P.~Bellamy, J.~M.~Lysko, 
   \textit{Factorwise rigidity of the product of two pseudo-arcs},
   Topology Proc.~\textbf{8}~(1983), no.~1, 21--27.
	
	
		
\bibitem[BOT]{BOT} A.~Blokh, L.~Oversteegen, E.~D.~Tymchatyn, 
    \textit{On minimal maps of $2-$manifolds},
    Ergodic Theory Dynam.~Systems \textbf{25}~(2005), no.~1, 41--57.
    
\bibitem[BCO]{BCO} J.~P.~Boro\'nski, A.~Clark, P.~Oprocha, 
   \textit{A compact minimal space $Y$ such that its square $Y\times Y$ is not minimal},
   Adv.~Math.~\textbf{335}~(2018), 261--275.

\bibitem[Bor]{Bor} K.~Borsuk, 
    \textit{On embedding curves in surfaces},
     Fund.~Math.~\textbf{59}~1966, 73--89.

\bibitem[Bre]{Bre} G.~E.~Bredon,
   \textit{Introduction to compact transformation groups. Pure and Applied Mathematics, Vol.~46.}
   Academic Press, New York-London, 1972.
    
\bibitem[BKS]{BKS} H.~Bruin, S.~Kolyada, \mL.~Snoha, 
    \textit{Minimal nonhomogeneous continua},
    Colloq.~Math.~\textbf{95}~(2003), no.~1, 123--132.
    
\bibitem[Cer]{Cer} A.~V.~\v Cernavski\u\i,
   \textit{Local contractibility of the group of homeomorphisms of a manifold. (Russian)},
   Mat.~Sb.~(N.S.) \textbf{79 (121)} 1969, 307--356.


\bibitem[Dir]{Di} M.~Dirb\'ak,
    \textit{Minimal extensions of flows with amenable acting groups},
    Israel J.~Math.~\textbf{207}~(2015), no.~2, 581--615.
    
\bibitem[DM]{DM} M.~Dirb\'ak, P.~Mali\v ck\'y,
	\textit{On the construction of non-invertible minimal skew products},
	J.~Math.~Anal.~Appl.~\textbf{375} (2011), no.~2, 436--442.
	 
\bibitem[DST]{DST} T.~Downarowicz, \mL{}.~Snoha,  D.~Tywoniuk,
   \textit{Minimal Spaces with Cyclic Group of Homeomorphisms},
   J.~Dynam.~Differential Equations \textbf{29}~(2017), no.~1, 243--257.

\bibitem[Dye]{Dy} E.~Dyer, 
    \textit{A fixed point theorem},
    Proc.~Amer.~Math.~Soc.~\textbf{7}~(1956), 662--672.
		
\bibitem[Eck]{Eck} B.~Eckmann, 
   \textit{\"Uber monothetische Gruppen. (German)}, Comment.~Math.~Helv.~\textbf{16}, (1944). 249--263.
   
\bibitem[EK]{EdwKir} R.~D.~Edwards, R.~C.~Kirby,
   \textit{Deformations of spaces of imbeddings}, Ann.~Math.~(2) \textbf{93}, (1971), 63--88.
		
\bibitem[Ega]{Eg}
   J.~Egawa,
   \emph{Eigenvalues of compact minimal flows},
   Math.~Sem.~Notes Kobe Univ.~\textbf{10}~(1982), no.~2, 281--291.

	
\bibitem[FH]{FH} A.~Fathi, M.~R.~Herman, 
   \textit{Existence de diff\'eomorphismes minimaux (French)},
   Dynamical systems, Vol.~I—Warsaw, pp.~37--59.
   Astérisque, No.~49, Soc.~Math.~France, Paris, 1977.

\bibitem[Fay]{Fa}
   B.~R.~Fayad,
   \emph{Topologically mixing and minimal but not ergodic, analytic transformation on $\mathbb T^5$},
   Bol.~Soc.~Brasil.~Mat. (N.S.) \textbf{31}~(2000), no.~3, 277--285.
   
\bibitem[Fur]{F}  H.~Furstenberg, 
  \textit{Disjointness in ergodic theory, minimal sets, and a problem in Diophantine approximation},
  Math.~Systems Theory \textbf{1}~1967, 1--49.

\bibitem[Gla]{Gla} E.~Glasner, 
   \textit{Ergodic theory via joinings}, Mathematical Surveys and Monographs, 101. American Mathematical Society, Providence, RI, 2003.

\bibitem[GW]{GW} S.~Glasner, B.~Weiss, 
   \textit{On the construction of minimal skew products},
   Israel J.~Math.~\textbf{34}~(1979), no.~4, 321--336.
   
		
\bibitem[Got]{Gott} W.~H.~Gottschalk,
   \textit{Orbit-closure decompositions and almost periodic properties},
   Bull.~Amer.~Math.~Soc.~\textbf{50}~(1944), 915--919.

	
\bibitem[GP]{GP} 
V.~Guillemin, A.~Pollack, 
\textit{Differential topology}, Prentice-Hall, 1974.


\bibitem[Hag]{Ha}
   C.~Hagopian,
   \emph{A characterization of solenoids},
   Pacific J.~Math.~\textbf{68}~(1977), no.~2, 425--435.

\bibitem[Han]{Han} M.~Handel, 
   \textit{A pathological area preserving $C\sp{\infty }$ diffeomorphism of the plane},
   Proc.~Amer.~Math.~Soc.~\textbf{86}~(1982), no.~1, 163--168.

\bibitem[HR]{HarRyl} S.~Hartman, C.~Ryll-Nardzewski,
   \textit{Zur Theorie der lokal-kompakten Abelschen Gruppen.~(German)}, Colloq.~Math.~\textbf{4} (1957), 157--188.

\bibitem[HM]{HofMor} K.~H.~Hofmann, S.~A.~Morris,
   \textit{The structure of compact groups. A primer for the student—a handbook for the expert. Second revised and augmented edition}, De Gruyter Studies in Mathematics, 25. Walter de Gruyter \& Co., Berlin, 2006.
   
   
\bibitem [KS]{KS} S.~Kolyada, \mL{}.~Snoha,
   \emph{Minimal dynamical systems},
   Scholarpedia \textbf{4(11):5803}~(2009),
   \url{http://www.scholarpedia.org/article/Minimal_dynamical_systems}.
         
\bibitem[KST]{KST} S.~Kolyada, \mL{}.~Snoha, S.~Trofimchuk,
   \textit{Minimal sets of fibre-preserving maps in graph bundles},
   Math.~Z.~\textbf{278}~(2014), no.~1-2, 575--614.
         
\bibitem [KN]{KuiNie} L.~Kuipers, H.~Niederreiter,
   \textit{Uniform distribution of sequences},
   Wiley-Interscience, New York-London-Sydney (1974).
   
         
\bibitem [KKT]{KKT}  K.~Kuperberg, W.~Kuperberg, W.~R.~R.~Transue,
   \textit{On the $2$-homogeneity of Cartesian products},
   Fund.~Math.~\textbf{110}~(1980), no.~2, 131--134.
   

\bibitem[Mun]{Mun} J.~R.~Munkres,
   \textit{Topology, Second edition}, 
   Prentice Hall, Inc., Upper Saddle River, NJ (2000).
   
\bibitem[Nad]{Nad} S.~B.~Nadler, Jr.,
   \textit{Continuum theory. An introduction}, 
   Marcel Dekker, Inc., New York (1992).
   
   
\bibitem[Nor]{Nor} A.~Norton, 
   \textit{Minimal sets, wandering domains, and rigidity in the $2$-torus},
    Continuum theory and dynamical systems (Arcata, CA, 1989), 129--138, Contemp.~Math., 117, 
    Amer.~Math.~Soc., Providence, RI, 1991.

\bibitem[Wal]{Wal} P.~Walters, 
   \textit{An introduction to ergodic theory.  Graduate Texts in Mathematics, 79.} Springer-Verlag, New York-Berlin, 1982.
   
\bibitem[Why]{Why} G.~T.~Whyburn,
   \textit{Topological characterization of the Sierpi\'nski curve},
   Fund.~Math.  \textbf{45}~(1958), 320--324.

\end{thebibliography}
\end{document}